\documentclass[11pt]{article}

\usepackage{graphics}
\usepackage{graphicx}
\usepackage{amsfonts}
\usepackage{amscd}
\usepackage{amsthm}
\usepackage{indentfirst}
\usepackage{amssymb, amsmath, amsthm, amsgen, amstext, amsbsy, amsopn}
\usepackage{epic,eepic}
\usepackage{hyperref}
\usepackage{enumitem} 
\usepackage{color}
\usepackage{ dsfont }
\usepackage{xcolor}
\usepackage[margin=1.3in]{geometry}
\usepackage{subcaption}

\newtheorem{thm}{Theorem}
\newtheorem*{thm*}{Theorem}
\newtheorem{lemma}{Lemma}[section]
\newtheorem{prop}[lemma]{Proposition}
\newtheorem{claim}[lemma]{Claim}
\newtheorem{cor}[lemma]{Corollary}

\theoremstyle{definition}
\newtheorem{defin}[lemma]{Definition}
\newtheorem{rem}[lemma]{Remark}
\newtheorem{exam}[lemma]{Example}
\newtheorem{notation}[lemma]{Notation}

\newcommand{\R}{{\mathbb{R}}}

\newcommand{\T}{{\mathbb{T}}}
\newcommand{\Z}{{\mathbb{Z}}}
\newcommand{\N}{{\mathbb{N}}}
\newcommand{\C}{{\mathbb{C}}}

\newcommand{\cA}{{\mathcal{A}}}

\newcommand{\cD}{{\mathcal{D}}}
\newcommand{\cC}{{\mathcal{C}}}
\newcommand{\cE}{{\mathcal{E}}}
\newcommand{\cF}{{\mathcal{F}}}

\newcommand{\cH}{{\mathcal{H}}}

\newcommand{\cJ}{{\mathcal{J}}}

\newcommand{\cL}{{\mathcal{L}}}
\newcommand{\cM}{{\mathcal{M}}}

\newcommand{\cP}{{\mathcal{P}}}

\newcommand{\cS}{{\mathcal{S}}}

\newcommand{\cZ}{{\mathcal{Z}}}

\def\id{{1\hskip-2.5pt{\rm l}}}

\newcommand{\grad}{{\hbox{\rm grad\,}}}

\newcommand{\spec}{{\text{spec}}}

\newcommand{\im}{{\operatorname{Im}\,}}
\newcommand{\ind}{{\operatorname{ind}\,}}

\newcommand{\hreg}{\cH_{\text{reg}}}

\begin{document}
\title{Floer theory of disjointly supported Hamiltonians on symplectically aspherical manifolds}
\author{Yaniv Ganor and Shira Tanny}
\maketitle
\begin{abstract}
	We study the Floer-theoretic interaction between disjointly supported Hamiltonians by comparing Floer-theoretic invariants of these Hamiltonians with the ones of their sum. These invariants include spectral invariants, boundary depth and Abbondandolo-Haug-Schlenk's action selector. Additionally, our method shows that in certain situations the spectral invariants of a Hamiltonian supported in an open subset of a symplectic manifold are independent of the ambient manifold.
\end{abstract}

\tableofcontents

\section{Introduction and results.}
The paper deals with Hamiltonian diffeomorphisms of symplectic manifolds, which model the Hamiltonian dynamics on phase spaces in classical mechanics. A central tool for studying Hamiltonian diffeomorphisms is Floer theory, which is an infinite-dimensional version of Morse theory applied to the action functional on the space of contractible loops. As such, Floer theory associates a chain complex to each Hamiltonian, which is generated by the critical points of the action functional and whose differential counts certain negative gradient flow lines, called {\it Floer trajectories}. 

Our main object of interest is Floer theory for Hamiltonians supported in pairwise disjoint open sets, namely $F= F_1 +...+F_N$ where $F_i$ is supported in $U_i$ and $U_1,\dots,U_N$ are pairwise disjoint.
On the level of dynamics, the Hamiltonian diffeomorphisms $\varphi_i$ corresponding to $F_i$ do not interact. The Hamiltonian diffeomorphism corresponding to $F$ is the composition $\varphi= \varphi_1\circ\cdots\circ\varphi_N$, and the diffeomorphisms $\varphi_i$ commute. However, it is unclear a priori whether in Floer theory there is any communication between the disjointly supported Hamiltonians
$F_i$. 
The Floer-theoretic interaction between disjointly supported Hamiltonians was studied by Polterovich \cite{polterovich2014symplectic}, Seyfaddini \cite{seyfaddini2014spectral}, Ishikawa \cite{ishikawa2015spectral} and Humili\`ere-Le Roux-Seyfaddini \cite{humiliere2016towards}, mostly through the relation between invariants of the sum of Hamiltonians and invariants of each one. These works suggest that such an interaction should be quite limited.  
The main finding of this paper is a construction, on symplectically aspherical manifolds and under some conditions on the domains $U_i$, of what we call a ``barricade" - a specific perturbation of the Hamiltonians $F_i$ near the boundaries of $U_i$, which prevents Floer trajectories from entering or exiting these domains. The presence of barricades limits the communication between disjointly supported Hamiltonians as expected. 
The construction is motivated by the following simple idea in Morse theory. Given a smooth function $F$ on a Riemannian manifold, that is supported inside an open subset $U$, one can perturb it into a Morse function $f$ that has a ``bump" in a neighborhood of the boundary, as illustrated in Figure~\ref{fig:Morse_bump}. The negative gradient flow-lines of $f$ cannot cross the bump, and therefore a flow-line starting inside $U$, and away from the boundary, remains there. On the other hand, flow-lines that start on the bump can flow both in and out of $U$. Since the Morse differential counts negative gradient flow-lines, such constraints can be used to gain  information about it. 
\begin{figure}[h]
	\centering
	\includegraphics[scale=0.76]{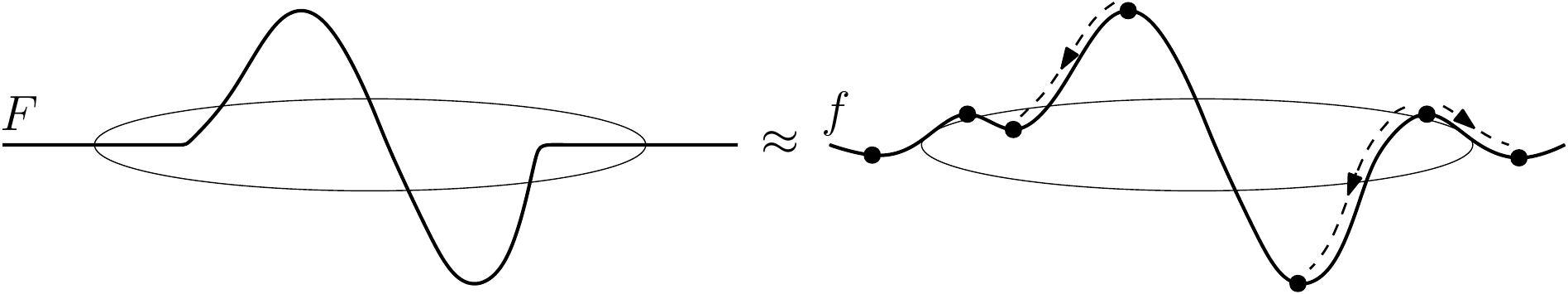}
	\caption{\small{We perturb the function $F$ to create a small ``bump" along a neighborhood of $\partial U$. The dashed lines represent (some of the) flow-lines of $-\grad f$.
	}}
	\label{fig:Morse_bump}
\end{figure}

This idea can be adapted to Floer theory on symplectically aspherical manifolds (that is, when the symplectic form $\omega$ and the first Chern class $c_1$ vanish on $\pi_2(M)$), and under certain assumptions on the domain $U$. The resulting  construction can be used to study Floer-theoretic invariants, such as {\it spectral invariants} and the {\it boundary depth}, of Hamiltonians supported in such domains. Spectral invariants measure the minimal action required to represent a given homology class in Floer homology. These invariants play a central role in the study of symplectic topology and Hamiltonian dynamics. Using the barricades construction, we prove that the spectral invariants with respect to the fundamental and the point classes of Hamiltonians supported in certain domains, do not depend on the ambient manifold. This result is stated formally in Section~\ref{subsec:locality}. Another application of the barricades construction concerns spectral invariants of Hamiltonians with disjoint supports. This problem was studied in \cite{polterovich2014symplectic,seyfaddini2014spectral,ishikawa2015spectral} and lastly in \cite{humiliere2016towards}, where Humili\`ere, Le Roux and Seyfaddini proved that the spectral invariant with respect to the fundamental  class satisfies a ``max formula", namely, the invariant of a sum of disjointly supported Hamiltonians is equal to the maximum over the invariants of the summands. This property does not hold for a general homology class. However, using barricades we show that an inequality holds in general, see Section~\ref{subsec:max_ineq_spec}. A third application of this method concerns the boundary depth, which was defined by Usher in \cite{usher2011boundary} and measures the maximal action gap between a boundary term and its smallest primitive in the Floer chain complex, see Section~\ref{subsec:boundary_depth}. We prove a relation between the boundary depths of disjointly supported Hamiltonians and that of their sum. The last application concerns a new invariant that was constructed by Abbondandolo, Haug and Schlenk in \cite{abbondandolo2019simple}. We give a partial answer to a question posed by them, asking whether a version of Humili\`ere, Le Roux and Seyfaddini's max formula holds for the new invariant, see Section~\ref{subsec:AHS}. 

\subsection{Results.} 
The limitation in Floer theoretic interaction between disjointly supported Hamiltonians is reflected through Floer theoretic invariants of these Hamiltonians and their sum. 
In order to define these invariants, we briefly describe {\it filtered Floer homology}. For more details, we refer to Section~\ref{sec:floer_preliminaries} and the references therein.  
Throughout the paper, $(M,\omega)$ denotes a closed symplectically aspherical manifold, namely, $\omega|_{\pi_2(M)}=0$ and $c_1|_{\pi_2(M)}=0$, where $c_1$ is the first Chern class of $M$. 
Given a Hamiltonian $F:M\times S^1\rightarrow \R$, its symplectic gradient is the vector field given by the equation $\omega(X_F,\cdot)=-dF(\cdot)$. The 1-periodic orbits of the flow of $X_F$, whose set is denoted by $\cP(F)$, correspond to critical points of the action functional associated to $F$ and generate the Floer complex $CF_*(F)$. The differential of this chain complex is defined by counting certain negative-gradient flow lines of the action functional and therefore decreases the value of the action. Note that the gradient of the action functional is taken with respect to a metric induced by an almost complex structure $J$ on $M$. The homology of this chain complex, denoted $HF_*(F)$, is known to be isomorphic to the singular homology of $M$ up to a degree-shift, $HF_*(F)\cong H_{*+n}(M;\Z_2)$. The complex $CF_*(F)$ is filtered by the action value, namely, for every $a\in\R$, we denote by $CF_*^a(F)$ the sub-complex generated by 1-periodic orbits whose action is not greater than $a$. The homology of this sub-complex is denoted by $HF_*^a(F)$.

In what follows we present four applications of the barricades construction, which is an adaptation to Floer theory of the idea presented in Figure~\ref{fig:Morse_bump} and is described in Section~\ref{subsec:barricades_presentation}. 
The class of admissible domains for the barricade construction include symplectic embeddings of nice star-shaped\footnote{A {nice} star-shaped domain is a bounded star-shaped domain in $\R^{2n}$ with a smooth boundary, such that the radial vector field is transverse to the boundary.} domains in $\R^{2n}$ into $M$. In order to present this class in full generality we need to recall a few standard notions.
Let $U\subset M$ be a domain with a smooth boundary. We say that $U$ has a {\it contact type boundary} if there exists a vector field $Y$, called the {\it Liouville vector field}, that is defined on a neighborhood of $\partial U$, is transverse to $\partial U$, points outwards of $U$ and satisfies $\cL_Y\omega= \omega$. 
If the Liouville vector field $Y$ extends to $U$, we say that $U$ is a {\it Liouville domain}.
Finally, a subset $X\subset M$ is called {\it incompressible} if the map $\iota_*:\pi_1(X)\rightarrow\pi_1(M)$, induced by the inclusion $X\hookrightarrow M$, is injective. In particular, every simply connected subset is incompressible.

\begin{defin}
	An open subset $U\subset M$ is called a {\it CIB (Contact Incompressible Boundary) domain} if for each connected component, $U_i$, of $U$, one of the following assertions holds:
	\begin{enumerate}
		\item $\partial U_i$ is of contact-type and is incompressible. 
		\item $U_i$ is an incompressible Liouville domain.
	\end{enumerate} 
\end{defin} 
\begin{exam}
	\begin{itemize}
	\item The image under a symplectic embedding of a nice star-shaped domain in $\R^{2n}$ into $M$ is a CIB domain.	
	\item A non-contractible annulus in $M=\T^2$ is a CIB domain. More generally, if $M=\T^{2n}=\C^n/\Z^{2n}$, then certain tubular neighborhoods of $L=\R^n/\Z^n$ in $M$ are CIB domains.
	\end{itemize}
\end{exam}
\begin{figure}
	\centering
	\begin{subfigure}{.5\textwidth}
		\centering
		\includegraphics[width=.8\linewidth]{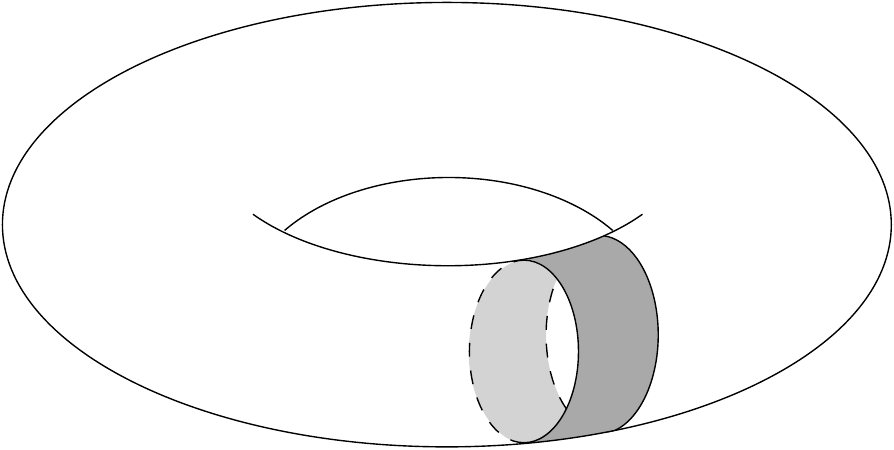}
		\caption{Incompressible embedding}
		\label{fig:Torus_Incompressible}
	\end{subfigure}%
	\begin{subfigure}{.5\textwidth}
		\centering
		\includegraphics[width=.74\linewidth]{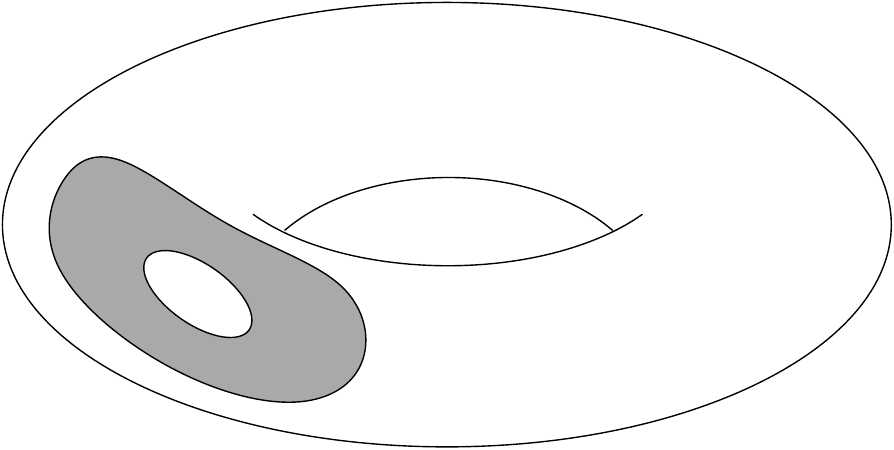}
		\caption{Not incompressible}
		\label{fig:Torus_Not_Incompressible}
	\end{subfigure}
	\caption{Two embeddings of the annulus into $\T^2$. The first is incompressible (as well as its boundary) and hence is a CIB domain. The second embedding is contractible in $\T^2$ and therefore not incompressible.}
	\label{fig:incompressible}
\end{figure}
\begin{rem}\label{rem:CIB_domains}
	\begin{itemize}
		\item Note that a disjoint union of CIB domains is again a CIB domain. 
		\item Every incompressible Liouville domain is a CIB domain. 
		\item Every CIB domain is incompressible, as the fact that $\partial U$ is incompressible implies that $U$ is incompressible, see  Appendix~\ref{app:incompressible}.
	\end{itemize}
\end{rem}

\subsubsection{Locality of spectral invariants and Schwarz's capacities.} \label{subsec:locality}
For a homology class $\alpha\in H_*(M;\Z_2)$ and a Hamiltonian $F$, the spectral invariant $c(F;\alpha)$ is the smallest action value $a$ for which $\alpha$ appears in $HF^a_*(F,J)$, namely,
\begin{equation*}
c(F;\alpha):=\inf\left\{a\ |\ \alpha\in \im(\iota^a_*) \right\},
\end{equation*}
where $\iota^a_*:HF^a_*(F)\rightarrow HF_*(F)$ is induced by the inclusion $\iota^a:CF^a_*(F)\hookrightarrow CF_*(F)$. 	
The following result states that the spectral invariants with respect to the fundamental and the point classes, of a Hamiltonian $F$ supported in a CIB domain, do not depend on the ambient manifold $M$. 
More formally, let $U\subset M$ be a CIB domain and assume that there exists a symplectic embedding, $\Psi:U\hookrightarrow N$, of $U$ into another closed symplectically aspherical manifold $(N,\Omega)$, such that $\Psi(U)$ is a CIB domain in $N$. 
Denote by $c_M(\cdot;\cdot)$, $c_N(\cdot;\cdot)$ the spectral invariants in the manifolds $M,N$ respectively. 

\begin{thm}\label{thm:indep_of_embedding}
	Let $F:M\times S^1\rightarrow\R$ be a Hamiltonian supported in $U$, then 
	\begin{equation}\label{eq:indep_of_embedding}
	c_M(F;[M])=c_N(\Psi_* F;[N])\  \text{ and }\  c_M(F;[pt])=c_N(\Psi_* F;[pt]) ,
	\end{equation}
	where $\Psi_* F :N\times  S^1\rightarrow\R$ is the extension by zero of $F\circ\Psi^{-1}$.
\end{thm}
The assertion of Theorem~\ref{thm:indep_of_embedding} does not hold when $M$ is not symplectically aspherical, or when $U$ is not incompressible in $M$. This is shown in Example~\ref{exa:locality_fails}. Theorem~\ref{thm:indep_of_embedding} also holds for the spectral invariants defined in \cite{frauenfelder2007hamiltonian} on open manifolds obtained as completions of compact manifolds with contact-type boundaries, see  Remark~\ref{rem:locality_open_manifolds}. Moreover, Theorem~\ref{thm:indep_of_embedding} can be extended to certain other homology classes, as stated in Claim~\ref{clm:locality_other_classes}. 
One corollary of Theorem~\ref{thm:indep_of_embedding} concerns Schwarz's relative capacities\footnote{We recall the definition of a capacity in Section~\ref{sec:shwartzs_cap}.}.
\begin{defin}[Schwarz, \cite{schwarz2000action}]
	Let $(M,\omega)$ be a symplectically aspherical manifold. For a subset $A\subset M$ define the {\it spectral capacity}
	\begin{equation}\label{eq:def_spec_capacity}
	c_\gamma(A;M):= \sup\left\{c(F;[M])-c(F;[pt])\ : 
	supp X_F\subset A\times S^1\right\}
	\end{equation}
\end{defin}
In \cite{schwarz2000action} Schwarz shows that if the spectral capacity of the support of $X_F$ is finite and  $\varphi_F^1\neq \id$, then the Hamiltonian flow of $F$ has infinitely many geometrically distinct non-constant periodic
points corresponding to contractible solutions. In Section~\ref{sec:shwartzs_cap}, we use Theorem~\ref{thm:indep_of_embedding} to show that, when $A$ is a contractible domain with a contact-type boundary, its spectral capacity does not depend on the ambient manifold. 
\begin{cor}\label{cor:schwarts_capacity}
	Let $\cS$ be the set of contractible compact symplectic manifolds with   contact-type boundaries that can be embedded into symplectically aspherical manifolds, e.g., nice star-shaped domains in $\R^{2n}$. 	
	Then, Schwarz's spectral capacities, $\{c_\gamma(\cdot;M)\}$, induce a capacity, $c_\gamma$, on $\cS$. 
	In particular, $c_\gamma(A;M)$ is finite for every $A\subset M$ such that $A\in\cS$ and can be symplectically embedded into $(\R^{2n},\omega)$:
	\begin{equation}\label{eq:schwarts_capacity}
	c_\gamma(A;M)=c_\gamma(A)\leq 2e(A;\R^{2n}) <\infty,
	\end{equation}
	where $e(A;\R^{2n})$ is the displacement energy\footnote{We remind the definition of the displacement energy in Section~\ref{sec:floer_preliminaries}, equation (\ref{eq:disp_energy_def}).} of $A$ in $\R^{2n}$.
\end{cor}
\noindent Here we used the fact that every bounded subset of $\R^{2n}$ is displaceable with finite energy.

Another corollary of Theorem~\ref{thm:indep_of_embedding} concerns the notions of  {\it heavy} and {\it super-heavy} sets, which were introduced by Entov and Polterovich in \cite{entov2009rigid}:
A closed subset $X\subset M$ is called heavy if
\begin{equation*}
\zeta(F) \geq \inf_{X\times S^1} F, \quad \forall F\in\cC^\infty(M\times S^1),
\end{equation*}
and is called super-heavy if 
\begin{equation*}
\zeta(F) \leq \sup_{X\times S^1} F, \quad \forall F\in\cC^\infty(M\times S^1),
\end{equation*}
where 
\begin{equation*}
\zeta(F) :=\lim_{k\rightarrow\infty}\frac{c(kF;[M])}{k} 
\end{equation*}
is the {\it partial symplectic quasi-state} associated to the spectral invariant $c$ and the fundamental class. 
The following corollary was suggested to us by Polterovich.
\begin{cor}\label{cor:heavy_sets}
	Let $A\subset M$ be a contractible domain with a contact-type boundary that can be symplectically embedded in $(\R^{2n},\omega_0)$. Then, $M\setminus A$ is super-heavy. In particular, $A$ does not contain a heavy set.  
\end{cor}
\noindent Corollary~\ref{cor:heavy_sets} can be viewed as an extension of the results of \cite{ishikawa2015spectral} to a wider class of domains, when restricting to symplectically aspherical manifolds. 
Theorem~\ref{thm:indep_of_embedding} and Corollaries \ref{cor:schwarts_capacity} and \ref{cor:heavy_sets} are proved in Section~\ref{sec:shwartzs_cap}.

\subsubsection{Max-inequality for spectral invariants.} \label{subsec:max_ineq_spec}
In \cite{humiliere2016towards}, Humili\`ere, Le Roux and Seyfaddini proved a max formula for the spectral invariants, with respect to the fundamental class, of Hamiltonians supported in disjoint incompressible Liouville domains in symplectically aspherical manifolds.
\begin{thm*}[Humili\`ere-Le Roux-Seyfaddini, {\cite[Theorem 45]{humiliere2016towards}}]
	Suppose that $F_1, \dots, F_N$ are Hamiltonians whose supports are contained, respectively, in pairwise disjoint incompressible Liouville domains $U_1,\dots, U_N$. Then,
	$$c(F_1+\cdots +F_N;[M])= \max \{c(F_1;[M]),\dots,c(F_N;[M])\}.$$
\end{thm*}
\noindent The existence of barricades can be used to give an alternative proof for this theorem, as well as to prove a version of it for other homology classes.
Clearly, other homology classes do not satisfy such a max formula - for example, by Poincar\'e duality the class of a point satisfies a min formula. However, an inequality does hold for a general homology class. 
\begin{thm}\label{thm:spectral_inv_max_ineq}
	Let $F,G$ be Hamiltonians supported in disjoint CIB domains and let $\alpha\in H_*(M)$. Then,
	\begin{equation}\label{eq:max-ineq}
	c(F+G;\alpha)\leq \max\left\{c(F;\alpha),c(G;\alpha)\right\}.
	\end{equation}
	Moreover, when $\alpha=[M]$, we have an equality.	
\end{thm}

Notice that, by definition, every incompressible Liouville domain is a CIB domain. Moreover, a disjoint union of CIB domains is again a CIB domain. Hence, the inequality for $N$ Hamiltonians follows by induction.
We also mention that a ``min inequality" does not hold in general, namely, $c(F+G;\alpha)$ might be strictly smaller than $\min\{c(F,\alpha),c(G,\alpha)\}$ as shown in Example~\ref{exa:min_ineq_spec}. 
Theorem~\ref{thm:spectral_inv_max_ineq} is proved in Section~\ref{sec:spec_inv_max_ineq}.

\subsubsection{The boundary depth of disjointly supported Hamiltonians.} \label{subsec:boundary_depth}
In \cite{usher2011boundary}, Usher defined the {\it boundary depth} of a Hamiltonian $F$ to be the largest action gap between a boundary term in $CF_*(F)$ and its smallest primitive, namely
\begin{equation*}
\beta(F):=	\inf\left\{b\in \R\ \big|\  CF^a_*(F)\cap\partial_{F,J}(CF_*(F))\subset \partial_{F,J} (CF_*^{a+b}(F)),\ \forall a\in \R\right\}.
\end{equation*}
The following result relates the boundary depths of disjointly supported Hamiltonians to that of their sum, and is proved in Section~\ref{sec:BD_inequality}. 
\begin{thm}\label{thm:bd_ineq}
	Let $F,G$ be Hamiltonians supported in disjoint CIB domains, then
	\begin{equation}\label{eq:bd_ineq_geq}
	\beta(F+G)\geq \max\left\{\beta(F),\beta(G)\right\}.
	\end{equation}	
\end{thm} 
\noindent Note that equality does not hold in (\ref{eq:bd_ineq_geq}) in general, as shown in Example~\ref{exa:BD_inverse_ineq}.

\subsubsection{Min-inequality for the AHS action selector.} \label{subsec:AHS}
In a recent paper, \cite{abbondandolo2019simple}, Abbondandolo, Haug and Schlenk presented a new construction of an action selector, denoted here by $c_{AHS}$, that does not rely on Floer homology. Roughly speaking, given a Hamiltonian $F$, the invariant $c_{AHS}(F)$ is the minimal action value that ``survives" under all homotopies starting at $F$.
In Section~\ref{sec:AHS_selector}, we review the definition of this selector and a few relevant properties. 
An open problem stated in \cite[Open Problem 7.5]{abbondandolo2019simple} is whether $c_{AHS}$ coincides with the spectral invariant of the point class. 
As a starting point,  Abbondandolo, Haug and Schlenk ask whether $c_{AHS}$ satisfies a {\it min formula} like the one proved by Humili\`ere, Le Roux and Seyfaddini in \cite{humiliere2016towards} for the spectral invariant with respect to the point class\footnote{As mentioned above, they proved a max formula for the spectral invariant of the fundamental class.  By Poincar\'e duality for spectral invariants, this is equivalent to a min formula for the point class.}. 
Due to a result from \cite{humiliere2016towards}, this will imply that 	$c_{AHS}$  coincides with the spectral invariant with respect to the point class in dimension 2 on autonomous Hamiltonians.
In Section~\ref{sec:AHS_selector}, we use barricades in order to prove an inequality for the AHS action selector.
\begin{thm}\label{thm:AHS_min_ineq}
	Let $F$, $G$ be Hamiltonians supported in disjoint incompressible Liouville domains, then,
	\begin{equation}\label{eq:AHS_min_ineq}
	c_{AHS}(F+G)\leq\min\{c_{AHS}(F),c_{AHS}(G)\}.
	\end{equation}
\end{thm}

\subsection{The main tool: Barricades.} \label{subsec:barricades_presentation}
The central construction in this paper is an adaptation of the idea presented in Figure~\ref{fig:Morse_bump} to Floer theory, which is an infinite-dimensional version of Morse theory, applied to the action functional associated to a given Hamiltonian $F:M\times S^1\rightarrow \R$. As in Morse theory, the Floer differential counts certain negative-gradient flow lines of the action functional. These flow lines are called ``Floer trajectories" and correspond to solutions $u:\R\times S^1\rightarrow M$ of a certain partial differential equation, called ``Floer equation" (\ref{eq:FE}), that converge to 1-periodic orbits of the Hamiltonian flow at the ends,
\begin{equation*}
\lim_{s\rightarrow \pm \infty}u(s,t) = x_\pm(t) \text{ for }x_\pm\in\cP(F).
\end{equation*}
In this case we say that  $u$ {\it connects } $x_\pm$, see Section~\ref{sec:floer_preliminaries} for more details.
Following the idea from Morse theory, given a Hamiltonian $F$ supported in a subset $U\subset M$, we wish to construct a perturbation for which Floer trajectories cannot enter or exit the domain. Moreover, we extend this construction to homotopies of Hamiltonians, namely, smooth functions $H:M\times S^1\times \R\rightarrow R$, for the following reason: Most of the results presented above compare Floer theoretic invariants of different Hamiltonians. Such a comparison is usually done using a morphism between the different chain complexes, that is defined by counting solutions of the Floer equation with respect to a homotopy between the two Hamiltonians.
We consider only homotopies that are constant outside of a compact set, namely there exists $R>0$ such that $\partial_s H(\cdot,\cdot,s)$ is supported in  $M\times S^1\times [-R,R]$. We denote by $H_\pm:=H(\cdot, \cdot, \pm R)$ the ends of the homotopy $H$. Note that we think of single Hamiltonians as a special case of this setting, by identifying them with constant homotopies, $H(x,t,s)=F(x,t)$.
Given an almost complex structure $J$ on $M$, we consider solutions of the Floer equation (\ref{eq:FE}) with respect to the pair $(H,J)$. The property of having a barricade is defined through constraints on these solutions.
\begin{defin}\label{def:barricade}
	Let $U$ and $U_\circ$ be open subsets of $M$ such that $U_\circ \Subset U$. We say that a pair $(H,J)$ of a homotopy and an almost complex structure {\it has a barricade in $U$ around $U_\circ$} if the periodic orbits of $H_\pm$ do not intersect the boundaries $\partial U$, $\partial U_\circ$, and for every $x_{\pm}\in\cP(H_{\pm})$ and every solution $u:\R\times S^1\rightarrow M$ of the corresponding Floer equation, connecting $x_{\pm}$, we have:
	\begin{enumerate}
		\item If $x_{-}\subset U_\circ$ then ${\im(u)}\subset U_\circ$.
		\item If $x_{+}\subset U$ then ${\im(u)}\subset U$.
	\end{enumerate} 
\end{defin}
See Figure~\ref{fig:barricade_def} for an illustration of solutions satisfying and not satisfying these constraints. 
When $H$ is a constant homotopy, corresponding to a Hamiltonian $F$, the presence of a barricade yields a decomposition of the Floer complex, in which the differential admits a triangular block form. To describe this decomposition, let us fix some notations: For a subset $X\subset M$ denote by $C_X(F)\subset CF_*(F)$ the subspace generated by orbits contained in $X$, and by $\partial|_X$ the map obtained by counting only solutions that are contained in $X$. Then, for a Floer regular pair $(F,J)$ with a barricade in $U$ around $U_\circ$, 
\begin{equation}\label{eq:block_form}
	CF_*(F):=C_{U_\circ}(F)\oplus C_{U^c}(F)\oplus C_{U\setminus U_\circ} (F),\quad \partial_{F,J} = \left(\begin{array}{ccc}
	\partial|_{U_\circ} & 0&\partial|_{U} \\
	0& * & * \\
	0 & 0& \partial|_{U} 
	\end{array}\right).
\end{equation}
	The block form (\ref{eq:block_form}) implies that the differential restricts to the subspace $C_{U_\circ}(F)$. We study the homology of the resulting subcomplex $(C_{U_\circ}(F),\partial|_{U_\circ})$ in Section~\ref{subsec:hom_of_subcomplex}.  
\begin{figure}
	\centering
	\begin{subfigure}{.5\textwidth}
		\centering
		\includegraphics[width=.8\linewidth]{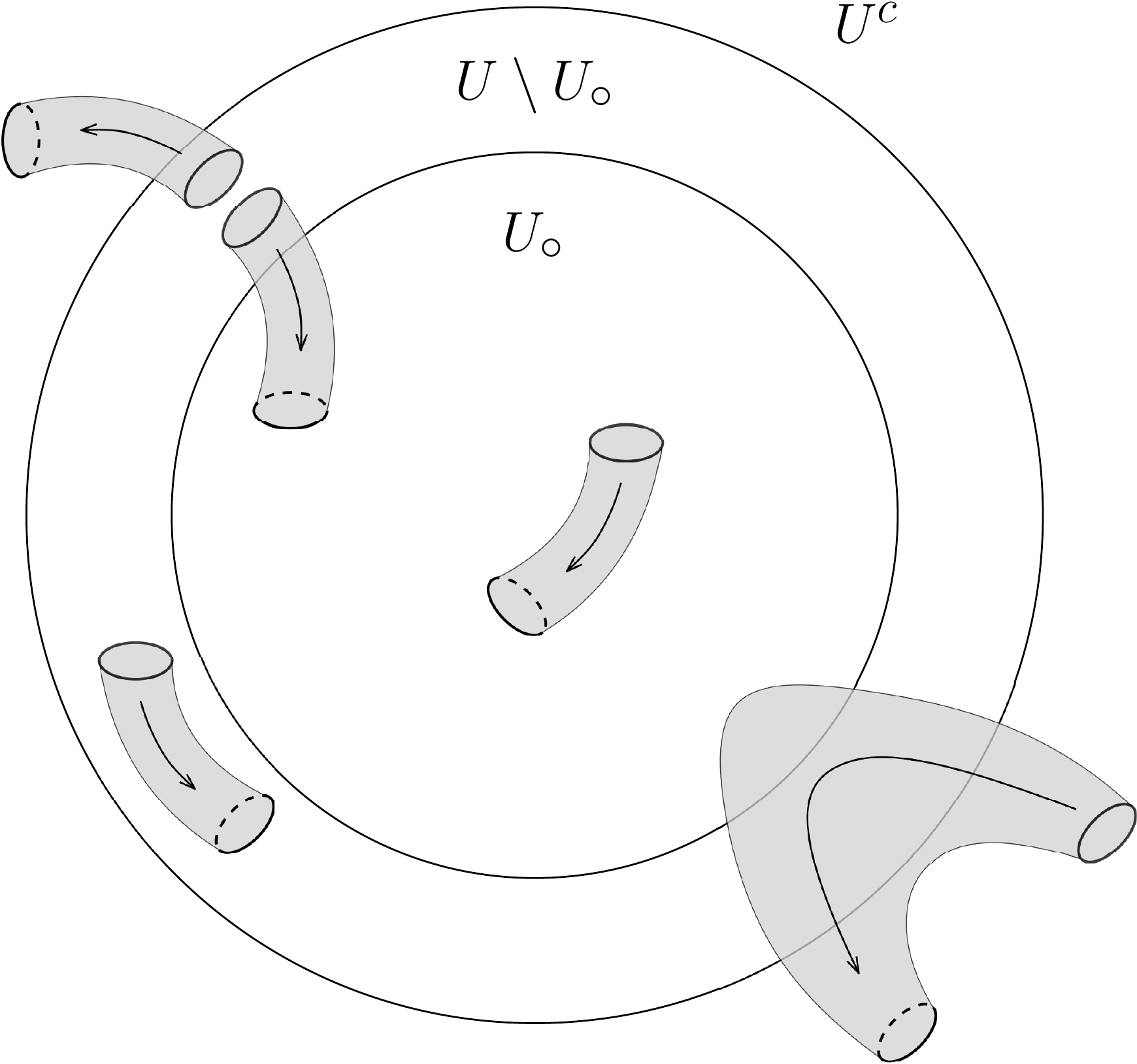}
		\caption{Allowed solutions}
		\label{fig:CylindersYes}
	\end{subfigure}%
	\begin{subfigure}{.5\textwidth}
		\centering
		\includegraphics[width=.74\linewidth]{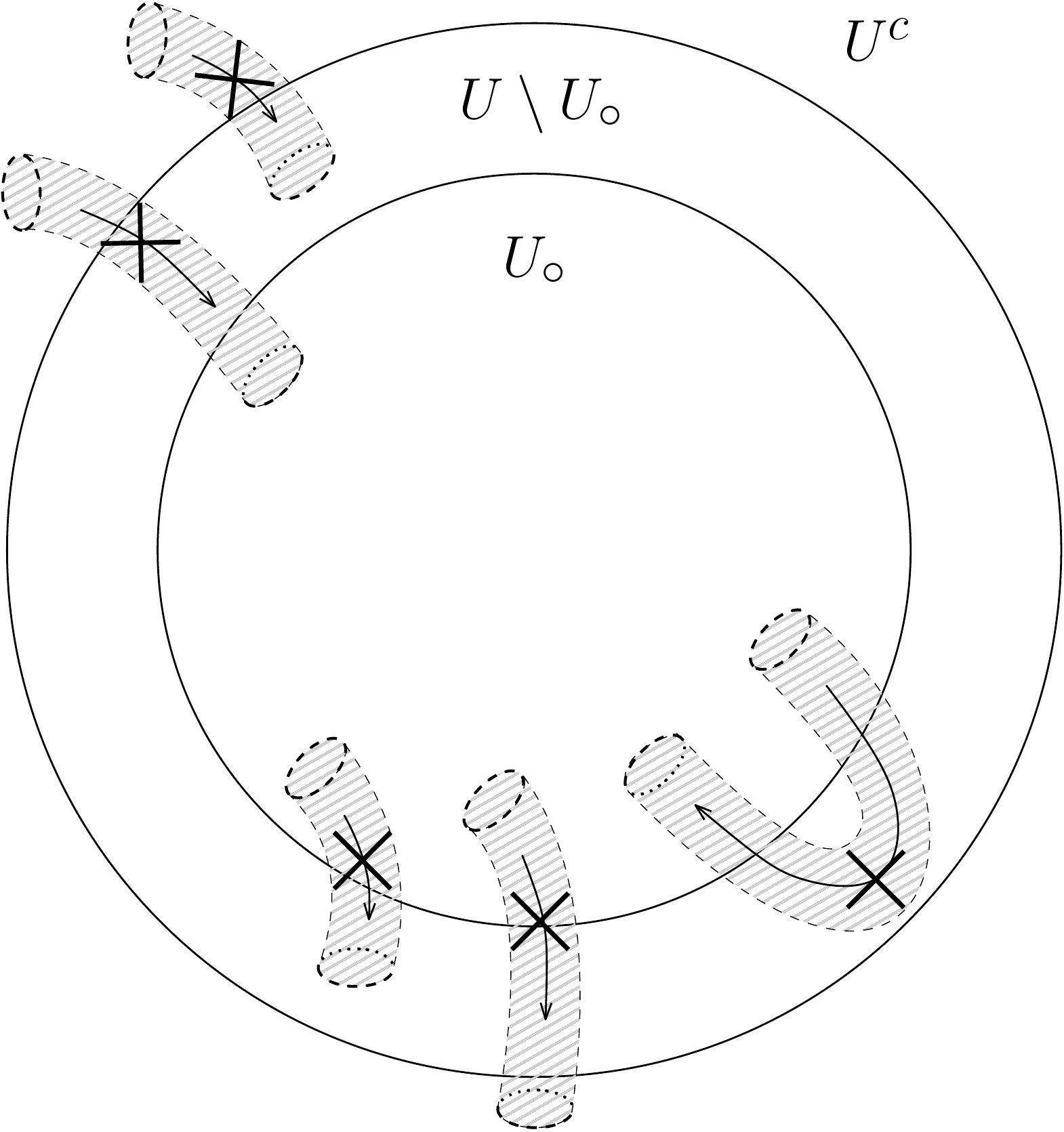}
		\caption{Forbidden solutions}
		\label{fig:CylindersNo}
	\end{subfigure}
	\caption{ An illustration of allowed and forbidden solutions for a pair $(H,J)$ with a barricade.}
	\label{fig:barricade_def}
\end{figure}

Given a homotopy $H$ that is compactly supported in a CIB domain, we construct a small perturbation $h$ of $H$ and an almost complex structure $J$, such that $(h,J)$ has a barricade.
\begin{thm}\label{thm:barricade}
	Let U be a CIB domain and let $H:M\times S^1\times \R\rightarrow\R$ be a homotopy of Hamiltonians, supported in $U\times S^1\times \R$, such that $\partial_s H$ is compactly supported. Then, there exist a $\cC^\infty$-small perturbation $h$ of $H$ and an almost complex structure $J$ such that the pairs $(h,J)$ and $(h_\pm, J)$ are Floer-regular and have a barricade in $U$ around $U_\circ$.
	In particular, when $H$ is independent of the $\R$-coordinate (namely, it is a single Hamiltonian), $h$ can be chosen to be independent of the $\R$-coordinate as well\footnote{For a constant homotopy $H$, the perturbation $h_-$ approximates $H$ and is independent of the $\R$-coordinate.}.
\end{thm} 
This result is proved in Section~\ref{sec:Barricades}, by an explicit construction of the perturbation $h$ and the almost complex structure $J$. We remark that the assumptions on $(M,\omega)$ being symplectically aspherical and $U$ having either incompressible boundary or being an incompressible Liouville domain are crucial for this construction. See the proofs of Lemmas \ref{lem:cross_solutions}-\ref{lem:solution_contained_U} for details. 

\subsection{Related works.}
There have been several works studying the Floer-theoretic interaction between disjointly supported Hamiltonians, mainly through the spectral invariants of these Hamiltonians and their sum. Early works in this direction, mainly by Polterovich \cite{polterovich2014symplectic}, Seyfaddini \cite{seyfaddini2014spectral} and Ishikawa \cite{ishikawa2015spectral}, established upper bounds for the invariant of the sum of Hamiltonians, which depend om the supports. Later, Humili\`ere, Le Roux and Seyfaddini \cite{humiliere2016towards} proved that in certain cases the invariant of the sum is equal to the maximum over the invariants of each individual summand. The method was also conceptually different. While previous works relied solely on the properties of spectral invariants,  Humili\`ere, Le Roux and Seyfaddini studied the Floer complex itself. We also take this approach and study the interaction between disjointly supported Hamiltonians on the level of the Floer complex, but our  methods are substantially different.

In a broader sense, it is worth to mention two works which regard {\it symplectic homology}. Symplectic homology is an umbrella term for a type of homological invariants of symplectic manifolds, or of subsets of symplectic manifolds, which are constructed via a limiting process from the Floer complexes of properly chosen Hamiltonians.
In this setting, questions regarding disjointly supported Hamiltonians correspond to local-to-global relations, such as a Mayer-Vietoris sequence.
In  \cite{cieliebak2018symplectic}, Cieliebak and Oancea   defined symplectic homology for Liouville domains and Liouville cobordisms and proved a Mayer-Vietoris relation. Their method include ruling out the existence of certain Floer trajectories, and partially rely on a work by Abouzaid and Seidel, \cite{abouzaid2010open}. Versions of some of these arguments are being used in Section~\ref{sec:Barricades} below.
Another work concerning Mayer-Vietoris property is by Varolgunes, \cite{varolgunes2018mayer}, in which he defines an invariant of compact subsets of closed symplectic manifolds, which is called relative symplectic homology, and finds a condition under which the Mayer-Vietoris property holds. In particular, for a union of disjoint compact sets, the relative symplectic homology splits into a direct sum.

\subsection*{Structure of the paper.}	
In Section~\ref{sec:floer_preliminaries} we review the necessary preliminaries from Floer theory and contact geometry. In Section~\ref{sec:Barricades} we construct barricades and prove Theorem~\ref{thm:barricade}. We then use it to prove Theorem~\ref{thm:indep_of_embedding} in Section~\ref{sec:shwartzs_cap}. In Section~\ref{sec:open_manifolds}, we discuss the relation to Floer homology on certain open manifolds and two extensions of Theorem~\ref{thm:indep_of_embedding}. Sections \ref{sec:spec_inv_max_ineq}-\ref{sec:AHS_selector} are dedicated to the proofs of Theorems \ref{thm:spectral_inv_max_ineq}-\ref{thm:AHS_min_ineq} respectively. Finally, on Section~\ref{sec:transversality_compactness} we prove several transversality  and compactness claims that are required for establishing the main results. Appendix~\ref{app:incompressible} contains a claim about  incompressibility, whose proof we include for the sake of completeness.

\subsection*{Acknowledgements.}	
We are very grateful to Lev Buhovsky and Leonid Polterovich for their guidance and insightful inputs. We thank Vincent Humili\`ere and Sobhan Seyfaddini for useful discussions and, in particular, for proposing a question that led us to Theorem~\ref{thm:indep_of_embedding}. We thank Alberto Abbondandolo, Carsten Haug and Felix Schlenk for sharing a preliminary version of their paper \cite{abbondandolo2019simple} with us. We also thank Felix Schlenk for suggesting to consider Floer homology on open manifolds, which led us to Section~\ref{sec:open_manifolds}.
We thank Mihai Damian and Jun Zhang for discussions concerning the arguments appearing in Section~\ref{app:regular_pair}.
Finally, we thank Henry Wilton for discussions regarding Proposition~\ref{pro:incompressible} over MathOverflow.  

Y.G. was partially supported by ISF Grant 1715/18 and by ERC Advanced grant 338809.
S.T. was partially supported by ISF Grant 2026/17 and by the Levtzion Scholarship.

\section{Preliminaries from Floer theory.}\label{sec:floer_preliminaries}
In this section we briefly review some preliminaries from Floer theory and contact geometry on closed symplectically aspherical manifolds (namely, when $\omega|_{\pi_2(M)}=0$ and $c_1|_{\pi_2(M)}=0$, where $c_1$ is the first Chern class of $M$). 
For more details see, for example, \cite{audin-damian,mcduff2012j,polterovich2014function}.
We also fix some notations that will be used later on.

\subsection{Floer homology, regularity and notations.}
Let $F:M\times S^1\rightarrow \R$ be a Hamiltonian on $M$. 
The corresponding action functional $\cA_F$ is defined on the space of contractible loops in $M$ by
\begin{equation*}
\cA_F(x) :=\int_0^1 F(x(t),t)\ dt - \int \bar{x}^* \omega,
\end{equation*}
where $x:S^1\rightarrow M$ and $\bar x:D^2\rightarrow M$ satisfies $\bar{x}(e^{2\pi i t})=x(t)$. The critical points of the action functional are the contractible 1-periodic orbits of the flow of $X_F$ and their set is denoted by $\cP(F)$. 
The Hamiltonian  $F:M\times S^1\rightarrow\R$ is said to be {\it non-degenerate} if the graph of the linearized flow of $X_F$ at time 1 intersects the diagonal in $TM\times TM$ transversally. In this case, the flow of $X_F$ has finitely many 1-periodic orbits. The Floer complex $CF_*(F)$ is spanned by these critical points, over $\Z_2$\footnote{The Floer complex can be defined over other coefficient rings, we chose to work in the simplest setting.}.
A time dependent $\omega$-compatible\footnote{An almost complex structure $J$  is called {\it $\omega$-compatible} if $\omega(J\cdot, J\cdot)=\omega(\cdot,\cdot)$ and $\omega(\cdot ,J\cdot)$ is an inner-product on $TM$. All almost complex structures considered in this paper are assumed to be $\omega$ compatible.} almost complex structure $J$ induces a metric on the space of contractible loops, in which negative-gradient flow lines of $\cA_F$  
are maps $u:\R\times S^1\rightarrow M$ that solve the Floer equation
\begin{equation}\tag{FE}\label{eq:FE}
\partial_su(s,t)+J\circ u(s,t) \cdot\left(\partial_t u(s,t)-X_F\circ u(s,t)\right)=0.
\end{equation}
The {\it energy} of such a solution is defined to be
$
E(u):=\int_{\R\times S^1}\|\partial_s u\|_J^2\ ds\ dt
$, where $\|\cdot\|_J$ is the norm induced by the the inner product associated to $J$, $\left<\cdot,\cdot\right>_J:=\omega(\cdot,J\cdot)$. When the Hamiltonian $F$ is non-degenerate, for every solution $u$ with finite energy, there exist $x_\pm \in \cP(F)$ such that $\lim_{s\rightarrow \pm\infty}u(s,t)=x_\pm(t)$, and we say that $u$ {\it connects } $x_\pm$. The well known energy identity for such solutions is a consequence of Stokes' theorem:     
\begin{equation}\label{eq:energy_id_Hamiltonian}
E(u):=\int_{\R\times S^1}\|\partial_s u\|^2_{J}\ ds\ dt
=\cA_{F_-}(x_-)-\cA_{F_+}(x_+),
\end{equation} 

For two 1-periodic orbits $x_\pm\in\cP(F)$ of $F$, we denote by $\cM_{(F,J)}(x_-,x_+)$ the set of all solutions $u:\R\times S^1\rightarrow M$ of the Floer equation (\ref{eq:FE}) that satisfy $\lim_{s\rightarrow \pm\infty}u(s,t)=x_\pm(t)$. Notice that $\R$ acts on this set by translation in the $s$ variable. We denote by $\cM_{(F,J)}$ the set of all finite energy solutions. It is well known (e.g., \cite[Theorem 6.5.6]{audin-damian}) that when $F$ is non-degenerate, $\cM_{(F,J)}:=\cup_{x_\pm\in \cP(H)}\cM_{(F,J)}(x_-,x_+)$. Moreover, for non-degenerate Hamiltonians one can define an index $\mu:\cP(F)\rightarrow\Z$, called {\it the Conley-Zehnder index}, which assigns an integer to each orbit (see e.g. \cite[Chapter 7]{audin-damian} and the references therein). The Floer complex is graded by the index $\mu$, namely, for $k\in\Z$, $CF_k(F)$ is the $\Z_2$-vector space spanned by the periodic orbits $x\in\cP(F)$ for which $\mu(x)=k$.

In order to define the Floer differential for the graded complex $CF_*(F)$, one needs an almost complex structure $J$, such that the pair $(F,J)$ is {\it Floer-regular}. The definition of Floer regularity concerns the surjectivity of a certain linear operator and is given in Section~\ref{app:regular_pair}. When the pair $(F,J)$ is Floer-regular, the space of solutions, $\cM_{(F,J)}(x_-,x_+)$ is a smooth manifold of dimension $\mu(x_-)-\mu(x_+)$, for all $x_\pm\in \cP(F)$. Dividing $\cM_{(F,J)}(x_-,x_+)$ by the $\R$ action, we obtain a manifold of dimension $\mu(x_-)-\mu(x_+)-1$.
 
Recall that an element $a\in CF_*(F)$ is a formal linear combination $a=\sum_x a_x\cdot x$ where $x\in\cP(F)$ and $a_x\in \Z_2$. For a Floer-regular pair $(F,J)$, the Floer differential $\partial_{(F,J)}:CF_*(F)\rightarrow CF_{*-1}(F)$ is defined by 
\begin{equation}
\partial_{(F,J)}(a) := \sum_{x_-\in \cP(F)}\sum_{\tiny{\begin{array}{c} 
		x_+\in \cP(F),\\ \mu(x_+)=\mu (x_-)-1
		\end{array}}} a_{x_-}\cdot \#_2\left(\frac{\cM_{(F,J)}(x_-,x_+)}{\R}\right)\cdot x_+,
\end{equation}
where $\#_2$ is the number of elements modulo 2.
The homology of the complex $(CF_*(F),\partial_{(F,J)})$ is denoted by $HF_*(F,J)$ or $HF_*(F)$. A fundamental result in Floer theory states that Floer homology is isomorphic to the singular homology, with a degree shift, $HF_*(F,J)\cong H_{*-n}(M;\Z_2)$. The Floer complex admits a natural filtration by the action value. We denote by $CF_*^a(F)$ the sub-complex spanned by critical points with value not-greater than $a$. Since the differential is action decreasing, it can be restricted to the sub-complex $CF^a_*(F)$. The homology of this sub-complex is denoted by $HF_*^a(F,J)$.

It is well known that when $F$ is a $\cC^2$-small Morse function, its 1-periodic orbits are its critical points, $\cP(F)\cong Crit(F)$, and their actions are the values of $F$, $\cA_F(p) = F(p)$. In this case, the Floer complex with respect to a time-independent almost complex structure $J$, coincides with the Morse complex when the degree is shifted by $n$ (which is half the dimension of $M$), since $\operatorname{Morse-ind}(p) = \mu(p)+n$ for every $p\in Crit(F)\cong\cP(F)$:
\begin{equation*}
\left(CF_*(F), \partial_{(F,J)}^{Floer}\right) = 
\left(CM_{*+n}(F), \partial_{(F,\left<\cdot, \cdot\right>_J)}^{Morse}\right).
\end{equation*} For a proof, see, for example, \cite[Chapter 10]{audin-damian}.  
We conclude this section by fixing notations that will be used later on. 
\begin{notation}\label{not:CF_element_contained}
	Let $a=\sum_x a_x\cdot x$ be an element of $CF_*(H)$.
	
	\begin{itemize}
		\item We say that $x\in a$ if $a_x\neq 0$.
		\item We denote the maximal action of an orbit from $a$ by $\lambda_H(a):=\max\{\cA_H(x):a_x\neq 0\}$. 
		
		\item For a subset $X\subset M$, let $C_X(H)\subset CF_*(H)$ be the subspace spanned by the 1-periodic orbits of $H$ that are contained in $X$.	Moreover, let $\pi_X:CF_*(H)\rightarrow C_X(H)$ be the projection onto this subspace. Note that $C_X(H)$ is not necessarily a subcomplex, and $\pi_X$ is not a chain map in general.
	\end{itemize}
\end{notation}

\subsection{Communication between Floer complexes using homotopies.}
Now let $H:M\times S^1\times\R\rightarrow\R$ denote a homotopy of Hamiltonians, rather than a single Hamiltonian. Throughout the paper, we consider only homotopies that are constant outside of a compact set. Namely, there exists $R>0$ such that $\partial_s H|_{|s|>R}=0$, and we denote by $H_\pm(x,t) := \lim_{s\rightarrow\pm \infty}H(x,t,s)$ the ends of the homotopy $H$. 
Given an almost complex structure $J$, we consider the Floer equation (\ref{eq:FE}) with respect to the pair $(H,J)$:
\begin{equation*}
\partial_su(s,t)+J\circ u(s,t) \cdot\left(\partial_t u(s,t)-X_{H_s}\circ u(s,t)\right)=0,
\end{equation*} 
where $H_s(\cdot,\cdot):=H(\cdot,\cdot,s)$. We sometimes refer this equation as ``the $s$-dependent Floer equation", to stress that it is defined with respect to a homotopy of Hamiltonians. For 1-periodic orbits $x_\pm\in \cP(H_\pm)$, we denote by $\cM_{(H,J)}(x_-,x_+)$ the set of all solutions $u:\R\times S^1\rightarrow M$ of the $s$-dependent Floer equation (\ref{eq:FE}) that satisfy $\lim_{s\rightarrow \pm\infty}u(s,t)=x_\pm(t)$. As before, $\cM_{(H,J)}$ denotes the set of all finite energy solutions and when the ends, $H_\pm$, are non-degenerate, $\cM_{(H,J)} = \cup_{x_\pm\in \cP(H_\pm)} \cM_{(H,J)}(x_-,x_+)$ (see, for example, \cite[Theorem 11.1.1]{audin-damian}).
The  energy identity for homotopies is:     
\begin{eqnarray}\label{eq:energy_id_homotopies}
	E(u)&:=&\int_{\R\times S^1}\|\partial_s u\|^2_{J}\ ds\ dt\nonumber\\
	&=&\cA_{H_-}(x_-)-\cA_{H_+}(x_+) +\int_{\R\times S^1}\partial_s H\circ u\ ds\ dt. 
\end{eqnarray}

As in the case of Hamiltonians, the definition of Floer-regularity concerns the surjectivity of a certain linear operator and is given in Section~\ref{app:regular_pair}. For a Floer-regular pair, $(H,J)$, the space $\cM_{(H,J)}(x_-,x_+)$ is a smooth manifold of dimension $\mu(x_-)-\mu(x_+)$.
In this case, one can define a degree-preserving chain map, called {\it the continuation map}, between the Floer complexes of the ends, $\Phi:CF_*(H_-)\rightarrow CF_*(H_+)$,  by 
\begin{equation}\label{eq:continuation_def} 
\Phi(a)=\sum_{x_-\in a}\sum_{\tiny{\begin{array}{c} 
		x_+\in \cP(H_+),\\ \mu(x_+)=\mu(x_-)
		\end{array}}}a_{x_-}\cdot \#_2\cM(x_-,x_+)\cdot x_+.
\end{equation} 
The regularity of the pair guarantees that the map $\Phi$ is a well defined chain map that induces isomorphism on homologies, see, e.g., \cite[Chapter 11]{audin-damian}.  

\subsection{Contact-type boundaries.}
In order to construct barricades for Floer solutions around a given domain, we need the boundary to have a {\it contact structure}:
Let $U\subset M$ be a domain with a smooth boundary. We say that $U$ has a {\it contact type boundary} if  there exists a vector-field $Y$, called the {\it Liouville vector field}, that is defined on a neighborhood of $\partial U$, is transverse to $\partial U$, points outwards of $U$ and satisfies $\cL_Y\omega= \omega$. The differential form $\lambda:=\iota_Y\omega$ is a primitive of $\omega$, namely $d\lambda =\omega$. The Reeb vector field $R$ is then defined by the following equations:
\begin{eqnarray}
R\in\ker d\lambda|_{T\psi^\tau\partial U},\quad\lambda(R)=1,
\end{eqnarray}
where $\psi^\tau$ is the flow of $Y$.
We stress that the differential form $\lambda$ and the vector field $R$ are defined wherever the Liouville vector field $Y$ is defined.
If the Liouville vector field $Y$ extends to $U$, we say that $U$ is a {\it Liouville domain}.

\section{Barricades for solutions of the ($s$-dependent) Floer equation.} \label{sec:Barricades}
In what follows, $H:M\times S^1\times\R\rightarrow\R$ denotes a homotopy of (time-dependent) Hamiltonians and $J$ denotes a (time-dependent) almost complex structure. We assume that $\partial_s H$ is compactly supported and  denote $H_\pm:=\lim_{s\rightarrow\pm \infty}H(\cdot, \cdot,s)$. Note the we consider the case where $H$ is a single Hamiltonian as a particular case, by identifying it with a constant homotopy.
Fix a CIB domain $U\subset M$, and denote by $Y$ and $R$ the Liouville and Reeb vector fields respectively. Then, $\lambda=\iota_Y\omega$ is the contact form on the boundary $\partial U$. The flow $\psi^{\tau}$ of $Y$ is called {\it the Liouville flow}, and is defined for short times.

In order to prove Theorem~\ref{thm:barricade}, namely, that there exist a perturbation $h$ of $H$ and an almost complex structure $J$ such that $(H,J)$ has a barricade, we construct $h$ and $J$ explicitly. Let us sketch the idea of this construction before giving the details. 
\begin{itemize}
	\item To construct $h$, we first add to $H$ a non-negative bump function in the radial coordinate, which is defined on a neighborhood of $\partial U$ using the Liouville flow. Then, we take $h$ to be a small non-degenerate perturbation of it.
	\item The almost complex structure $J$ is taken to be {\it cylindrical} near $\partial U$ (see Definition~\ref{def:cylindrical} below).
\end{itemize}
We want to rule out the existence of solutions violating the constrains from Definition~\ref{def:barricade}. Suppose there exists a solution $u$ connecting $x_-\subset U_\circ$ with $x_+\subset U_\circ^c$. Then, the image of $u$ intersects $\partial U_\circ$, say along a loop $\Gamma$. We first bound the action of $\Gamma$ (Lemma~\ref{lem:bdry_loop}), and then conclude a negative upper bound for the action of $x_+$ (Lemma~\ref{lem:cross_solutions}). Since $h\approx0$ on $U_\circ^c\supset x_+$, the action of $x_+$ can be taken to be arbitrarily close to zero, in contradiction.

\subsection{Preliminary computations.}
Some of the arguments and results in this section were carried out by Cieliebak and Oancea  in \cite{cieliebak2018symplectic} for the setting of completed Liouville domains, instead of closed symplectically aspherical manifolds. Specifically, some of the computations appearing in the proofs of Lemma~\ref{lem:bdry_loop} and Lemma~\ref{lem:solution_contained_U} can be found in the proof of \cite[Lemma 2.2]{cieliebak2018symplectic}, which follows Abouzaid and Seidel's work, in \cite[Lemma 7.2]{abouzaid2010open}.
\begin{figure}
	\centering
	\includegraphics[scale=0.7]{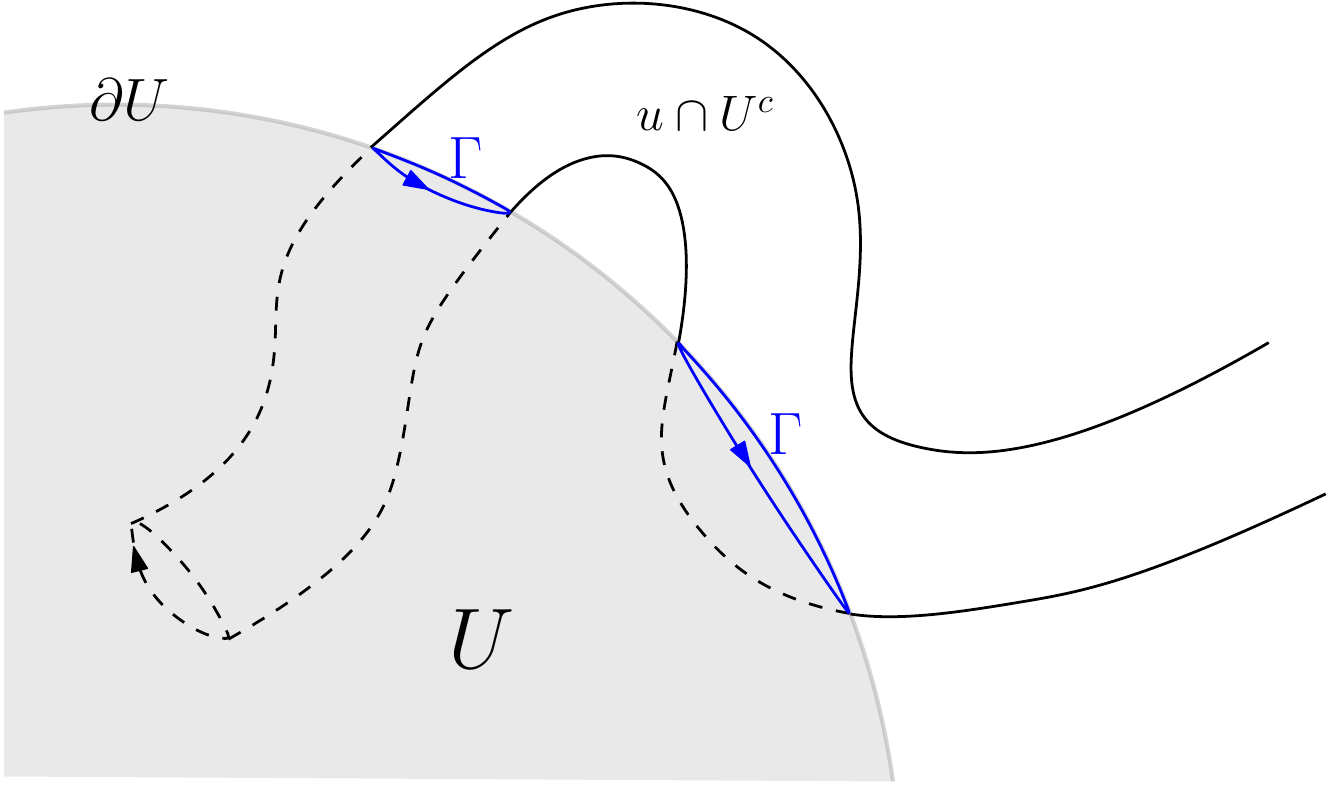}
	\caption{\small{An example of the setting described in Lemma~\ref{lem:bdry_loop}. 
			The gray region is the set $U\subset M$ and the loops $\Gamma$ are given by $\im (u)\cap \partial U =\Gamma$, oriented as the boundary of $\im(u)\cap U^c$. 
	}}
	\label{fig:bdry_loop}
\end{figure}

\begin{defin}\label{def:cylindrical}
	We say that a pair $(H,J)$ of a homotopy and an almost complex structure is {\it $\delta$-cylindrical near $\partial U$}, for $\delta\in \R\setminus\{0\}$, if
	\begin{enumerate}
		\item $J$ is cylindrical near $\partial U$, namely,  $JY=R$ on an open neighborhood of $\partial U$.
		\item $\partial U\times S^1\times \R=\{H=c\}$ is a regular level set of $H$.
		\item $\grad_J H=\delta Y$ on a neighborhood of $\partial U$ and $H$ has no 1-periodic orbits in this neighborhood.
	\end{enumerate}
\end{defin}
\noindent We remark that conditions 2,3 in the above definition imply that, near $\partial U$, $H$ does not depend on the $\R$-coordinate. 
Suppose that $(H,J)$ is $\delta$-cylindrical near $\partial U$ and let $u:\R\times S^1 \rightarrow M$ be a solution of the ($s$-dependent) Floer equation (\ref{eq:FE}) with finite energy, $E(u)<\infty$.
The following lemma gives an upper bound for the integral of $\lambda$ along the oriented curve  $\Gamma:=\partial (\im(u)\cap U^c)$, see Figure~\ref{fig:bdry_loop}.

\begin{lemma}\label{lem:bdry_loop}
	Let  $(H,J)$ be a pair that is $\delta$-cylindrical near $\partial U$ and let  $u:\R\times S^1$ be a finite-energy solution of the $s$-dependent Floer equation connecting $x_\pm\in\cP(H_\pm)$. Suppose that $u$ intersects $\partial U$ transversely and let $\Gamma:=\im(u)\cap \partial U$ denote the intersection, oriented as the boundary of $\im(u)\cap U^c$. Then,
	\begin{equation}
	\int_\Gamma \lambda \leq \begin{cases}
	-\delta, & \text{if } x_-\subset U,\ x_+\subset U^c,\\
	\delta, & \text{if } x_-\subset U^c,\ x_+\subset U,\\
	0, & \text{if } x_\pm\subset U \text{ or } x_\pm\subset U^c.
	\end{cases}
	\end{equation}
\end{lemma}

\begin{proof}
	Set $\Sigma:=u^{-1}(U^c)\subset\R\times S^1$ and denote its boundary by $\gamma$, then $u(\gamma) = \Gamma$, since $x_\pm$ do not intersect $\partial U$. 
	The orientation on $\Sigma$ is given by the positive frame $(\partial_s,\partial_t)$. Let $\gamma_i$ be a connected component of $\gamma$, then $\Gamma_i:=u(\gamma_i)$ is connected. Let $\tau\in [0, T_i]$ be a unit-speed parametrization of $\gamma_i$, and notice that this induces parametrization on $\Gamma_i$. Denote by $\nu(\tau)$ the outer normal to $\Sigma$ at $\gamma_i(\tau)$, then $\dot\gamma_i (\tau)= j\nu(\tau)$, where $j$ is the standard complex structure on $\R\times S^1$ (i.e., $j\partial_s=\partial_t$). Pushing  $(\nu(\tau),\dot{\gamma_i}(\tau))$ to $TM$ we obtain
 	\begin{equation*}
 	N(\tau) := Du\left(\nu(\tau)\right),\quad \dot\Gamma_i(\tau) = Du\left(\dot\gamma_i(\tau)\right).
 	\end{equation*}
 	We remark that $N(\tau)$ is not necessarily normal to $\partial U$ (with respect to the inner product induced by $J$), but is always pointing inwards (or tangent to the boundary), see Figure~\ref{fig:bdry_loop2}.
	\begin{figure}
		\centering
		\includegraphics[scale=0.85]{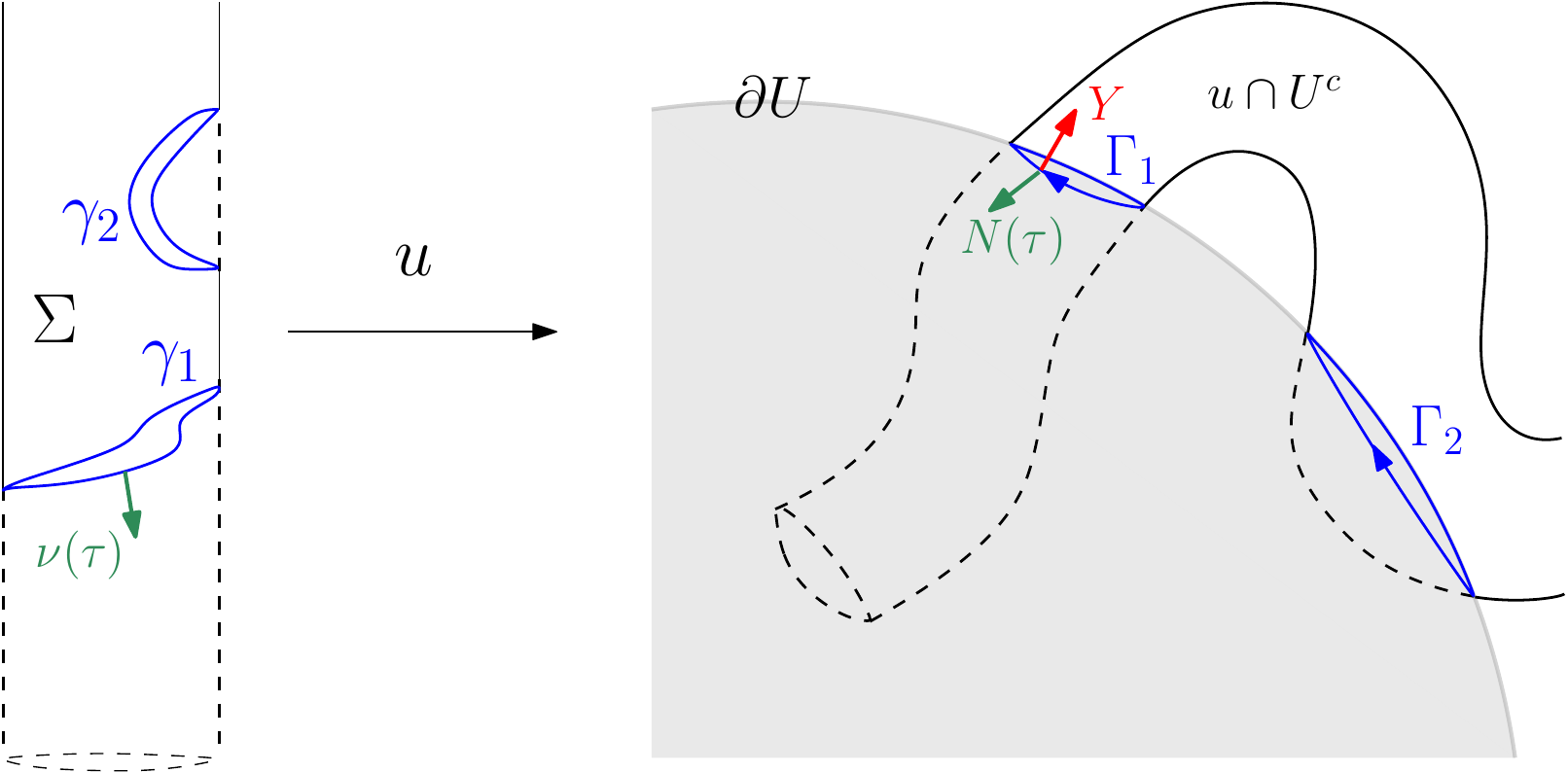}
		\caption{\small{The normal $\nu(\tau)$ to the component $\gamma_1$ of $\partial\Sigma$ and its image, $N(\tau)$, under $Du$. 
		}}
		\label{fig:bdry_loop2}
	\end{figure}
	The relation between $N(\tau)$ and $\dot\Gamma_i(\tau)$ goes through the Floer equation (\ref{eq:FE}), which can be written in the following form:
 	\begin{equation*}
 	J\circ Du = Du\circ j -X_H\circ u\cdot\left<\cdot,\partial_s\right>_j+JX_H\circ u\cdot\left<\cdot,\partial_t\right>_j.
 	\end{equation*}
 	It follows that $\dot\Gamma_i(\tau)$ can be written as a linear combination of $JN(\tau)$, the gradient of $H$ and the symplectic gradient of $H$:
 	\begin{eqnarray*}
 	\dot\Gamma_i(\tau) &=& Du(\dot \gamma_i(\tau)) = Du\left(j\nu(\tau)\right)\\
 	&=&JDu\left(\nu(\tau)\right)+X_H\circ u\cdot\left<\nu(\tau),\partial_s\right>_j -JX_H\circ u\cdot\left<\nu(\tau),\partial_t\right>_j\\
 	&=& JN(\tau) +X_H\circ u\cdot\left<\nu(\tau),\partial_s\right>_j -JX_H\circ u\cdot\left<\nu(\tau),\partial_t\right>_j.
 	\end{eqnarray*}
  	Using this to compute the integral of $\lambda$  along $\Gamma_i$, we obtain
 	\begin{eqnarray}\label{eq:two_terms}
 	\int_{\Gamma_i}\lambda &=& \int \lambda(\dot{\Gamma_i}(\tau))\ d\tau= \int \omega\left(Y\circ\Gamma_i(\tau),\dot{\Gamma_i}(\tau)\right) d\tau \nonumber\\
 	&=& \int \omega\left(Y\circ\Gamma_i(\tau),JN(\tau)\right) d\tau 
 	 +\int \left[\omega(Y,X_H)\cdot\left<\nu,\partial_s\right>_j -\omega(Y,JX_H)\cdot\left<\nu,\partial_t\right>_j\right]  d\tau \nonumber\\
	&=& \int \left<Y\circ\Gamma_i(\tau),N(\tau)\right>_J d\tau 
	\nonumber\\ && 
	 +\int \left[\omega(Y,J\nabla_J H)\cdot\left<\nu,\partial_s\right>_j -\omega(Y,-\nabla_J H)\cdot\left<\nu,\partial_t\right>_j\right]  d\tau. 
	\nonumber 
	\end{eqnarray}
	Recalling our assumptions that $\nabla_J H=\delta Y$ on $\partial U$ and that $JY$ is the Reeb vector-field, we obtain
	\begin{eqnarray}
	\int_{\Gamma_i}\lambda 
	&=& \int \left<Y\circ\Gamma_i(\tau),N(\tau)\right>_J d\tau 
	 +\delta\int\left[ \omega(Y,JY)\cdot\left<\nu,\partial_s\right>_j -\omega(Y,-Y)\cdot\left<\nu,\partial_t\right>_j\right]  d\tau \nonumber\\
 	&=& \int \left<Y\circ\Gamma_i(\tau),N(\tau)\right>_J d\tau + \delta \cdot 1 \cdot \int \left<\nu,\partial_s\right>_j  d\tau,
 	\end{eqnarray}
 	 Let us estimate separately each term in the sum (\ref{eq:two_terms}), starting with the first:
 	Since $JY=R$, the vector field $Y$ is perpendicular to the hyperplane $T(\partial U)$ at each point and is pointing outwards of $U$. By our construction, $N(\tau)$ points inwards to $U$ (as it is tangent to $\im(u)$ and points out of $\im(u)\cap U^c$) and therefore $\left<Y\circ\Gamma_i,N\right> \leq 0$ for all $\tau$. We conclude that 
 	\begin{equation}\label{eq:int_normal_times_liouville}
 	\int \left<Y\circ\Gamma_i(\tau),N(\tau)\right>_J d\tau\leq 0.
 	\end{equation} 
 	We turn to estimate the second summand in (\ref{eq:two_terms}): Noticing that  $\left<\nu,\partial_s\right>_j = \left<j\nu,j\partial_s\right>_j= \left<\dot\gamma_i,\partial_t\right>_j = dt(\dot{\gamma_i})$, we have
 	\begin{equation*}
 	 \int \left<\nu,\partial_s\right>_j  d\tau =  \int dt(\dot\gamma_i)\ d\tau =\int_{\gamma_i} dt.
 	\end{equation*}  
 	Let $\hat \Sigma$ be the closure of $\Sigma$ in the compactification $(\R\cup\{\pm\infty\})\times S^1$ of the cylinder, then $\partial \hat \Sigma\subset \partial \Sigma \cup \{\pm\infty\}\times S^1$. Notice that $\partial \hat \Sigma$ contains $\{-\infty\}\times S^1$ (resp., $\{+\infty\}\times S^1$) if and only if $x_-\subset U^c$ (resp., $x_+\subset U^c$). As $\int_{\{\pm\infty\}\times S^1} dt = \pm 1$ and, by Stokes' theorem, $\int_{\partial \hat \Sigma} dt=0$, we conclude that
 	\begin{equation}\label{eq:homology_gamma_in_sigma}
 	\sum_i\int_{\gamma_i}dt = \int_{\gamma} dt = \int_{\partial \hat \Sigma}dt - \begin{cases}
 	1, & \text{if } x_-\subset U,\ x_+\subset U^c,\\
 	-1, & \text{if } x_-\subset U^c,\ x_+\subset U,\\
 	0, & \text{if } x_-,x_+\subset U \text{ or } x_-,x_+\subset U^c. 
 	\end{cases}
 	\end{equation}
 	Combining (\ref{eq:two_terms}), (\ref{eq:int_normal_times_liouville}) and (\ref{eq:homology_gamma_in_sigma}) we obtain
 	\begin{equation*}
 	\int_{\Gamma}\lambda= \sum_i\int_{\Gamma_i}\lambda \leq 0 + \delta\cdot\begin{cases}
 	-1, & \text{if } x_-\subset U,\ x_+\subset U^c,\\
 	1, & \text{if } x_-\subset U^c,\ x_+\subset U,\\
 	0, & \text{if } x_-,x_+\subset U \text{ or } x_-,x_+\subset U^c. 
 	\end{cases}
 	\end{equation*} 
 \end{proof}
 
When the homotopy $H$ is non-increasing in $U^c$, Lemma~\ref{lem:bdry_loop} can be used to bound the action of the ends of solutions that cross the boundary of $U$. Lemma~\ref{lem:cross_solutions} below is similar to a result obtained by Cieliebak and Oancea  in \cite[Lemma 2.2]{cieliebak2018symplectic} for the setting of completed Liouville domains, using {\it neck-stretching}. The proof of Lemma~\ref{lem:cross_solutions} uses a different approach and is an application of Lemma~\ref{lem:bdry_loop} above.
 
\begin{lemma}\label{lem:cross_solutions}
	Suppose that $(H,J)$ is $\delta$-cylindrical near $\partial U$ and assume in addition that $\partial_sH \leq 0$ on $U^c$. For every finite-energy solution $u$ connecting $x_\pm\in\cP(H_\pm)$,
	\begin{enumerate}
		\item if $x_-\subset U$ and $x_+\subset U^c$ then $\cA_{H_+}(x_+)< c-\delta$,
		\item if $x_-\subset U^c$ and $x_+\subset U$ then $\cA_{H_-}(x_-)> c-\delta$,
	\end{enumerate}
	where $c$ is the value of $H$ on $\partial U$.
\end{lemma}
 
\begin{figure}
 	\centering
 	\includegraphics[scale=0.6]{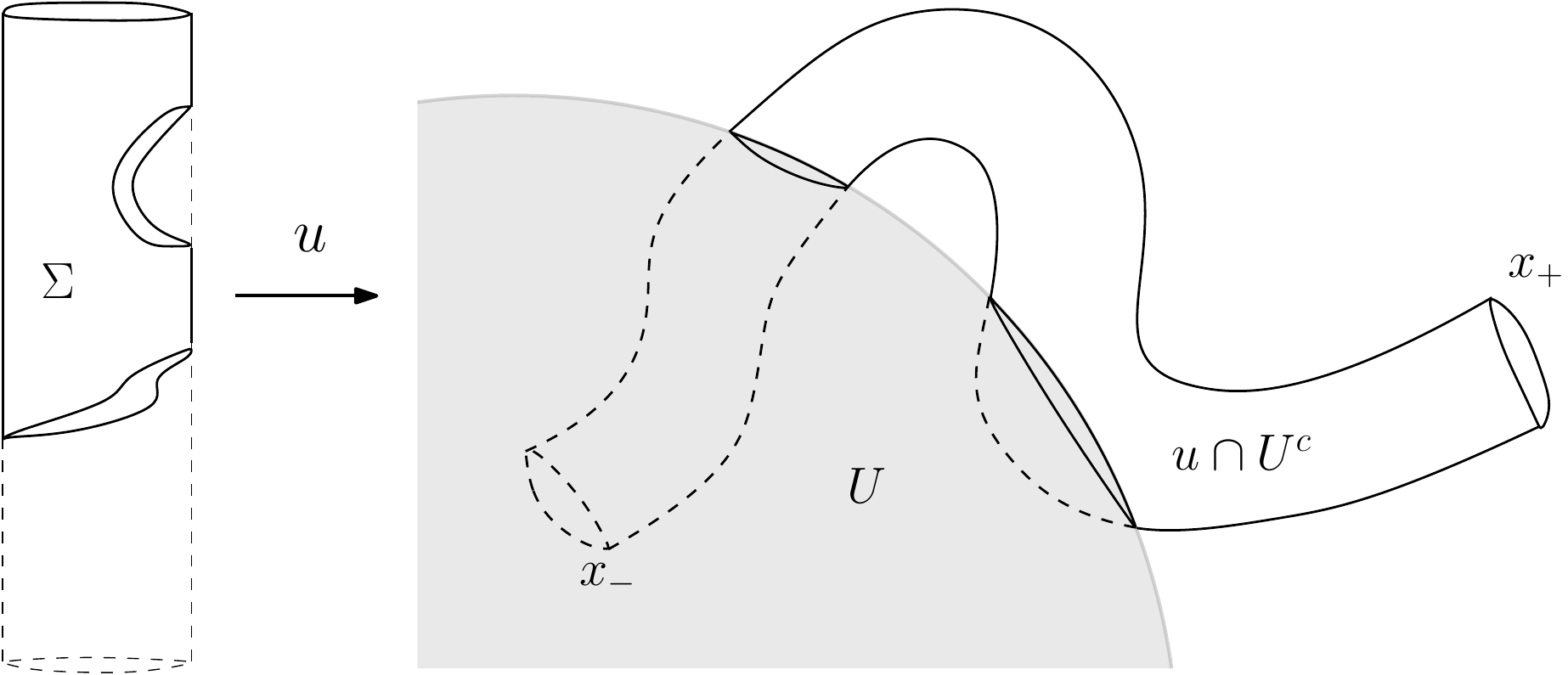}
 	\caption{\small{An example for the setting described in Lemma~\ref{lem:cross_solutions}. 
 			The gray region is the set $U\subset M$ and $\Sigma:=u^{-1}(U^c)\subset\R\times S^1$.
	 	}}
 	\label{fig:cross_diff}
 \end{figure}

\begin{proof}
	We prove the first statement, where $x_-\subset U$ and $x_+\subset U^c$. The second statement is proved similarly.
	As in \cite[Lemma 2.2]{cieliebak2018symplectic}, after replacing $U$ by its image, $\psi^{\tau} U$, under the Liouville flow for small time $\tau$, we may assume that $u$ is transverse to $\partial U$\footnote{The proof of this statement is similar to that of Thom's transversality theorem.}.
	Note that, since $\nabla_JH=\delta Y$ on a neighborhood of $\partial U$, $H$ is constant on $\partial (\psi^\tau U) = \psi^\tau(\partial U)$. Moreover, choosing the sign of $\tau$ to be opposite to the sign of $\delta$, the value of $H$ on ${\psi^\tau(\partial U)}$ is smaller than $c$ (in order to prove the second statement, choose $\tau$ to be of the same sign as $\delta$, and then the value of $H$ on ${\psi^\tau(\partial U)}$ will be greater than $c$).
	Denote $\Sigma:= u^{-1}(U^c)\subset \R\times S^1$ and let us compute an  energy-identity for the restriction $u|_\Sigma$:
	\begin{eqnarray}\label{eq:energy_id_sigma}
	\int_{u|_\Sigma} \omega &=& \int_{\Sigma} \omega(\partial_s u,\partial_t u)\ ds\wedge dt \overset{(\ref{eq:FE})}{=}  \int_{\Sigma} \omega(\partial_s u,J\partial_s u +X_H\circ u)\ ds\wedge dt
	\nonumber\\
	&=& \int_{\Sigma} \|\partial_s u\|_J^2 \ ds\wedge dt + \int_{\Sigma} dH(\partial_s u)\ ds\wedge dt
	\nonumber\\
	&=& E(u|_\Sigma)+ \int_{\Sigma} \frac{\partial}{\partial s} \left(H\circ u\right) \ ds\wedge d t- \int_{\Sigma} ({\partial_s}H)\circ u\ ds\wedge d t
	\nonumber\\
	&=& E(u|_\Sigma)+ \int_{\Sigma}d(\left(H\circ u\right) \ d t)- \int_{\Sigma} ({\partial_s}H)\circ u\ ds\wedge d t
		\nonumber\\
	&\geq&  E(u|_\Sigma)+ \int_{\Sigma} d\left(H\circ u\ dt\right)>\int_{\Sigma} d\left(H\circ u\ dt\right),
`	\end{eqnarray}
	where, in the last two inequalities, we used our assumption that $\partial_s H\leq 0$, and the positivity of the energy, respectively.
	As before, denote by $\hat \Sigma$ the closure of $\Sigma$ in the compactification $(\R\cup\{\pm\infty\})\times S^1$, then $\partial \hat\Sigma=\gamma\cup \{+\infty\}\times S^1$. Since $H$ is constant on 
	$\partial U$,
	$\int_{\gamma}H\circ u\ dt = H({\partial U})\cdot \int_{\gamma} dt=-H(\partial U)$, where the last equality follows from (\ref{eq:homology_gamma_in_sigma}) for $\gamma= \partial \Sigma$. Therefore, using Stokes' theorem, we obtain
	\begin{equation}\label{eq:H_integrals}
	\int_{\Sigma} d\left(H\circ u\ dt\right)=\int_{\partial \hat \Sigma} H\circ u\ dt= -H({\partial U}) +\int_0^1 H\circ x_+.
	\end{equation}
	Let $\bar x_\pm$ be capping disks of $x_\pm$ respectively, 
	and let $v\subset \bar U$ be a union of disks capping the connected components of $\Gamma:=u(\gamma)$, such that the contact form $\lambda$ is defined on $v$. The existence of such disks follows from our definition of a CIB domain: If the relevant connected component of $U$ is an incompressible Liouville domain, then we can take a capping disk that is contained in that component. Otherwise, the boundary of the relevant connected component of $U$ is incompressible and we can take the capping disk to lie in the boundary. Since $M$ is symplectically aspherical and $\omega = d\lambda$ where $\lambda$ is defined, we have 
	\begin{equation}\label{eq:area_comparisons}
	\int_{u|_\Sigma}\omega=\int_{\bar x_+}\omega +\int_v \omega = \int_{\bar x_+}\omega+\int_{\Gamma}\lambda.
	\end{equation}
	Combining (\ref{eq:H_integrals}) and (\ref{eq:area_comparisons}) yields
	\begin{equation*}
	\cA_{H_+}(x_+)=\int_0^1 H\circ x_+-\int_{\bar x_+}\omega = \int_{\Sigma} d\left(H\circ u\ dt\right)+H(\partial U) -\int_{u|_\Sigma}\omega + \int_{\Gamma}\lambda<c +\int_{\Gamma}\lambda,
	\end{equation*}
	where the last inequality is due to (\ref{eq:energy_id_sigma}). Using Lemma~\ref{lem:bdry_loop} we conclude that $\cA_{H_+}(x_+) <c-\delta$. 
\end{proof}

The following Lemma is essentially a version of \cite[Lemma 2.2]{cieliebak2018symplectic} for closed symplectically aspherical manifolds instead of completed Liouville domains.

\begin{lemma}\label{lem:solution_contained_U}
	Suppose that $(H,J)$ is $\delta$-cylindrical near $\partial U$ and that $\partial_s H\leq 0$ on $U^c$. Then, for every $x_\pm\in\cP(H_\pm)$ that are contained in $U$, every solution $u$ connecting them is contained in $U$.
\end{lemma}
\begin{proof}
	As before, after replacing $U$ by its image, $\psi^{\tau} U$, under the Liouville flow for a small time $\tau$, we may assume that $u$ is transverse to $\partial U$. Setting again $\Sigma:= u^{-1}(U^c)\subset \R\times S^1$ and computing an energy identity, as in (\ref{eq:energy_id_sigma}), for the restriction  of $u$ to $\Sigma$, we have
	\begin{equation*}
		\int_{u|_\Sigma} \omega \geq E(u|_\Sigma)+ \int_{\partial\hat \Sigma} H\circ u\ dt,
	\end{equation*}
	where, as before, $\hat{\Sigma}$ is the closure of $\Sigma$ in the compactification of the cylinder. This time, both ends $x_\pm$ are contained in $U$ and hence $\partial \hat \Sigma = \gamma$. Since $H$ is constant on $\partial U$, it follows from (\ref{eq:homology_gamma_in_sigma}) that
	\begin{equation*}
	\int_{\partial\hat \Sigma} H\circ u\ dt = \int_{\gamma} H\circ u\ dt = H(\partial U)\cdot \int_{\gamma} dt =0.
	\end{equation*}
	On the other hand, taking $v\subset\bar U$ to be a union of disks capping the connected components of $\Gamma=u(\gamma)$ (which is oriented as the boundary of $\im(u)\cap U^c$), such that $\lambda$ is defined on $v$, the fact that $M$ is symplectically aspherical implies that
	\begin{equation*}
	\int_{u|_\Sigma}\omega = \int_v \omega = \int_\Gamma\lambda\leq 0,
	\end{equation*}
	where the last inequality follows from Lemma~\ref{lem:bdry_loop}.
	Combining the above two inequalities we find 
	\begin{equation*}
	E(u|_\Sigma)\leq \int_{u|_\Sigma}\omega\leq 0. 
	\end{equation*}
	Since we assumed that $H_\pm$ are non-degenerate and have no 1-periodic orbits intersecting $\partial U$, this implies $\im(u)\cap int(U^c)=\emptyset$ and hence $\im(u)\subset \bar U$. Noticing that we may argue similarly for the image $\psi^{\tau}U$ of $U$ under the Liouville flow for small negative time, $\tau<0$, we conclude that $\im(u)\subset \overline{\psi^{\tau} U}\subset U$.
\end{proof}

\subsection{Constructing the barricade.}
As before, $U$ denotes a CIB domain and $\psi^{\tau}$  is the flow of the Liouville vector field $Y$, which is defined in a neighborhood of the boundary $\partial U$. Consider a pair $(H,J)$ of a homotopy (or, in particular, a Hamiltonian) and an almost complex structure. The following definition is an adaptation of Figure~\ref{fig:Morse_bump} to Floer theory.
\begin{defin}\label{def:bump}
	We say that the pair $(H,J)$ admits a {\it cylindrical bump of width $\tau>0$ and slope $\delta>0$ around $\partial U$} (abbreviate to $(\tau,\delta)$-bump around $\partial U$) if
\begin{enumerate}
	\item $H=0$ on $\partial U\times S^1\times \R$ and on $\partial U_\circ\times S^1\times\R$, where $U_\circ:=\psi^{-\tau}U$.
	\item \label{itm:J_bicylindrical}
	$J$ is cylindrical near $\partial U$ and $\partial U_\circ$, namely, $JY=R$ on an open neighborhood of $\partial U\cup \partial U_\circ$. 
	\item $\nabla_J H=\delta Y$ near $\partial U_\circ\times S^1\times\R$ and $\nabla_J H=-\delta Y$ near $\partial U\times S^1\times\R$.
	\item The only 1-periodic orbits of $H_\pm$ that are not contained in $U_\circ$ are critical points with values in $(-\delta,\delta)$.
\end{enumerate}
\end{defin}
\begin{figure}
	\centering
	\includegraphics[scale=1.2]{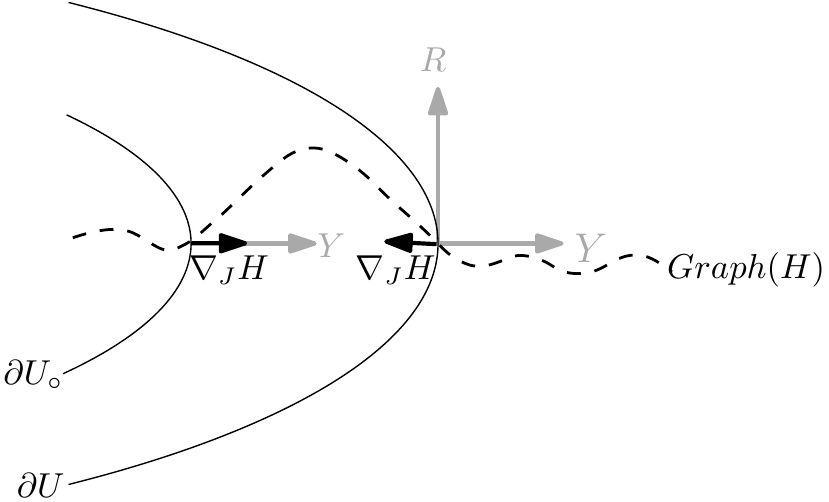}
	\caption{\small{An illustration of a pair with a cylindrical bump.
		}}
		\label{fig:cylindrical_bump}
\end{figure}
\noindent In analogy with the discussion in Morse theory, we show that a pair with a cylindrical bump has a barricade. 
\begin{prop}\label{pro:bump_implies_barricade}
	Let $(H,J)$ be a pair with a cylindrical bump of width $\tau$ and slope $\delta$. Then, the pair $(H,J)$ has a barricade in $U$ around $U_\circ:=\psi^{-\tau}U$.
\end{prop}
\begin{proof}
	The proof essentially follows from Lemmas \ref{lem:cross_solutions} and \ref{lem:solution_contained_U}, together with the fact that a pair $(H,J)$ with a $(\tau,\delta)$-bump around $\partial U$ is in particular cylindrical near both $\partial U$ and $\partial U_\circ$. 
	Let $u$ be a solution of the $s$-dependent Floer equation, with respect to $H$ and $J$, that connects $x_\pm\in\cP(H_\pm)$. We need to show that $u$ satisfies the constraints from Definition~\ref{def:barricade}, and therefore split into two cases:
	\begin{enumerate}
		\item \label{itm:start_in_U0}
		Suppose $x_-\subset U_\circ$. If $x_+\subset U_\circ$ we may apply Lemma~\ref{lem:solution_contained_U} to $H$, $J$ and $U_\circ$, and conclude that $\im(u)\subset U_\circ$ as required. Otherwise,
		$x_+\subset U_\circ^c$ is a critical point of $H_+$ and its value lies in the interval $(-\delta, \delta)$. On the other hand, applying  Lemma~\ref{lem:cross_solutions} to $H$, $J$ and $U_\circ$ yields that $\cA_{H_+}(x_+)<-\delta$, in contradiction. 
		
		\item \label{itm:end_in_U}
		Suppose $x_+\subset U$. As before, if $x_-\subset U$ then applying Lemma~\ref{lem:solution_contained_U} to $H$, $J$ and $U$ yields $u\subset  U$ as required. Otherwise, $x_-\subset U^c$ is a critical point of $H_-$ and its value lies in $(-\delta, \delta)$. On the other hand, applying Lemma~\ref{lem:cross_solutions} to $H$, $J$ and $U$, and noticing that $\nabla_JH=-\delta Y$ on $\partial U$, we find that $\cA_{H_-}(x_-)>\delta$, in contradiction.
	\end{enumerate}	
\end{proof}
		 In order to prove Theorem~\ref{thm:barricade}, it remains to guarantee the regularity assertion, for which we use the result from Section~\ref{app:barricade_survives_perturbation} below.
\begin{proof}[Proof of Theorem~\ref{thm:barricade}]
	Let $H$ be a homotopy of Hamiltonians that is supported in $U\times  S^1\times \R$. Then, there exists $\tau>0$ small enough, such that $H$ is supported inside $\psi^{-\tau}U=:U_\circ$. Fix an almost complex structure $J$ that is cylindrical near both $\partial U$ and $\partial U_\circ$ (see Item~\ref{itm:J_bicylindrical} of Definition~\ref{def:bump} above), and let $h$ be a $\cC^\infty$-small perturbation of $H$ such that the pair $(h,J)$ admits a $(\tau,\delta)$-bump around $\partial U$ and $h_\pm$ are non-degenerate. Notice that, by definition, the pairs $(h_\pm,J)$ also admit a $(\tau,\delta)$-bump around $\partial U$. By Proposition~\ref{pro:bump_implies_barricade}, the pairs $(h,J)$ and $(h_\pm, J)$ have a barricade in $U$ around $U_\circ$. 
	
	The pairs $(h,J)$, $(h_\pm,J)$ constructed above are not necessarily Floer-regular. In order to achieve regularity, we perturb the homotopy $h$ and its ends. Proposition~\ref{pro:appendix} below states that for a homotopy $h'$ that   satisfies $\cP(h'_\pm)=\cP(h_\pm)$ and $supp (\partial_s h')\subset {M\times S^1\times I}$ for some fixed finite interval $I$, if $h'$ is close enough to $h$, then $(h',J)$ also has a barricade in $U$ around $U_\circ$. Therefore, it remains to describe a perturbation that satisfies the above constraints, and ensures regularity.
	Starting with the ends and recalling that $h_\pm$ are non-degenerate, we  perturb them without changing their periodic orbits to guarantee that the pairs $(h_\pm, J)$ are Floer-regular (the fact that this is possible is a well known result from Floer theory, cited in Claim~\ref{clm:app_reg_Hamiltonians} below). If the homotopy $h$ is constant, that is, corresponds to a single Hamiltonian, we are done. Otherwise, let us perturb $h$ so that its ends will agree with the regular perturbations of $h_\pm$. Finally, we perturb the resulting homotopy on the set $M\times S^1\times I$, for some fixed finite interval $I$, to make the pair $(h,J)$ Floer-regular. This is possible due  Proposition~\ref{pro:pert_hom_bdd_supp} below, which is a slight modification of standard claims from Floer theory and is proved in Section~\ref{app:regular_pair}. 
\end{proof}

\begin{rem}\label{rem:bump_implies_barricade}
	Proposition~\ref{pro:bump_implies_barricade} suggests that, when given a homotopy (or a Hamiltonian) $H$ that is supported in $U\times S^1\times \R$, we have some freedom in choosing the pair $(h,J)$ from Theorem~\ref{thm:barricade}. 
	Let us mention some additional properties that can be granted for the perturbation $h$ and the almost complex structure $J$, and will be useful in applications. 
	\begin{enumerate}
		\item \label{itm:time-indep}
		The almost complex structure $J$ can be taken to be time-independent. Moreover, if one of the ends of $H$, say $H_-$, is zero, then $h$ can be chosen such that $h_-$ is any time-independent small Morse function that has a cylindrical bump around $\partial U$. To see this, choose $h\approx H$ and $J$ such that $(h,J)$ has a cylindrical bump around $\partial U$, and $J$, $h_-$ are time-independent. Then, the pair $(h_-,J)$ is Floer-regular and, by perturbing $h_+$ first and then replacing the homotopy by a compactly supported perturbation, we end up with a pair $(h,J)$ that is Floer regular, as well as its ends, and $(h_-,J)$ is time-independent.
		
		\item \label{itm:ends_agree_on_P} When the homotopy $H$ is constant on some domain, we can choose the perturbation $h$ such that, on this domain, its ends, $h_\pm$, agree on their 1-periodic orbits up to second order. This follows from the use of Claim~\ref{clm:app_reg_Hamiltonians} in the proof of Theorem~\ref{thm:barricade}.
		
		\item \label{itm:hom_const_outside_of_01} Given an interval $[a,b]\subset \R$ such that $H$ is a constant homotopy for $s\notin[a,b]$, we can chose the perturbation $h$ of $H$ to be also constant outside of $[a,b]$, namely $supp(\partial_s h)\subset M\times S^1\times [a,b]$. This follows from the use of Proposition~\ref{pro:pert_hom_bdd_supp} in the proof of Theorem~\ref{thm:barricade}.
		
		\item Proposition~\ref{pro:bump_implies_barricade} also holds when considering a homotopy of almost complex structures, $\{J_s\}_{s\in\R}$, but the demand on $(h,J)$ to have a $(\tau,\delta)$-bump around $\partial U$ limits the dependence of $J_s$ on $s$ there.
	\end{enumerate}	
\end{rem}

\section{Locality of spectral invariants, Schwarz's capacities and super heavy sets.} \label{sec:shwartzs_cap}
In this section we use barricades to prove Theorem~\ref{thm:indep_of_embedding} and derive Corollaries~\ref{cor:schwarts_capacity} and \ref{cor:heavy_sets}. We will use the definitions and notations from Section~\ref{sec:floer_preliminaries}, in particular Notations~\ref{not:CF_element_contained} and Formula (\ref{eq:continuation_def}). We will also use the following properties of spectral invariants (see \cite[Proposition 12.5.3]{polterovich2014function}, for example):
\begin{enumerate}
	\item{(spectrality)} $c(F;\alpha)\in \spec(F)$.
	\item(stability/continuity) For any Hamiltonians $F$, $G$ and homology class $\alpha\in H_*(M)$,
	\begin{equation*}
	\int_0^1 \min_{x\in M} (F(x,t)-G(x,t))\ dt \leq c(F;\alpha)-c(G;\alpha)\leq \int_0^1 \max_{x\in M} (F(x,t)-G(x,t))\ dt. 
	\end{equation*}
	In particular, the functional $c(\cdot;\alpha):\cC^\infty(M\times S^1)\rightarrow\R$ is continuous.
	\item(Poincar\'e duality) For any Hamiltonian $F$,
	$
	c(F;[M]) = -c(-F;[pt]).
	$ 
	\item(energy-capacity inequality) If the support of $F$ is displaceable, then $c(F)$ is bounded by the displacement energy of the support in $M$, namely $c(F)\leq e(supp(F);M)$.
	We remind that a subset $X$ of a symplectic manifold is displaceable if there exists a Hamiltonian $G$ such that $\varphi_G^1(X)\cap X=\emptyset$. In this case, the displacement energy of $X$ is given by 
	\begin{equation}\label{eq:disp_energy_def}
	e(X;M):=\inf_{G: \varphi_G^1(X)\cap X=\emptyset} \int_0^1 \left(\max_M G(\cdot,t)-\min_M G(\cdot,t)\right)\ dt. 
	\end{equation}
\end{enumerate}
Let us sketch the idea of the proof of Theorem~\ref{thm:indep_of_embedding} before giving the details. We will prove the statement for the class of a point, and use Poincar\'e duality to deduce the same for the fundamental class. We start by showing that the spectral invariant, with respect to $[pt]$, of a Hamiltonian supported in a CIB domain is non-positive (Lemma~\ref{lem:spec_inv_non_neg}). Then, after properly choosing regular perturbations with barricades (Lemma~\ref{lem:reg_pairs_for_locality}), we consider a representative of $[pt]$ of negative action on $M$. Such a representative must be a combination of orbits in $U_\circ$ and thus can be pushed to a cycle on $N$. Finally, we use continuation maps, induced by homotopies to small Morse functions, to conclude that the cycle on $N$ represents $[pt]$ there. 

As mentioned above, our first step towards proving Theorem~\ref{thm:indep_of_embedding} is showing that the spectral invariant with respect to $[pt]$ of a Hamiltonian supported in a CIB domain is always non-positive. 
\begin{lemma}\label{lem:spec_inv_non_neg}
	Let $F:M\times S^1\rightarrow\R$ be a Hamiltonian supported in a CIB domain $U$. Then $c(F;[pt])\leq 0$. 
\end{lemma}
\begin{proof}
	Let $H$ be a linear homotopy\footnote{A {\it linear homotopy} is a homotopy of the form $H(x,t,s) = H_-(x,t)+\beta(s)(H_+(x,t)-H_-(x,t))$, where $\beta:\R\rightarrow\R$ is a smooth step function.} from $H_-:= 0$ to $H_+:=F$. By Theorem~\ref{thm:barricade}, there exist a small perturbation $h$ of $H$ and an almost complex structure $J$ such that $(h,J)$ and $(h_\pm, J)$ are Floer-regular and have a barricade in $U$ around $U_\circ$, where $U_\circ $ contains the support of $F$. By Remark~\ref{rem:bump_implies_barricade}, item~\ref{itm:time-indep}, we can choose $J$ to be time independent and $h$ such that $h_-$ is a time-independent small Morse function. Moreover, we may assume that $h_-$ has a minimum point $p$ that is contained in $U^c$.  Since the Floer complex and differential of $(h_-,J)$ agree with the Morse ones, the point $p$ represents $[pt]$ in $CF_*(h_-)\cong CM_{*+n}(h_-)$. Denoting by $\Phi_{(h,J)}:CF_*(h_-)\rightarrow CF_*(h_+)$ the continuation map associated to the pair $(h,J)$, the presence of the barricade guarantees that $\Phi_{(h,J)}(p)\subset C_{U^c}(h_+)$. Indeed, otherwise, we would have a continuation solution starting at $p\subset U^c$ and ending at some $x_+\subset U$, in contradiction. The image $\Phi_{(h,J)} (p)$ is a cycle representing $[pt]$ in $CF_*(h_+)$ and its action level is close to zero. Indeed, since $h_+$ approximates $F$, which is supported in $U_\circ$, the restriction $h_+|_{U_\circ^c}$ is a small Morse function. Its 1-periodic orbits there are critical points and their actions are the critical values. Therefore, using the stability property of spectral invariants, we conclude that $c(F;[pt])\leq c(h_+;[pt]) +\delta \leq \lambda_{h_+}(\Phi_{(h,J)}(p)) +\delta \leq 2\delta$, for small $\delta>0$. 
\end{proof}

\begin{rem}
	\begin{itemize}
		\item Using Poincar\'e duality for spectral invariants, the above lemma implies that $c(F;[M])\geq 0$ for every Hamiltonian $F$ supported in a CIB domain. This is already known for incompressible Liouville domains. Indeed, it follows easily from the max formula, proved in \cite{humiliere2016towards}, when applied to the functions $F_1=F$ and $F_2=0$:
		$$
		c(F+0;[M]) = \max \{c(F;[M]),c(0;[M])\}\geq 0.
		$$
		
		\item Lemma~\ref{lem:spec_inv_non_neg} does not hold if $M$ is not symplectically aspherical. For example, the equator in $S^2$ is known to be super-heavy. Therefore, if $F$ is a Hamiltonian on $S^2$ which is supported on a disk containing the equator, then
		$$
		\zeta(F)=\lim_{k\rightarrow\infty} c(kF;[M])/k 
		$$ 
		is not greater than the maximal value that $F$ attains on the equator, see \cite[Chapter 6]{polterovich2014function}. Therefore, one can construct a Hamiltonian supported in a disk on $S^2$ with a negative spectral invariant with respect to the fundamental class. 
	\end{itemize}
\end{rem}

Our next step towards the proof of Theorem~\ref{thm:indep_of_embedding} is choosing suitable perturbations for the Hamiltonians $F$ and $\Psi_* F$, as well as homotopies from them to small Morse functions. 
Before that, we use the embedding $\Psi$ to define a linear map between subspaces of Floer complexes of Hamiltonians on $M$ and on $N$, that agree on $U$ through $\Psi$.
\begin{defin}
	Consider non-degenerate Hamiltonians $f_M$ on $M$ and $f_N$ on $N$, such that $f_M$ and $f_N\circ \Psi$ have the same 1-periodic orbits in $U$. For an element $a\in C_U(f_M)\subset CF_*(f_M)$ that is a combination of orbits contained in $U$, we define its pushforward with respect to the embedding $\Psi$ to be 
	\begin{equation*}
	\Psi_*a:= \sum_{x\in a} a_x\cdot \Psi(x)\in C_{\Psi(U)}(f_N)\subset CF_*(f_N).
	\end{equation*}
\end{defin}
\begin{lemma}[Set-up]\label{lem:reg_pairs_for_locality}
	There exist homotopies and time-independent almost complex structures $h_M$ and $J_M$  on $M$, and $h_N$ and $J_N$ on $N$, such that the following hold:
	\begin{enumerate}
		\item The pairs $(h_M, J_M)$, $(h_{M\pm}, J_M)$, $(h_N, J_N)$ and $(h_{N\pm}, J_N)$ are all Floer-regular and have barricades in $U$ around $U_\circ$ and in $\Psi(U)$ around $\Psi(U_\circ)$, respectively, for some $U_\circ\Subset U$ containing the support of $F$.
		\item $h_{M-}$ and $h_{N-}$ are small perturbations of $F$ and $\Psi_*F$ respectively, and $h_{M+}$, $h_{N+}$ are small time-independent Morse functions.
		\item \label{itm:agree_on_emb} On $\Psi(U)$: the Hamiltonians $h_{N-}$ and $h_{M-}\circ\Psi^{-1}$ agree on their periodic orbits up to second order and $J_N=J_M\circ\Psi^{-1}$.
		\item The differentials and continuation maps commute with the pushforwad map $\Psi_*$ when restricted to $U_\circ$:
		\begin{equation}\label{eq:continuation_from_U-}
		\Phi_{(h_N,J_N)}\circ\Psi_*\circ \pi_{U_\circ} 
		=\Psi_*\circ \Phi_{(h_M,J_M)}\circ \pi_{U_\circ}, 
		\end{equation}	
		and 
		\begin{equation}\label{eq:diffetential_in_U-}
		\partial_{(h_{N\pm},J_N)}\circ\Psi_*\circ \pi_{U_\circ} 
		=\Psi_*\circ \partial_{(h_{M\pm},J_M)}\circ \pi_{U_\circ} 
		\end{equation}
	\end{enumerate}
\end{lemma}

We postpone the proof of Lemma~\ref{lem:reg_pairs_for_locality}, and prove Theorem~\ref{thm:indep_of_embedding} first.
\begin{proof}[Proof of Theorem~\ref{thm:indep_of_embedding}]
	We will prove that $c_M(F;[pt]) = c_N(\Psi_*F;[pt])$, and the claim for the fundamental class will follow from Poincar\'e duality for spectral invariants. 
	Suppose that at least one of $c_M(F;[pt])$, $c_N(\Psi_*F;[pt])$ is non-zero, otherwise there is nothing to prove. Without loss of generality, assume that $c_M(F;[pt])\neq 0$, then, by Lemma~\ref{lem:spec_inv_non_neg}, $c_M(F;[pt])<0$. We will show that $c_M(F;[pt])\geq c_N(\Psi_*F;[pt])$. This will imply that $c_N(\Psi_*F;[pt])< 0$ and equality will follow by symmetry. 
	Let $(h_M, J_M)$ and $(h_N, J_N)$ be pairs of homotopies and almost complex structures on $M$ and $N$ respectively, that satisfy the assertions of Lemma~\ref{lem:reg_pairs_for_locality}, and denote $f_M:= h_{M-}$, $f_N:=h_{N-}$. By the continuity of spectral invariants, it is enough to prove the claim for $f_M$ and $f_N$.
	
	Since $c_M(F;[pt])<0$ and $F|_{U_\circ^c}=0$, by taking $f_M$ to be close enough to $F$ and $F|_{U_\circ^c}=0$, we may assume that $c_M(f_M;[pt])<\min_{U_\circ^c} f_M$. Recalling that $f_M$ is a small Morse function on $U_\circ^c$, its 1-periodic orbits there are its critical points, and their actions are the critical values. 
	As a consequence, a representative $a\in CF_*(f_M)$ of $[pt]$ of action level $\lambda_{f_M}(a) = c_M(f_M;[pt])$ is a combination of orbits that are contained in $U_\circ$, namely $a\in C_{U_\circ}(f_M)$. Therefore, the pushforward $\Psi_*a\in CF_*(f_N)$ is defined, and by (\ref{eq:diffetential_in_U-}), $\Psi_* a$ is closed in $CF_*(f_N)$. To see that $\Psi_*a$ represents the class of a point, we will use (\ref{eq:continuation_from_U-}). Indeed, since $a$ represents $[pt]$ on $M$, and continuation maps induce isomorphism on homologies, $\Phi_{(h_M,J_M)}(a)$ is a representative of $[pt]$ in $CF_*(h_{M+})$. Since $h_{M+}$ is a small time-independent Morse function (and $J_M$ is time-independent), its Floer complex and differential coincide with the Morse ones, $(CF_*(h_{M+}),\partial_{(h_{M+},J_M)}) \cong (CM_{*+n}(h_{M+}),\partial_{(h_{M+},g_{J_M})}^{Morse})$. As a consequence, $\Phi_{(h_M,J_M)}(a)$ is a sum of an odd number of minima\footnote{See, for example, the proof of Proposition 4.5.1 in \cite{audin-damian}.}. Using (\ref{eq:continuation_from_U-}), we find that $\Phi_{(h_N,J_N)}(\Psi_*a) = \Psi_* (\Phi_{(h_M,J_M)} a)$ is also a sum of an odd number of minima, and as such, represents the point class in $CM_{*+n}(h_{N+}) \cong CF_*(h_{N+})$. Since $\Psi_*a$ is closed, we conclude that it represents $[pt]$ in $CF_*(f_N)$. Together with the fact that, in $\Psi(U)$, $f_M\circ \Psi^{-1}$ and $f_N$ agree on their 1-periodic orbits, this implies that 
	$$
	c_N(f_N;[pt])\leq \lambda_{f_N}(\Psi_* a) = \lambda_{f_M}(a) = c_M(f_M;[pt]),
	$$
	where the equality $ \lambda_{f_N}(\Psi_* a) = \lambda_{f_M}(a)$ follows from the fact that $U$ is incompressible, see Remark~\ref{rem:CIB_domains} and Proposition~\ref{pro:incompressible}.
\end{proof}

\begin{proof}[Proof of Lemma~\ref{lem:reg_pairs_for_locality}]
	Let $H_M:M\times S^1\times \R \rightarrow\R$ be a linear homotopy from $F$ to zero, that is constant outside of $[0,1]$, i.e., $\partial_s H_M|_{s\notin[0,1]}=0$. Then, $H_M$ is supported in $U$, and its pushforward $H_N:=\Psi_* H_M$ is a linear homotopy from $\Psi_*F$ to zero on $N$. Let $J_M$ be a time-independent almost complex structure on $M$ and let $h_M^\flat$ be a homotopy with non-degenerate ends, that is constant outside of $[0,1]$ and approximates $H_M$, such that the pair $(h_M^\flat, J_M)$ has a $(\tau,\delta)$-bump around $\partial U$ for some $\tau$ and $\delta$, and set $U_\circ=\psi^{-\tau}U$. 
	Let $J_N$ be a time-independent almost complex structure obtained as an extension of $J_M\circ \Psi^{-1}$ from $\Psi(U)$ to $N$\footnote{The fact that $J_M\circ \Psi^{-1}$ can be extended to an almost complex structure on $N$ can be deduced from the path connectivity of the set of almost complex structures on symplectic vector bundles (see, e.g., \cite[Proposition 2.63]{mcduff2012j}), together with the fact that $\partial U$ has a tubular neighborhood.}. Extending $h_M^\flat\circ\Psi^{-1}$ to $N$ by a homotopy of small Morse functions with eigenvalues in $(-\delta, \delta)$, we obtain a pair $(h_N^\flat, J_N)$ with a $(\tau,\delta)$-bump around $\Psi(\partial U) = \partial \Psi (U)$. Moreover, $h_N^\flat$ is a homotopy with non-degenerate ends, it approximates $H_N$, and we can choose it to be constant for $s\notin [0,1]$. 
	Noticing that the ends of these homotopies have $(\tau,\delta)$-bumps as well, Proposition~\ref{pro:bump_implies_barricade} guarantees that the pairs $(h_M^\flat, J_M)$, $(h_{M\pm}^\flat, J_M)$, $(h_N^\flat, J_N)$ and $(h_{N\pm}^\flat, J_N)$ all have barricades in $U$ around $U_\circ$ and in $\Psi(U)$ around $\Psi(U_\circ)$, respectively.
	
	Let us now perturb $h_M^\flat$ to make all of the pairs defined on $M$ regular. As in the proof of Theorem~\ref{thm:barricade}, we first perturb the ends $h_{M\pm}^\flat$ into $h_{M\pm}$, without changing their periodic orbits, so that the pairs $(h_{M\pm},J_M)$ are Floer-regular (as cited in  Claim~\ref{clm:app_reg_Hamiltonians} below). Then, perturb the homotopy $h_M^\flat$ to obtain a homotopy, $h_M$, whose ends are the regular perturbations, $h_{M\pm}$, and that is constant for $s\notin [0,1]$. Finally, Proposition~\ref{pro:pert_hom_bdd_supp} below states that we can perturb the homotopy $h_M$ on the set $M\times S^1\times[0,1]$ to make the pair $(h_M,J_M)$ Floer-regular. We stress that after the perturbations the regular homotopy $h_M$ is constant for $s\notin[0,1]$ as well.  Proposition~\ref{pro:appendix} guarantees that every perturbation of $h_M^\flat$ that is constant outside of $[0,1]$ and whose ends have the same periodic orbits as the ends of $h_M^\flat$, also has a barricade in $U$ around $U_\circ$, when paired with $J_M$. Arguing similarly for the ends $h_{M\pm}$ we conclude that the pairs $(h_M, J_M)$ and $(h_{M\pm},J_M)$ all have barricades in $U$ around $U_\circ$. 
	
	We turn to construct the pairs on $N$. Let $h_N'$ be an extension to $N$ of the homotopy $h_M\circ \Psi^{-1}$, which is defined on $\Psi(U)$. Notice that by replacing $h_M$ with a smaller perturbation of $h_M^\flat$ if necessary, $h_N'$ can be taken to be arbitrarily close to $h_N^\flat$. This way, we can use Proposition~\ref{pro:appendix} again to conclude that $(h_N',J_N)$ has a barricade in $\Psi(U)$ around $\Psi(U_\circ)$. Finally, we repeat the arguments made above and perturb $h_N'$ to make all of the pairs on $N$ Floer-regular. We obtain a homotopy $h_N$ that is constant for $s\notin[0,1]$, approximates $h_M\circ \Psi^{-1}$ on $\Psi(U)$ and such that the pairs $(h_N, J_N)$ and $(h_{N\pm},J_N)$ are all Floer-regular and have barricades in $\Psi(U)$ around $\Psi(U_\circ)$.
	
	It remains to prove that, in $U_\circ$, the pushforward map commutes with the  continuation maps and the differentials for the homotopies $h_M$, $h_N$ and their ends, respectively. We will write the proof for the continuations maps, the proof for the differentials is analogous. We first show that the continuation maps of $h_N$ and $h_N'$ agree on $\Psi(U_\circ)$, and then prove that the commutation relation (\ref{eq:continuation_from_U-}) holds for $h_M$ and $h_N'$, which agree on $U$ through $\Psi$. 
	Proposition~\ref{pro:app_pert_hom_regular_on_U} (for the differentials, Proposition~\ref{pro:app_pert_Ham_regular_on_U}) states that the restriction of the continuation map to $C_{\Psi(U_\circ)}$ does  not change under small perturbations, when the pairs have a barricade and satisfy a certain regularity assumption on $\Psi(U)$. This assumption holds for Floer-regular pairs, as well as for pairs that coincide on $U$ with a Floer-regular pair. Therefore, recalling that $h_N$ is a small perturbation of $h_N'$, and that the pair $(h_N',J_N)$ agrees, on $\Psi(U)$, through a symplectomorphism, with the Floer-regular pair $(h_M, J_M)$, we may apply Proposition~\ref{pro:app_pert_hom_regular_on_U} and conclude that $\Phi_{(h_N,J_N)} \circ \pi_{\Psi(U_\circ)} = \Phi_{(h_N',J_N)}\circ\pi_{\Psi(U_\circ)}$.
	In order to prove $\Phi_{(h_N',J_N)}\circ\Psi_*\circ \pi_{U_\circ} =\Psi_*\circ \Phi_{(h_M,J_M)}\circ \pi_{U_\circ}$, recall the definitions of  $\Psi_*$ and the continuation maps (\ref{eq:continuation_def}). We need to show that for every $x_\pm\in\cP(h_{M\pm})$ such that $x_-\subset U_\circ$, it holds that $\#_2 \cM_{(h_M,J_M)}(x_-,x_+) = \#_2 \cM_{(h_N',J_N)}(\Psi(x_-), \Psi(x_+))$. This essentially follows from the fact that both pairs $(h_M,J_M)$ and $(h_N',J_N)$ have barricades, and that $h_M= h_N'\circ \Psi$ and $J_M= J_N\circ \Psi$ on $U$. Indeed, it follows from $x_-\subset U_\circ$ that $\Psi(x_-)\subset\Psi(U_\circ)$ and thus the barricades guarantee that all of the elements of $\cM_{(h_M,J_M)}(x_-,x_+)$ and $\cM_{(h_N',J_N)}(\Psi(x_-), \Psi(x_+))$ are contained in $U_\circ$ and $\Psi(U_\circ)$ respectively. The symplectic embedding $\Psi$ induces a bijection between these two sets, and so it follows that the counts of their elements coincide.
\end{proof}

Having established Theorem~\ref{thm:indep_of_embedding}, we now explain how to derive Corollaries~\ref{cor:schwarts_capacity}, \ref{cor:heavy_sets}. Let us start by recalling the definition of a symplectic capacity:
\begin{defin}[See, for example, \cite{cieliebak2005quantitative,hofer2012symplectic}]
Given a class $\cS$ of symplectic manifolds, a symplectic capacity on $\cS$ is a map $c:\cS\rightarrow [0,\infty]$ that satisfies the following properties:
\begin{itemize}
	\item (Monotonicity) $c(U,\omega)\leq c(V,\Omega)$ if there exists a symplectic embedding $(U,\omega)\hookrightarrow (V,\Omega)$.
	\item (Conformality) $c(U,\tau\omega)=|\tau|\cdot c(U,\omega)$ for all $\tau\in \R\setminus\{0\}$. 
	\item (Nontriviality) $c(B^{2n}(1),\omega_0) >0$ and $c(Z^{2n}(1),\omega_0)<\infty$, where $B^{2n}(1)\subset \R^{2n}$ is the unit ball and $Z^{2n}(1) = B^2(1)\times \R^{2n-2}$ is the symplectic cylinder\footnote{This is under the assumption that $c$ is defined for the cylinder.}.
\end{itemize}
Let us use Theorem~\ref{thm:indep_of_embedding} to show that Schwarz's relative capacities, which are defined for subsets of a given closed symplectically aspherical manifold, induce a capacity on the class of contractible compact symplectic manifolds with contact-type boundaries that can be embedded into symplectically aspherical manifolds. 
\end{defin}
\begin{proof}[Proof of Corollary~\ref{cor:schwarts_capacity}]
	Let $A\in\cS$ be a contractible symplectic manifold with a contact-type boundary that can be embedded into a symplectically aspherical manifold $(M,\omega)$. Abusing the notations, we write $A\subset M$. Recalling the definition of Schwarz's relative capacity (\ref{eq:def_spec_capacity}),
	\begin{equation*}
	c_\gamma(A;M):= \sup\left\{c(F;[M])-c(F;[pt])\ : 
	supp X_F\subset A\times S^1	\right\},
	\end{equation*}
	we consider a Hamiltonian $F$ on $M$ such that $X_F$ is supported in $A\times S^1$. Since $A$ is contractible, its boundary connected and therefore  $F$ is constant on $\partial A$, as well as on the complement, $M\setminus A$. Denoting $C:=F|_{M\setminus A}$, the difference $F-C$ is supported in $A$. Moreover, it follows from the spectrality and stability of spectral invariants that $c_M(F-C;\alpha) = c_M(F;\alpha)-C$ for every homology class $\alpha\in H_*(M)$. In particular, $c_M(F-C;[M])-c_M(F-C;[pt])=c_M(F;[M])-c_M(F;[pt])$ and hence, by replacing $F$ with $F-C$, we may assume that $F$ is supported in $A$. 
	Suppose that $A$ can be embedded into another symplectically aspherical manifold $(N,\Omega)$. Since $A$ is contractible, its boundary is simply connected, and in particular, incompressible in both $M$ and $N$. Since $\partial A$ is of contact-type, we conclude that $A\subset M$ and $A\subset N$ are CIB domains. 
	By Theorem~\ref{thm:indep_of_embedding}, the spectral invariants of Hamiltonians supported in $A$ on $M$ and $N$ coincide, and therefore the relative capacities of $A$ with respect to $M$ and $N$ agree, and we can define
	\begin{equation*}
	c_\gamma(A):=c_\gamma(A;M) = c_\gamma(A;N).
	\end{equation*}
	We may extend this definition to unbounded domains $U\subset\R^{2n}$ by taking the supremum over all $A\in \cS$ that can be embedded into $U$. 
	Before proving that $c_\gamma$ satisfies the axioms of a symplectic capacity, let us prove the second assertion of the corollary.
	Given $A\in\cS$ that can be symplectically embedded into $(\R^{2n},\omega_0)$, we need to show that $c_\gamma(A;M)\leq 2e(A;\R^{2n})$. 
	Let $Q=[-R,R]^{2n}\subset\R^{2n}$ be a large cube such that the embedding of $A$ into $\R^{2n}$ is displaceable in $Q$ with energy $e(A;\R^{2n})$. Then, embedding $Q$ into a large torus $N=\R^{2n}/(3R\Z^{2n})\cong \T^{2n}$, we conclude that $A$ is displaceable in $N$ with the same energy. By the energy-capacity inequality, for every Hamiltonian $F$ supported in the embedding of $A$ into $N$, and  for every homology class $\alpha$ one has $c(F;\alpha)\leq e(A;N)=e(A;\R^{2n})$. Using Theorem~\ref{thm:indep_of_embedding} we conclude that for every symplectically aspherical $M$ and an embedding of $A$ into $M$, $c_\gamma(A;M)=c_\gamma(A)\leq 2e(A;\R^{2n})$.  
	
	We now briefly explain why $c_\gamma$ satisfies the axioms of a capacity. Nontriviality follows from the fact that Schwarz's capacities are not smaller than the Hofer-Zehnder capacity, and are not greater than twice the displacement energy, see \cite{schwarz2000action}. Monotonicity follows from the definition of $c_\gamma(\cdot ;M)$, together with the fact that the image of every embedding of a domain in $\cS$ into a symplectically aspherical manifold is a CIB domain. To prove the conformality property, suppose $(A,\Omega)\in\cS$ is embedded into $(M,\omega)$, then $(A,\tau\cdot \Omega)$ is embedded into $(M,\tau\cdot \omega)$. In order to prove that 
	$$
		c_\gamma\left((A,\tau\Omega),(M,\tau\omega)\right) = |\tau|\cdot c_\gamma\left((A,\Omega),(M,\omega)\right),
	$$
	we show that for every $F$ such that $supp(X_F)\subset A\times S^1$ and for every homology class $\alpha\in H_*(M)$, it holds that 
	\begin{equation}\label{eq:spec_inv_rescaling}
	c_{(M,\tau\omega)}(|\tau|\cdot F;\alpha) = |\tau|\cdot c_{(M,\omega)}(F;\alpha)
	\end{equation}  
	Starting from the case where $\tau>0$, we notice that the action functional with respect to the form $\tau\omega$ and the Hamiltonian $\tau F$ is proportional to the action functional with respect to $\omega$ and $F$. The Floer complexes of $(\omega,J,F)$ and $(\tau\omega,J,\tau F)$ coincide, while the action filtration is rescaled by $\tau$, and therefore (\ref{eq:spec_inv_rescaling}) holds. It remains to deal with $\tau=-1$. In this case, the Floer complexes of $(\omega, J,F)$ and $(-\omega,-J,F)$ are isomorphic via the map $t\mapsto-t$, and the action filtration is the same. This implies that (\ref{eq:spec_inv_rescaling}) holds for negative $\tau$ as well.
\end{proof}

\begin{proof}[Proof of Corollary~\ref{cor:heavy_sets}]
	Let $A\subset M$ be a contractible domain with a contact-type boundary that can be symplectically embedded in $(\R^{2n},\omega_0)$. As in the proof of Corollary~\ref{cor:schwarts_capacity}, let $Q\subset \R^{2n}$ be a large enough cube such that the image of $A$ in $\R^{2n}$ is displaceable in $Q$. Embedding $Q$ into a large torus, $N\cong\T^{2n}$, we denote by $\Psi:A\hookrightarrow N$ the composition of the embeddings. As $\Psi(A)$ is displaceable in $N$, it follows from non-nativity of $c(\cdot;[M])$ (Lemma~\ref{lem:spec_inv_non_neg}), Theorem~\ref{thm:indep_of_embedding} and the energy capacity inequality that for every Hamiltonian $F:M\times S^1\rightarrow \R$ supported in $A$, 
	$$
	0\leq c_M(F;[M]) = c_N(\Psi_*F;[N])\leq e(A;N)<\infty.
	$$
	As a consequence, the partial symplectic quasi-state, $\zeta$, associated to $c$ vanishes on functions supported in $A$. The fact that the complement of $A$ is super-heavy follows from the following equivalent description of super-heavy sets. 
	\begin{itemize}
		\item[]\cite[Definition~6.1.10]{polterovich2014function}: A closed subset $X\subset M$ is super-heavy if $\zeta(F)=0$ for every Hamiltonian $F$ that vanishes on $X$.
	\end{itemize}  
	The fact that $A$ cannot contain a heavy set can be seen directly from the definition. Alternatively, this fact follows from the intersection property of heavy and super-heavy sets, established by Entov and Polterovich in \cite{entov2009rigid}: Every super-heavy set intersects every heavy set.
\end{proof}

We conclude this section with two examples, showing that Theorem~\ref{thm:indep_of_embedding} does not hold in a more general setting.
\begin{exam}\label{exa:locality_fails}
	The conditions on the manifolds $M$, $N$ and the domain $U$ in 	Theorem~\ref{thm:indep_of_embedding} are necessary:
	\begin{itemize}
		\item The condition on $M$ and $N$ being symplectically aspherical in Theorem~\ref{thm:indep_of_embedding} is necessary. A simple example is to embed the unit disk $D\subset \R^2$ into a small sphere and into a large sphere. Namely, take $M$ and $N$ to be spheres of areas $1.5\pi$ and $2\pi$ respectively. Then, there exist Hamiltonians, supported in the embedding of $D$ into $M$, with arbitrarily large spectral invariants with respect to the fundamental class. This follows from the fact that the embedding of $D$ into $M$ contains the equator, which is a heavy set, see  \cite[Chapter 6]{polterovich2014function}. A Hamiltonian $F$ that attains large values on the equator in $M$ has a large spectral invariant.
		
		On the other hand, the spectral invariant of any Hamiltonian that is supported in the embedding of $D$ into $N$ is bounded by  the displacement energy of this embedded disc in $N$, which is equal to $\pi$.
		
		\item The condition on $\partial U$ to be incompressible is also necessary. One can construct two different embeddings of the annulus $A:=int(D\setminus\frac{1}{2}D)$ into a torus of large area, such that the image under one embedding is heavy (and the boundary is incompressible), and the image under the other embedding is displaceable (and the boundary is not incompressible). As mentioned above, in the first case one can construct Hamiltonians with arbitrarily large spectral invariants (with respect to the fundamental class), and in the second case, the spectral invariant is bounded by the (finite) displacement energy. In particular, the assertion of Theorem~\ref{thm:indep_of_embedding} cannot hold in this case.
	\end{itemize}
\end{exam}

\section{Relation to certain open symplectic manifolds.}\label{sec:open_manifolds}
In this section we discuss an extension of Theorem~\ref{thm:indep_of_embedding}  to CIB domains in certain open symplectic manifolds. We start by briefly reviewing Floer homology on such manifolds, following \cite{frauenfelder2007hamiltonian}\footnote{Note that a lot of our sign choices are opposite to those of \cite{frauenfelder2007hamiltonian}. Essentially, the complex defined in \cite{frauenfelder2007hamiltonian} for a Hamiltonian $F$ coincides with the complex defined here for $-F$.}. 
Let $(W,\omega)$ be a $2n$-dimensional compact symplectic manifold with a contact-type boundary. 
Using the Liouville vector field $Y$, we can symplectically identify a neighborhood of the the boundary in $W$ with $\partial W\times (\varepsilon,0]$ endowed with the symplectic form $d(e^r \lambda)$, where $\lambda = \iota_Y\omega$ and $r$ is the coordinate on the interval. 
The {\it completion} of $(W,\omega)$ is defined to be 
\begin{eqnarray}
\widehat{W} &:=&  W\cup_{\partial W} \partial W\times [0,\infty), \nonumber\\
\widehat{\omega} &:=& \begin{cases}
\omega & \text{on }W,\\
d(e^r \lambda) & \text{on }\partial W \times (-\varepsilon, \infty).
\end{cases} \nonumber
\end{eqnarray}
Let $J$ be an $\widehat{\omega}$-compatible almost complex structure on $\widehat{W}$ that, on $\partial W$, maps $Y$ to the Reeb vector field $R$ and, on $\partial W\times [0,\infty)$, is time-independent and is invariant under $r$-translations. A time dependent Hamiltonian $F$ on $\widehat{W}$ is called {\it admissible} if it coincides on $\partial W\times [0,\infty)$ with $\rho(e^r)$ 
for a function $\rho:[0,\infty)\rightarrow\R$ whose derivative on $(0,\infty)$ is positive and smaller than the minimal period of a periodic Reeb orbit (note that in this case, $F$ has no 1-periodic orbits in $W\times (0,\infty)$).
For a generic admissible Hamiltonian, the Floer complex of the pair $(F,J)$ on the open manifold $(\widehat{W},\widehat{\omega})$ is generated by the 1-periodic orbits of $F$ in $W$, and the differential is defined by counting solutions of the Floer equation, as in the closed case (see Section \ref{sec:floer_preliminaries}). The above assumptions on $F$ and $J$ guarantee that finite energy solutions are contained in $W$. This follows form a standard application of the max-principle (see, for example, \cite[Lemma 1.8]{viterbo1999functors} and \cite[Lemma 2.1]{ritter2013topological}), or from Lemma~\ref{lem:solution_contained_U} above. The homology of this complex is independent of $F$ and $J$ and is isomorphic to the homology of $W$.
Spectral invariants on open manifolds were defined in \cite[Section 5]{frauenfelder2007hamiltonian} in complete analogy with the closed case\footnote{The definition in \cite{frauenfelder2007hamiltonian} is given for the point class, but generalizes as is to any $\alpha\in H_*(W)$.}. These invariants extend by continuity to any Hamiltonian supported in $W$. 

\begin{rem}\label{rem:locality_open_manifolds}
	It was suggested to us by Schlenk that Theorem~\ref{thm:indep_of_embedding} holds for the spectral invariant with respect to the point class on the above open manifolds as well. Namely, given a CIB domain $U$ in $W$ and a symplectic embedding $\Psi:(U,\omega) \rightarrow (W',\omega')$ whose image is a CIB domain in $W'$, for every Hamiltonian $F$ supported in $U$, 
	\begin{equation*}
	c_W(F;[pt])=c_{W'}(\Psi_* F;[pt]),
	\end{equation*}
	where $\Psi_* F :W'\times  S^1\rightarrow\R$ is the extension by zero of $F\circ\Psi^{-1}$.
\end{rem}

\subsection{The homology of the subcomplex $C_{U_\circ}(f)$.} \label{subsec:hom_of_subcomplex}
In what follows, $(M,\omega)$ denotes a closed symplectic manifold, as always. Given a Hamiltonian $F$ supported in $U$, let $(f,J)$ be a Floer regular pair on $M$ with a barricade in $U$ around $U_\circ$, for some $U_\circ\Subset U$. The block form (\ref{eq:block_form}) of the differential implies that the differential restricts to $C_{U_\circ}(f)\subset CF(f)$. In this section we study the homology of this subcomplex. 
We show that for a properly chosen such pair $(f,J)$, the homology of  $(C_{U_\circ}(f),\partial|_{U_\circ})$ coincides with the homology of $U$, namely,
\begin{equation}\label{eq:hom_subcomplex}
H_*\left(C_{U_\circ}(f),\partial|_{U_\circ}\right)\cong H_*(U).
\end{equation}  
For that end, consider a perturbation $f^\flat$ of $F$ such that $(f^\flat,J)$ has a $(\tau,\delta)$-bump around $\partial U$ (in the sense of Definition~\ref{def:bump}). In particular, we assume that $J$ is cylindrical. Let $f$ to be a $C^2$-small perturbation of $f^\flat$ such that the pair $(f,J)$ is Floer regular. 
As argued in the proof of Theorem~\ref{thm:barricade}, it follows from   Proposition~\ref{pro:bump_implies_barricade} and Proposition~\ref{pro:appendix} that the pair $(f,J)$ has a barricade in $U$ around $U_\circ:=\psi^{-\tau}U$. 
Taking $f$ to be close enough to $f^\flat$, the restriction $f\vert^{}_{U_\circ}$ of $f$ to $U_\circ$ can be extended to an admissible Hamiltonian $\hat f:=\widehat{f|_{U_\circ}}$ on $\widehat{U}$ that has no additional 1-periodic orbits. Here $(\widehat{U},\widehat{\omega})$ is the open symplectic manifold obtained as the completion of $U$. As the 1-periodic orbits of $\hat f$ in $\widehat{U}$ coincide with the 1-periodic orbits of $f$ that are contained in $U_\circ$, the Floer complex of $\hat f$ on the open manifold $\widehat{U}$ coincides with $C_{U_\circ}(f)$. Since both in $M$ and in $\widehat{U}$ all finite energy solutions of the Floer equation among orbits in $U_\circ$  are contained in $U_\circ$, the differentials coincide. We conclude that the homology of $(C_{U_\circ}(f),\partial|_{U_\circ})$ indeed coincides with $H_*(U)$.

\subsection{Locality of spectral invariants with respect to other homology classes.}
In this section we show how Floer homology on open manifolds is useful in the study of Floer complexes of Hamiltonians supported in CIB domains in closed manifolds\footnote{The results in this section can be achieved within the scope of Floer homology on closed manifolds, but the proof is slightly more complicated and less natural.}. In particular, we explain how to extend Theorem~\ref{thm:indep_of_embedding} to homology classes in the image of the map induced by the inclusion $\iota:U\hookrightarrow M$.

\begin{claim}\label{clm:locality_other_classes}
	For every class $\alpha\in \operatorname{im}(\iota_*)\subset H_*(M)$ and a Hamiltonian $F$ supported in $U$,
	\begin{equation}\label{eq:locality_other_classes}
	c_M(F;\alpha) = \min_{\tiny{\begin{array}{c}
			\beta\in H_*(U)\\ \iota_*(\beta)=\alpha
			\end{array}}} c_{\widehat{U}}(\hat F;\beta),
	\end{equation}
	where $c_M$ and $c_{\widehat{U}}$ are the spectral invariants in the manifolds $(M,\omega)$ and $(\widehat{U},\widehat{\omega})$ respectively, and $\hat F$ is the extension by zero of $F|_U$ to $\widehat{U}$.
\end{claim}
\begin{proof}
	The proof relies on the observations of Section~\ref{subsec:hom_of_subcomplex}: let $f$ be a perturbation of $F$ and $J$ an almost complex structure such that $(f,J)$ has a barricade in $U$ around $U_\circ$. Assume in addition that the perturbation is chosen to be arbitrarily close to some $f^\flat$, for which the pair $(f^\flat, J)$ has a cylindrical bump around $\partial U$. As explained previously, the Floer complex of $\hat f:=\widehat{f|_{U_\circ}}$ on $(\widehat{U},\widehat{\omega})$ coincides with the subcomplex $C_{U_\circ}(f)$ of $CF(f)$ in $M$.  We will show that formula (\ref{eq:locality_other_classes}) holds for $f$ and $\hat f$ up to $2\delta$, 
	for some $\delta$ which can be made arbitrarily small by shrinking the size of the perturbations.
	
	We start by noticing that given a class $\beta\in \iota_*^{-1}(\alpha)$, every representative $b\in C_{U_\circ}(f)$ of $\beta$, is a representative of $\alpha$ in $CF_*(f)$. This immediately implies that $c_M(f;\alpha) \leq \min_{\beta\in \iota_*^{-1}(\alpha)} c_{U_\circ}(\hat f;\beta)$. To prove inequality in the other direction, let $\beta\in\iota_*^{-1}(\alpha)$ be a class on which the minimum in the RHS of (\ref{eq:locality_other_classes}) is attained, and let $a\in CF_*(f)$ and $b\in C_{U_\circ}(f)$ be representatives of $\alpha$ and $\beta$ of minimal action levels. We need to show that $\lambda_f(b)\leq \lambda_f(a)+2\delta$, where $\lambda_f:CF_*(f)\rightarrow\R$ is the maximal action of an orbit, as defined in Notations~\ref{not:CF_element_contained}. 
	Notice that if $a\in C_{U_\circ}(f)$, then it represents in $C_{U_\circ}(f)$ a class in $\iota_*^{-1}(a)$ and, by our choice of $b$, $\lambda_f(b) \leq \lambda_f(a)$, which concludes the proof for this case. 
	Therefore we suppose that $a$ contains critical points in $M\setminus U_\circ$, which implies that $\lambda_f(a)>-\delta$. Assume for the sake of contradiction that  $\lambda_f(b) > \lambda_f(a)+2\delta$, then $\lambda_f(b)>\delta$.
	Recalling that $a$ and $b$ are homologous in $CF_*(f)$ (they both represent $\alpha$), there exists $c\in CF_*(f)$ such that $\partial c= a-b$. Consider the decomposition $c=\pi_{U_\circ}c+\pi_{U_\circ^c}c$, then,
	\begin{equation*}
	b':= b+\partial \pi_{U_\circ} c  = a-\partial\pi_{U_\circ^c}c\in C_{U_\circ}(f)
	\end{equation*}
	is homologous to $b$ in $C_{U_\circ}(f)$. This follows from the fact that $\partial\circ \pi_{U_\circ}=\partial|_{U_\circ}$, since $(f,J)$ has a barricade in $U$ around $U_\circ$. Therefore, $b'$ represents in $C_{U_\circ}(f)$ a class in $\iota_*^{-1}(\alpha)$, and by our choice of $b$, it holds that $\lambda_f(b)\leq \lambda_f(b')$. On the other hand, 
	\begin{equation*}
	\lambda_f(b') = \lambda_f( a-\partial\pi_{U_\circ^c}c)\leq \max\{ \lambda_f(a) ,\lambda_f(\partial\pi_{U_\circ^c}c)\}\leq \max\{ \lambda_f(a) ,\delta\}<\lambda_f(b),
	\end{equation*}
	in contradiction.
\end{proof}

\begin{rem}
	When $U$ is a disjoint union of $\{U_i\}$ and $\alpha=[pt]\in H_*(M)$, equality (\ref{eq:locality_other_classes}) implies the min formula for the point class, which is equivalent, by Poincar\'e duality, to Theorem 45 in \cite{humiliere2016towards} (the max formula). 
\end{rem}

\section{Spectral invariants of disjointly supported Hamiltonians.}	\label{sec:spec_inv_max_ineq}
In this section we use barricades to prove Theorem~\ref{thm:spectral_inv_max_ineq}, which states that a max inequality holds for spectral invariants of Hamiltonians supported in disjoint CIB domains, with respect to a general class $\alpha\in H_*(M)$, and that equality holds when $\alpha=[M]$. Suppose $F$ and $G$ are two Hamiltonians supported in disjoint CIB domains.
In order to prove the max inequality (\ref{eq:max-ineq}) for a homology class $\alpha\in H_*(M)$, we construct a representative of $\alpha$ in the Floer complex of (a perturbation of) the sum $F+G$, out of representatives from the Floer complexes of (perturbations of) $F$ and $G$. The communication between the different Floer complexes is through continuation maps, corresponding to (perturbations of) linear homotopies. The barricades will be used to study the continuation maps, or, more accurately, their restrictions to the CIB domains. In particular, we will use the observation that having a barricade for a disjoint union implies having a barricade for each component:
\begin{rem}
	Consider two disjoint domains $U$ and $V$ in $M$, and a pair $(H,J)$ of a homotopy (or a Hamiltonian) and an almost complex structure, that has a barricade in $U\cup V$ around $U_\circ\cup V_\circ$, for some $U_\circ\Subset U$ and $V_\circ\Subset V$. It follows from Definition~\ref{def:barricade} of the barricade that the pair $(H,J)$ has a barricade in $U$ around $U_\circ$ (and, similarly, in $V$ around $V_\circ$).
\end{rem}
We start by arranging the set-up required for the proof of Theorem~\ref{thm:spectral_inv_max_ineq}.

\begin{lemma}[Set-up]\label{lem:reg_pairs_for_spec_max_ineq}
	Let $F$ and $G$ be Hamiltonians supported in disjoint CIB domains, $U$ and $V$, respectively. Then, there exist an almost complex structure $J$ and homotopies $h_{F}$ and $h_G$ such that the following hold:
	\begin{enumerate}
		\item The pairs $(h_{F},J)$, $(h_{F\pm},J)$, $(h_{G},J)$ and $(h_{G\pm},J)$ are all Floer-regular and have barricades in  $U\cup V$ around $U_\circ\cup V_\circ$ for some $U_\circ\Subset U$ and $V_\circ\Subset V$ containing the supports of $F$ and $G$, respectively.
		\item The left ends, $h_{F-}$ and $h_{G-}$, are small perturbations of $F$ and $G$, respectively. The right ends coincide, $h_{F+}=h_{G+}$, and are a small perturbation of the sum $F+G$.
		\item On $U\times S^1$ (respectively,  $V\times S^1$) the homotopy $h_{F}$ (respectively, $h_G$) is a small perturbation of a constant homotopy, and its ends agree on their 1-periodic orbits up to second order. In particular, $h_{F-}$ and $h_{F+}$ (resp., $h_{G-}$ and $h_{G+}$) have the same 1-periodic orbits in $U$ (resp., $V$).
	\end{enumerate}
\end{lemma}

\begin{proof}
	Let $H_F$ and $H_G$ be linear homotopies from $F$ and $G$, respectively, to the sum $F+G$. As in the proof of Lemma~\ref{lem:reg_pairs_for_locality}, we consider perturbations, $h_F^\flat$ and $h_G^\flat$,  of the linear homotopies, that, when paired with $J$, have a cylindrical bump around $\partial U\cup \partial V$. We demand in addition that all ends are non-degenerate, that the right ends coincide, $h_{F+}^\flat =  h_{G+}^\flat$, and that the homotopies are constant on $U$ and $V$ respectively, $h_{F}^\flat|_U\equiv h_{F_-}^\flat|_U$ and $h_{G}^\flat|_V\equiv h_{G_-}^\flat|_V$. By proposition~\ref{pro:bump_implies_barricade}, these homotopies and their ends, when paired with $J$, have barricades in $U\cup V$ around $U_\circ\cup V_\circ$. It remains to perturb again to ensure regularity. As in the proof of Theorem~\ref{thm:barricade}, we replace the ends with regular  perturbations $h_{F-}$, $h_{G-}$ and $h_{F+}=h_{G+}$, without changing their periodic orbits (as cited in Claim~\ref{clm:app_reg_Hamiltonians}, for example), then perturb the homotopies to glue to these regular perturbed Hamiltonians, and finally perturb the homotopies on the set $M\times S^1\times I$ for some fixed finite interval $I$, to obtain homotopies that are Floer regular when paired with $J$. The last step is possible due to Proposition~\ref{pro:pert_hom_bdd_supp}) below. Proposition~\ref{pro:appendix} states that barricades survive under perturbations that do not change the periodic orbits of the ends and are constant (as homotopies) outside of some fixed finite interval.
\end{proof}

The following lemma is actually a part of the proof of Theorem~\ref{thm:spectral_inv_max_ineq}, but, in our opinion, might be interesting on its own.

\begin{lemma}\label{lem:min_ineq}
	Let $\alpha\in H_*(M)$ and let $F,G:M\times  S^1\rightarrow\R$ be Hamiltonians supported in disjoint CIB domains $U$, $V$ respectively. Assume in addition that $c(F;\alpha)< 0$, then
	\begin{equation}\label{eq:min_ineq}
	c(F+G;\alpha)\leq \min\{c(F;\alpha),c(G;\alpha)\}.
	\end{equation}	
\end{lemma}
\begin{proof}
	Let us show that $c(F+G;\alpha)\leq c(F;\alpha)$. The result will follow by symmetry, since, if $c(G;\alpha)<c(F;\alpha)$, then it is in particular negative. 
	
	Let $h_F$ and $J$ be the homotopy and almost complex structure from the Set-up Lemma, \ref{lem:reg_pairs_for_spec_max_ineq} (we will not use $h_G$ in this proof), and denote the left end of the homotopy by $f:=h_{F-}$. Then $f$ approximates $F$ and, since $c(F;\alpha)<0$ and $F|_{U_\circ^c}=0$, we may assume that $c(f;\alpha)<\min_{U_\circ^c}f$. Outside of $U_\circ$, $f$ is a small Morse function and hence its 1-periodic orbits there are critical points, and their actions are the critical values. As a consequence, a representative $a\in CF_*(f)$ of the class $\alpha$, of action level $\lambda_f(a)=c(f;\alpha)$ must be a combination of orbits that are contained in $U_\circ$, namely, $a\in C_{U_\circ}(f)$. As the continuation map, $\Phi_{(h_F,J)}:CF_*(f)\rightarrow CF_*(h_{F+})$, induces isomorphism on homologies, the image $\Phi_{(h_F,J)}(a)$ of $a$ represents the class $\alpha$ in $CF_*(h_{F+})$. Recalling that, on $U$, the homotopy $h_F$ is a small perturbation of a constant homotopy, it follows from Corollary~\ref{cor:almost_constant_homotopy}  that the restriction of the continuation map $\Phi_{(h_F,J)}$ to orbits contained in $U_\circ$ is the identity map:
	\begin{equation*}
	\Phi_{(h_F,J)}\circ \pi_{U_\circ} = \id \circ \pi_{U_\circ}.
	\end{equation*}
	Therefore, $\Phi_{(h_F,J)}(a)=a$ is a representative of the class $\alpha$ of action level $\lambda_{h_{F+}}(\Phi_{(h_F,J)}(a)) = \lambda_f(a)=c(f;\alpha)$. We conclude that $c(h_{F+};\alpha)\leq c(f;\alpha)$ as required.
\end{proof}

The following example shows that a strict inequality can be attained in (\ref{eq:min_ineq}).
\begin{exam}\label{exa:min_ineq_spec}
	Let $(M,\omega)$ be a genus-2 surface endowed with an area form, and let $x,y:S^1\rightarrow M$ be two disjoint non-contractible loops representing two different homology classes $\alpha, \beta\in H_1(M;\Z_2)$ respectively. 
	Let $F,G:M\rightarrow\R$ be two small Morse functions with disjoint supports, such that $F$ vanishes on $y$ and takes a negative value on $x$, whereas $G$ vanishes on $x$ and is negative on $y$. See {Figure}~\ref{fig:sunglasses} for an illustration. After perturbing $F$, $G$ and $F+G$ into Morse functions, representatives of the sum $\alpha+\beta$ first appear for $F$ and $G$ on a sub-level set of values approximately zero. However, this sum of classes appears for $F+G$ in a sub-level set with negative value. We therefore conclude that the spectral invariants of both $F$ and $G$ with respect to the sum $\alpha+\beta$ vanish. On the other hand, the spectral invariant of $F+G$ is negative, and thus
	\begin{equation*}
	c(F+G;\alpha+\beta) < 0 = \min\{c(F;\alpha+\beta),c(F;\alpha+\beta)\}.
	\end{equation*}
\end{exam}
\begin{figure}
	\centering
	\begin{subfigure}{.5\textwidth}
		\centering
		\includegraphics[width=.9\linewidth]{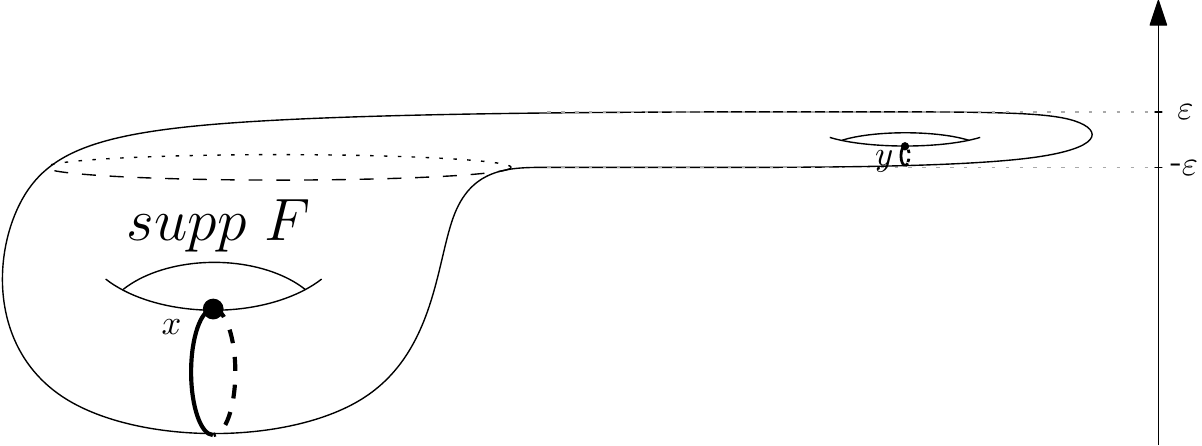}
		\caption{A perturbation of $F$}
		\label{fig:sunglasses1}
	\end{subfigure}%
	\begin{subfigure}{.5\textwidth}
		\centering
		\includegraphics[width=.9\linewidth]{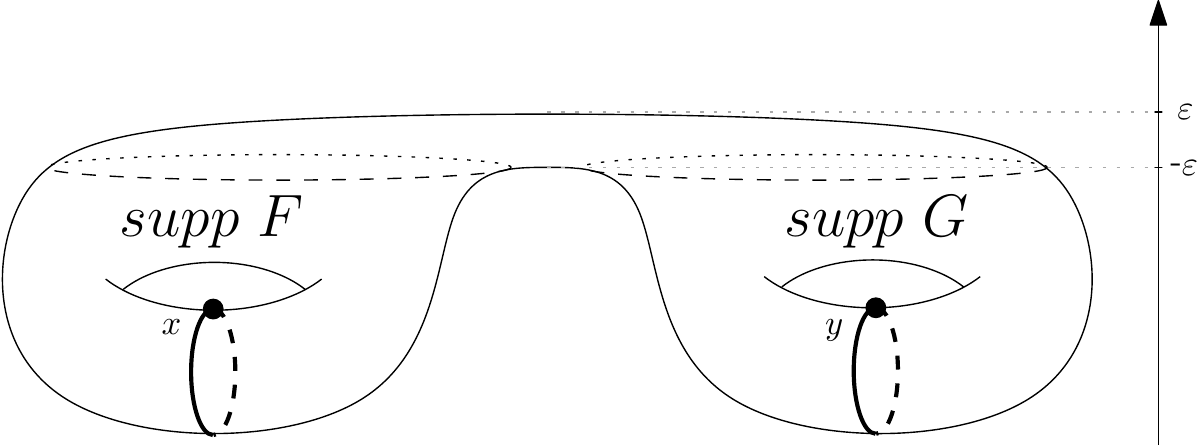}
		\caption{A perturbation of $F+G$}
		\label{fig:sunglasses2}
	\end{subfigure}
	\caption{ An illustration of non-degenerate perturbations of $F$ and $F+G$. A representative of the class $\alpha+\beta$ appears at level $\approx0$ for $F$, and at a negative value for $F+G$.}
	\label{fig:sunglasses}
\end{figure}

The following inequality is a simple application of Lemma~\ref{lem:spec_inv_non_neg} and Lemma~\ref{lem:min_ineq}, and will be used to prove that equality holds in (\ref{eq:max-ineq}) for the fundamental class.
\begin{lemma}\label{lem:spec_inv_inverse_ineq}
	Let $F,G:M\times  S^1\rightarrow\R$ be Hamiltonians supported in disjoint CIB domains, then 
	\begin{equation*}
	c(F+G;[M])\geq \max\{c(F;[M]), c(G;[M])\}.
	\end{equation*}	
\end{lemma}
\begin{proof}
	By Lemma~\ref{lem:spec_inv_non_neg}, the spectral invariants of $F$, $G$ and $F+G$ with respect to $[pt]$ are non-positive, and thus, using the Poincar\'e duality property for spectral invariants, we conclude that $c(F+G;[M])$, 	$c(F;[M])$ and $c(G;[M])$ are all non-negative. If both $c(F;[M])$ and $c(G;[M])$ are equal to zero, the claim is trivial. Thus we assume, without loss of generality, that $c(F;[M])>0$. By Poincar\'e duality, $c(-F;[pt])=-c(F;[M])<0$ and we may apply Lemma~\ref{lem:min_ineq} to $-F$, $-G$ and $\alpha=[pt]$:
	\begin{eqnarray*}
		c(F+G;[M]) &=& -c(-F-G;[pt]) \geq -\min\{c(-F;[pt]), c(-G;[pt])\} \\
		&=&  -\min\{-c(F;[M]), -c(G;[M])\} \\
		&=& \max \{c(F;[M]), c(G;[M])\}
	\end{eqnarray*}
\end{proof}

\begin{proof}[Proof of Theorem~\ref{thm:spectral_inv_max_ineq}]
	In what follows we prove that the spectral invariant of the sum $F+G$ with respect to a homology class $\alpha$ is not greater than the maximum. The equality for the fundamental class will follow from  Lemma~\ref{lem:spec_inv_inverse_ineq}. Consider the almost complex structure $J$, and the homotopies, $h_F$ and $h_G$, from the Set-up Lemma, \ref{lem:reg_pairs_for_spec_max_ineq}, and denote 
	$$
	f:=h_{F-}\approx F,\quad g:=h_{G-}\approx G \quad \text{and}\quad {h_+}:=h_{F+}=h_{G+}\approx F+G.
	$$
	Set $\lambda:=\max\{c(f;\alpha), c(g;\alpha)\}$ and notice that, due to Lemma~\ref{lem:min_ineq} and the continuity of spectral invariants, we may assume that $\lambda\geq -\delta$ if $\delta>0$ is small enough.
	Let $\tilde a \in CF_*(f)$, $\tilde b\in CF_*(g)$ be representatives of $\alpha$ of action levels $\lambda_f(\tilde a),\lambda_g(\tilde b)\leq \lambda$, then $a:=\Phi_{(h_F,J)} \tilde a$ and $b:=\Phi_{(h_G,J)} \tilde b$ are both representatives of $\alpha$ in $CF_*({h_+})$. Notice that $a$ and $b$ might be of action-level higher than $\lambda$. We wish to construct out of $a$ and $b$ a representative of $\alpha$ of action level approximately bounded by $\lambda$. Let $p$ be a primitive of $a-b$, and set $d:=\left(\partial_{({h_+},J)} \pi_V -\pi_V\partial_{({h_+},J)}\right) p$. We claim that  
	$$
	e:=\pi_{V^c}a +\pi_Vb -d
	$$ 
	is a representative of $\alpha$ of the required action level. Indeed, 
	\begin{eqnarray*}
		[\pi_{V^c}a +\pi_Vb -d] &=& [\pi_{V^c}a +\pi_Vb -\partial_{({h_+},J)} (\pi_Vp) +\pi_V(\partial_{({h_+},J)} p)]\\
		&=& [\pi_{V^c}a +\pi_Vb -\partial_{({h_+},J)} (\pi_Vp) +\pi_V(a-b)]\\
		&=& [\pi_{V^c}a +\pi_Vb -\partial_{({h_+},J)} (\pi_Vp) +\pi_Va-\pi_Vb]\\
		&=& [\pi_{V^c}a +\pi_Va -\partial_{({h_+},J)} (\pi_Vp) ] = [a] = \alpha.
	\end{eqnarray*}	
	Let us now bound the action level of $e$. First, notice that outside of $U_\circ \cup V_\circ$, ${h_+}$ is a small Morse function (as it approximates a Hamiltonian that is supported in $U_\circ \cup V_\circ$).  Therefore, its 1-periodic orbits there are its critical points and their actions are the critical values, which we may assume to be bounded by $\delta$. It follows that the action level of the projection $\pi_{U_\circ^c\cap V_\circ^c} (e)$ is bounded by $\delta$, and so it remains to bound the action levels of $\pi_{U_\circ} e$ and $\pi_{V_\circ}e$. It follows form the fact that $({h_+},J)$ has a barricade in $U$ around $U_\circ$ and in $V$ around $V_\circ$ (more specifically, from (\ref{eq:block_form})), that
	\begin{equation}\label{eq:composition_proj_and_diff_0}
	\pi_{U_\circ}\circ \partial_{({h_+},J)} \circ \pi_{U^c} = 0, \quad \pi_{V_\circ}\circ \partial_{({h_+},J)} \circ \pi_{V^c} = 0.
	\end{equation}
	Using this observation, we bound the action levels of the projections of $e$:
	\begin{itemize}
		\item $\lambda_{h_+}(\pi_{U_\circ} e)$: Notice that $\pi_{U_\circ} d=0$. Indeed, 
		\begin{equation*}
		\pi_{U_\circ} d = \pi_{U_\circ}\circ (\partial_{({h_+},J)} \circ \pi_V -\pi_V\circ\partial_{({h_+},J)} )p=\pi_{U_\circ}\circ \partial_{({h_+},J)} \circ \pi_Vp
		\overset{(\ref{eq:composition_proj_and_diff_0})}{=}	0.
		\end{equation*}
		As a consequence, $\pi_{U_\circ} e = \pi_{U_\circ} a = \pi_{U_\circ} \Phi_{(h_F,J)} \tilde a$. Since, on $U$, the homotopy $h_F$ is a perturbation of the constant homotopy, we can apply Corollary~\ref{cor:almost_constant_homotopy} and conclude that $\pi_{U_\circ}\circ \Phi_{(h_F,J)} = \pi_{U_\circ}$. Overall we obtain
		\begin{equation*}
		\lambda_{h_+}(\pi_{U_\circ} e) = \lambda_{h_+}(\pi_{U_\circ}\circ \Phi_{(h_F,J)}\tilde a) = \lambda_{h_+}(\pi_{U_\circ}\tilde a) = \lambda_f(\pi_{U_\circ}\tilde a) \leq \lambda_f(\tilde a)\leq \lambda,
		\end{equation*}
		where we used the fact that in $U$, $f=h_{F-}$ and $h_+ = h_{F+}$ agree on their 1-periodic orbits, and hence the action of $\pi_{U_\circ}\tilde a$ with respect to $h_+$ coincides with the action with respect to $f$. 
		\item $\lambda_{h_+}(\pi_{V_\circ} e)$: Here $\pi_{V_\circ} d=0$ as well, but the computation is a little different:
		\begin{eqnarray*}
		\pi_{V_\circ} d &=& \pi_{V_\circ}\circ (\partial_{({h_+},J)} \circ \pi_V -\pi_V\circ\partial_{({h_+},J)} )p=(\pi_{V_\circ}\circ \partial_{({h_+},J)} \circ \pi_V - \pi_{V_\circ}\circ \partial_{({h_+},J)})p\\
		&=& (\pi_{V_\circ}\circ \partial_{({h_+},J)}-\pi_{V_\circ}\circ \partial_{({h_+},J)} \circ \pi_{V^c} - \pi_{V_\circ}\circ \partial_{({h_+},J)})p
		\overset{(\ref{eq:composition_proj_and_diff_0})}{=}	0.
		\end{eqnarray*}
		Therefore, $\pi_{V_\circ} e = \pi_{V_\circ} b = \pi_{V_\circ} \Phi_{(h_G,J)} \tilde b$ and, since, on $V$, the homotopy $h_G$ is a perturbation of the constant homotopy, we apply Corollary~\ref{cor:almost_constant_homotopy} and conclude that $\pi_{V_\circ}\circ \Phi_{(h_G,J)} = \pi_{V_\circ}$. Overall,
		\begin{equation*}
		\lambda_{h_+}(\pi_{V_\circ} e) = \lambda_{h_+}(\pi_{V_\circ}\circ \Phi_{(h_G,J)}\tilde b) = \lambda_{h_+}(\pi_{V_\circ}\tilde b) = \lambda_g(\pi_{V_\circ}\tilde b \leq \lambda_g(\tilde b) \leq \lambda,
		\end{equation*}
		where we used the fact that on $V$, $g=h_{G-}$ and $h_+ = h_{G+}$ agree on their 1-periodic orbits, and hence the action of $\pi_{V_\circ}\tilde a$ with respect to $h_+$ coincides with the action with respect to $g$. 
	\end{itemize}
	We conclude that 
	\begin{eqnarray}
	\nonumber c(h_+;\alpha) &\leq& \lambda_{h_+} (e) \leq \max\{\lambda_{h_+}(\pi_{U_\circ} e ),\lambda_{h_+}(\pi_{V_\circ} e ),\lambda_{h_+}(\pi_{U_\circ^c\cap V_\circ^c} e )\}\\
	\nonumber &=& \max\{\lambda, \delta\}\leq \lambda +2\delta.
	\end{eqnarray}
\end{proof}

\section{Boundary depth of disjointly supported Hamiltonians.} \label{sec:BD_inequality}
In this section, we use barricades to compare the boundary depths of disjointly supported Hamiltonians and that of their sum. As in the previous section, the communication between Floer complexes of different Hamiltonians is through continuation maps corresponding to homotopies that have barricades. Since we replace the Hamiltonians and their sum by regular perturbations, we will use the continuity property of the boundary-depth:
\begin{itemize}
	\item[] \cite[Theorem 1.1]{usher2011boundary}: Given two Hamiltonians $F$ and $G$, 
	$$
	|\beta(F)-\beta(G)|\leq \int_0^1\left(\max_M(F-G)-\min_M(F-G)\right)dt.
	$$
\end{itemize}
As before, we use Notations~\ref{not:CF_element_contained}. Let us start with a lemma that will enable us to push certain boundary terms from one Floer complex to another.	
\begin{lemma}\label{lem:bdry_term_in_U0}
	Let $J$ be an almost complex structure and let $h$ be a homotopy, such that the pairs $(h,J)$ and $(h_\pm,J)$ are Floer-regular and have a barricade in $U$ around $U_\circ$. Assume in addition that on $U$, $h$ is a small perturbation of a constant homotopy, and that its ends, $h_\pm$, agree up to second order on their 1-periodic orbits in $U$. Then, every boundary term  $a\in \partial_{(h_+,J)}CF_*(h_+) $ that is a combination of orbits in $U_\circ$, namely $a\in C_{U_\circ}(h_+)$, is also a boundary term in $CF_*(h_-)$.	
\end{lemma}
\begin{proof}
	We start with the observation that, since $h_-$ and $h_+$ are close on $U$ and agree on their periodic orbits there, the vector spaces $C_U(h_-)$ and $C_U(h_+)$ coincide. Therefore, a boundary term $a\in CF_*(h_+)$ that is a combination of orbits from $U_\circ$ is also an element of $C_{U_\circ}(h_-)$. Let us show that $a$ is a boundary term in the Floer complex of $(h_-, J)$.   
	As the homotopy $h$ is close to a constant homotopy on $U$, we may use  Corollary~\ref{cor:almost_constant_homotopy} and conclude that  $\Phi_{(h,J)}\circ\pi_{U_\circ} = \pi_{U_\circ}$. Applying this equality to $a$, we obtain 
	$$
	\Phi_{(h,J)}a = \Phi_{(h,J)}\circ\pi_{U_\circ}a = \pi_{U_\circ}a = a,
	$$
	namely, $a\in CF_*(h_+)$ is the image of itself under the continuation map. As $\Phi_{(h,J)}$ induces isomorphism on homologies, it is enough to show that $a$ is closed in $CF_*(h_-)$, and it will then follow that it is a boundary term. To see that $a$ is closed in $CF_*(h_-)$, notice that the presence of a barricade (in particular, (\ref{eq:block_form})) implies that $\partial_{(h_-, J)}a\in C_{U_0}(h_-)$, namely, $\partial_{(h_-, J)}a= \pi_{U_\circ} \partial_{(h_-, J)}a$. Therefore,
	$
		\partial_{(h_-,J)} a = \pi_{U_\circ}\partial_{(h_-,J)} a = \pi_{U_\circ}\Phi_{(h,J)}\partial_{(h_-,J)} a = \pi_{U_\circ}\partial_{(h_+,J)} \Phi_{(h,J)}a = \pi_{U_\circ}\partial_{(h_+,J)} a  =0.
	$
\end{proof}

We are now ready to prove Theorem~\ref{thm:bd_ineq}.
\begin{proof}[Proof of Theorem~\ref{thm:bd_ineq}]
	 In what follows we show that $\beta(F+G)\geq \beta(F)$. Inequality (\ref{eq:bd_ineq_geq}) follows by symmetry.
	Let $H$ be a linear homotopy from $F+G$ to $F$. Notice that, since $F$ and $F+G$ agree on $U$, $H$ is a constant homotopy there. By Theorem~\ref{thm:barricade}, there exist a perturbation $h$ of $H$ and an almost complex structure $J$, such that the pairs $(h,J)$ and $(h_\pm,J)$ are Floer-regular and have a barricade in $U\cup V$ around $U_\circ\cup V_\circ$, for $U_\circ\Subset U$, $V_\circ\Subset V$ containing the supports of $F$, $G$, respectively. Since $H$ is a constant homotopy on $U$, it follows from  Remark~\ref{rem:bump_implies_barricade}, item~\ref{itm:ends_agree_on_P}, that $h$ can be chosen such that, in $U$, $h_\pm$ agree on their 1-periodic orbits  up to second order. We stress that $h_-$ approximates $F+G$ and that $f:=h_+$ approximates $F$. Hence, fixing an arbitrarily small $\delta>0$, we may assume (by taking $h$ to be close enough to $H$) that $h_-|_{U_\circ^c\cap V_\circ^c}$ and $f|_{U_\circ^c}$ are small Morse functions with values in $(-\delta,\delta)$.
	Due to the continuity of the boundary depth, it is enough to prove that $\beta(f)$ is approximately bounded by $\beta(h_-)$.
		
	Fix a boundary term  $a\in CF_*(f)$, and let us show that there exists a primitive of $a$ whose action level is bounded by $\lambda_f(a)+\beta(h_-)+4\delta$, for $\delta$ that was fixed above. We prove this claim in two steps:

	\noindent\underline{Step 1:} Assume that $a$ is a combination of orbits that are contained in $U_\circ$, namely $a\in C_{U_\circ}(f)$. Applying Lemma~\ref{lem:bdry_term_in_U0} to $(h,J)$, we find that $a\in CF_*(h_-)$ is also a boundary term. Therefore, there exists $b\in CF_*(h_-)$ such that $\partial_{(h_-,J)}b=a$ and $\lambda_{h_-}(b)\leq \lambda_{h_-}(a)+\beta(h_-)$. Let us split into two cases:
	\begin{itemize}
		\item $\lambda_{h_-}(b)<-\delta$: Since $h_-$ is a small Morse function outside of $U_\circ \cup V_\circ$, its 1-periodic orbits there are its critical points, and their actions are the critical values, which are all contained in the interval $(-\delta, \delta)$. As a consequence, $b$ is necessarily a combination of orbits that are contained in ${U_\circ\cup V_\circ}$, namely,  $b\in C_{U_\circ\cup V_\circ}({h_-})$. Writing $b= \pi_{U_\circ} b+\pi_{V_\circ}b$, the presence of the barricade (in particular, (\ref{eq:block_form}))  guarantees that  $\partial_{(h_-,J)}\pi_{U_\circ} b\in C_{U_\circ}({h_-})$ and $\partial_{(h_-,J)} \pi_{V_\circ}b\in C_{V_\circ}({h_-})$. Recalling that $\partial_{(h_-,J)} b=a\in C_{U_\circ}({h_-})$, we conclude that $\partial_{(h_-,J)} \pi_{V_\circ}b=0$:
		\begin{eqnarray}
			\partial_{(h_-,J)} \pi_{V_\circ}b &=&\pi_{V_\circ} \left(\partial_{(h_-,J)} \pi_{V_\circ}b \right)  = \pi_{V_\circ} \left(\partial_{(h_-,J)} b-\partial_{(h_-,J)} \pi_{U_\circ}b\right) \nonumber\\
			&=&\pi_{V_\circ} \left(a-\pi_{U_\circ}\partial_{(h_-,J)} \pi_{U_\circ}b\right)=0.	 \nonumber
		\end{eqnarray} 
		Replacing $b$ by $\pi_{U_\circ}b$, we still have a primitive of $a$ of non-greater action level, as $\lambda_{h_-}(b) = \max\{\lambda_{h_-}(\pi_{U_\circ}b), \lambda_{h_-}(\pi_{V_\circ}b)\}$. Therefore, we may assume that $b\in C_{U_\circ}({h_-})$, and so it is also an element of $C_{U_\circ}(f)$. Recalling that  $h$ is a perturbation of a constant homotopy on $U$, Corollary~\ref{cor:almost_constant_homotopy} states that $\Phi_{(h,J)}\circ \pi_{U_\circ}  = \pi_{U_\circ}$, and hence $\Phi_{(h,J)}b=b$ and $\Phi_{(h,J)}a=a$. Thus,
		$$
		\partial_{(f,J)} (b)= \partial_{(f,J)}(\Phi_{(h,J)}b)= \Phi_{(h,J)}(\partial_{(h_-,J)} b )= \Phi_{(h,J)}a = a,
		$$
		i.e., $b$ is a primitive of $a$ in $CF_*(f)$, with small enough action level: $\lambda_f(b)=\lambda_{h_-}(b)\leq \lambda_f(a) +\beta({h_-})$.
		
		\item If $\lambda_{h_-}(b)\geq -\delta$. Then, writing $b=\pi_U b+\pi_{U^c}b$, the presence of a barricade (in particular, (\ref{eq:block_form})) implies that $\Phi_{(h,J)} \pi_{U^c} b\in C_{U^c}(f)$ and hence $\lambda_f (\Phi_{(h,J)} \pi_{U^c} b)\leq \delta$. Turning to bound the action of the projection onto $U$, recall that $h$ is a perturbation of a constant homotopy on $U$, and by Corollary~\ref{cor:almost_constant_homotopy}, $\pi_U \circ \Phi_{(h,J)} = \pi_U$. Overall,
		\begin{eqnarray*}
		\lambda_f (\Phi_{(h,J)} b)&\leq &\max\{\lambda_{f}(\Phi_{(h,J)} \pi_{U^c}b), \lambda_{f}(\Phi_{(h,J)} \pi_U b)\} \\
		&\leq& \max\{\delta, \lambda_{h_-}(b)\} \leq \lambda_f(a)+\beta({h_-})+2\delta.
		\end{eqnarray*}
	\end{itemize} 		
	\noindent\underline{Step 2:} Let us prove the claim for general $a$. Note that if $\lambda_f(a)<-\delta$ then $a\in C_{U_\circ}(f)$ and the claim follows from the previous step. Therefore, we assume that $\lambda_f(a)\geq-\delta$. Let $b$ be any primitive of $a$ in $CF_*(f)$, namely, $\partial_{(f,J)} b=a$, and write $b=\pi_{U_\circ}b+\pi_{U_\circ^c}b$.
	Both $\pi_{U_\circ^c}b$ and $\partial_{(f,J)}\pi_{U_\circ^c}b$ have action levels bounded by $\delta$. Set $a':=\partial_{(f,J)} b_{U_\circ}$, then 
	$$
	\lambda_f(a')= \lambda_f(a-\partial_{(f,J)}\pi_{U_\circ^c}b)\leq \max\{\lambda_f(a), \lambda_f(\partial_{(f,J)} \pi_{U_\circ^c}b)\}\leq \lambda_f(a)+2\delta.
	$$
	Moreover, the presence of the barricade implies that $a'\in C_{U_\circ}(f)$. Therefore, we may apply the previous step to $a'$ and obtain $b'\in CF_*(f)$ such that $\partial_{(f,J)}b'=a'$ and $\lambda_f(b')\leq \lambda_f(a')+\beta(h_-)+2\delta\leq \lambda_f(a)+\beta(h_-)+4\delta$. To conclude the proof, notice that $b'+b_{U_\circ^c}$ is a primitive of $a$ and $$\lambda_f(b'+b_{U_\circ^c})\leq \max\{\lambda_f(b'),\lambda_f(b_{U_\circ})\}\leq \lambda_f(a)+\beta(h_-)+4\delta.$$
\end{proof}

The following example shows that equality does not hold in (\ref{eq:bd_ineq_geq}).
\begin{exam}\label{exa:BD_inverse_ineq}
	Let $M=\T^2$ be the two dimensional torus equipped with an area form and take $F$ and $G$ be disjointly supported $\cC^2$-small non-negative bumps, see Figure~\ref{fig:bdry_depth}. Approximating $F$, $G$ and $F+G$ by small Morse functions, their Floer complexes and differentials are equal to the Morse complexes and differentials. Hence, the Floer differentials of both $F$ and $G$ vanish and in particular $\beta(F)=0=\beta(G)$. On the other hand, $\beta(F+G)=\min\{\max F,\max G\}$.
\end{exam}
\begin{figure}
	\centering
	\includegraphics[scale=0.9]{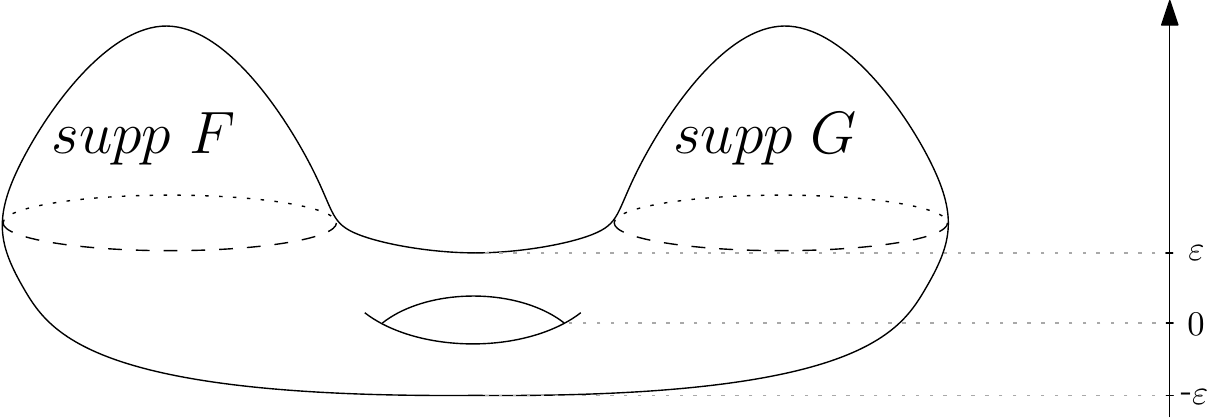}
	\caption{\small{An illustration of a non-degenerate perturbation of the sum $F+G$ from Example~\ref{exa:BD_inverse_ineq}.
	}}
	\label{fig:bdry_depth}
\end{figure}

\section{Min inequality for the AHS action selector.}\label{sec:AHS_selector}
In this section, we use barricades to prove a ``min inequality" for the action selector defined by  Abbondandolo, Haug and Schlenk, in \cite{abbondandolo2019simple}, on symplectically aspherical manifolds. We start by reviewing the construction of this action selector, which we denote by $c_{AHS}$, and state a few of its properties.

Let $H:M\times S^1\times \R\rightarrow\R$ be a homotopy of Hamiltonians and let $J:S^1\times \R\rightarrow\cJ_\omega$ be a homotopy of time-dependent almost complex structures (that are compatible with $\omega$). Assume that $\partial_sH$ and $\partial_s J$ have compact support and denote by $H_\pm$, $J_\pm$ the ends of the homotopies. 
As before, we denote by $\cM_{(H,J)}$ the set of all finite-energy solutions of the Floer equation (\ref{eq:FE}) with respect to $(H,J)$. On this space, define the functional $a_{H_-}:\cM_{(H,J)}\rightarrow\R$ by $a_{H_-}(u):= \lim_{s\rightarrow-\infty} \cA_{H_-}(u(s,\cdot))$. The existence of this limit follows from the fact that the homotopies $H$ and $J$ are constant outside of a compact set, and hence, when $s$ approaches $-\infty$, the function $s\rightarrow\cA_{H_-}(u(s,\cdot))$ is non-increasing and bounded, see for example \cite[p.8]{abbondandolo2019simple}.
Given a Hamiltonian $F:M\times S^1\rightarrow \R$, denote by $\cD(F):=\left\{(H,J)\ |\ H_- = F	\right\}$ the set of all pairs of homotopies  that are constant outside of some compact set, and such that $F$ is the left end of $H$.
		
\begin{defin}[{\cite[Definition 3.1]{abbondandolo2019simple}}]
	Let $F:M\times S^1\rightarrow\R$ be any Hamiltonian and let $(H,J)\in \cD(F)$.  Set 
	\begin{equation}
	A(H,J):= \min_{u\in \cM_{(H,J)}} a_F(u),\qquad c_{AHS}(F):=\sup_{(H,J)\in\cD(F)} A(H,J).
	\end{equation}
\end{defin}
In \cite{audin-damian}, Abbondandolo, Haug and Schlenk proved that the functional $c_{AHS}$ is continuous and monotone, and that it takes values in the action spectrum, namely $c_{AHS}(F)\in \spec(F)$. Let us state the result establishing the continuity of $c_{AHS}$:
\begin{claim}[{\cite[Proposition 3.4]{abbondandolo2019simple}}] \label{clm:AHS_continuity}
	For all $F, G\in \cC^\infty(M\times S^1)$, we have
	\begin{equation*}
	\int_{S^1}\min_{x\in M} \left(F(x,t)-G(x,t)\right) dt 
	\leq c_{AHS}(F)-c_{AHS}(G)
	\leq \int_{S^1}\max_{x\in M} \left(F(x,t)-G(x,t)\right) dt .
	\end{equation*}
\end{claim}
In addition, they proved that the action selector takes non-positive values on Hamiltonians supported in incompressible Liouville domains.
\begin{claim}[{\cite[Proposition 5.4]{abbondandolo2019simple}}] \label{clm:c_AHS_leq0}
	If $F$ has support in an incompressible Liouville domain, then $c_{AHS}(F) \leq 0$. In particular, $c_{AHS}(F) = 0$ for all non-negative Hamiltonians which are supported in an	incompressible Liouville domain.
\end{claim}
Using these claims, the barricades construction and ideas from the proof of Proposition~3.3 from \cite{abbondandolo2019simple}, one can prove that a min inequality holds for $c_{AHS}$.

\begin{proof}[Proof of Theorem~\ref{thm:AHS_min_ineq}]
	Let $F$ and $G$ be Hamiltonians supported in disjoint incompressible Liouville domains, which we denote by $U$ and $V$ respectively. Fixing an arbitrarily small $\delta>0$, we will prove that $c_{AHS}(F+G)\leq c_{AHS}(F)+3\delta$. The claim for $G$ will follow by symmetry. 
	We remark that by Claim~\ref{clm:c_AHS_leq0}, $c_{AHS}(F+G)\leq 0$, and hence the result is immediate if $c_{AHS}(F)\geq -3\delta$. Therefore, we assume that $c_{AHS}(F)<-3\delta$.
	We break the proof into several steps.\\
		
	\noindent\underline{Step 1:} Our first step is to perturb $F$ and $F+G$ (as well as a homotopy between them) to create barricades. Let $H$ be a linear homotopy from $F$ to $F+G$ that is constant outside of $[0,1]$, that is, $\partial_s H|_{s\notin [0,1]}=0$. Then, $H$ is supported in the domain $U\cup V$, which, as an incompressible Liouville domain, is also a CIB domain. Applying Theorem~\ref{thm:barricade} to the homotopy $H$ and the domain $U\cup V$, we conclude that there exists a perturbation $h$ of $H$, an almost complex structure $J^\flat$ and subsets $U_\circ\Subset U$, $V_\circ\Subset V$, containing the supports of $F$, $G$ respectively, such that the pairs $(h,J^\flat)$ and $(h_\pm, J^\flat)$ are Floer-regular and have a barricade in $U\cup V$ around $U_\circ \cup V_\circ$. In particular, the ends of $h$ are non-degenerate,  $f:=h_-$ approximates $F$ and $h_+$ approximates $F+G$. By taking $h$ to be close enough to $H$, we can assume that, outside of $U_\circ$, $f$ is a small Morse function with values in $(-\delta, \delta)$. Moreover, by {Remark~\ref{rem:bump_implies_barricade}}, item~\ref{itm:hom_const_outside_of_01}, we can choose the perturbation $h$ such that
	the homotopy $h$ is constant outside of $[0,1]$, namely, $\partial_sh|_{s\notin[0,1]} =0$. Finally, taking these perturbations to be small enough, it follows from Claim~\ref{clm:AHS_continuity} that $c_{AHS}(f)<-2\delta$, and it is sufficient to prove that $c_{AHS}({h_+})\leq c_{AHS}(f) +\delta$.\\
	
	\noindent\underline{Step 2:} Recalling the definition of the action selector $c_{AHS}$, we need to show that for every $(K,J)\in\cD({h_+})$, it holds that $A(K,J)\leq c_{AHS}(f)+\delta$. Therefore, our second step is to construct pairs in $\cD(f)$ out of a given pair in $\cD({h_+})$. Fix $(K,J)\in\cD({h_+})$ and assume, without loss of generality, that $K$ and $J$ stabilize for $s\leq 0$, namely, $K(x,t,s)={h_+}(x,t)$ and $J(s)=J_-$  for $s\leq 0$. We construct a sequence of pairs in $\cD(f)$ by concatenating the homotopies $(K,J)$ with shifts the homotopy $h$ and a homotopy  $\tilde J = \{\tilde J^s\}_{s\in \R}$ of almost complex structures from $J^\flat$ to $J_-$, that is constant outside of $[0,1]$, namely, $\partial_s \tilde J|_{s\notin [0,1]}=0$.  More precisely, for $s\in\R$ denote by $\tau_s$ the shift by $s$, namely, $\tau_{s}h(\cdot,\cdot,\cdot)= h(\cdot,\cdot,\cdot+s)$ and $\tau_s\tilde J(\cdot,\cdot) = \tilde J(\cdot,\cdot+s)$, and consider the sequences
	\begin{figure}
		\centering
		\includegraphics[scale=0.7]{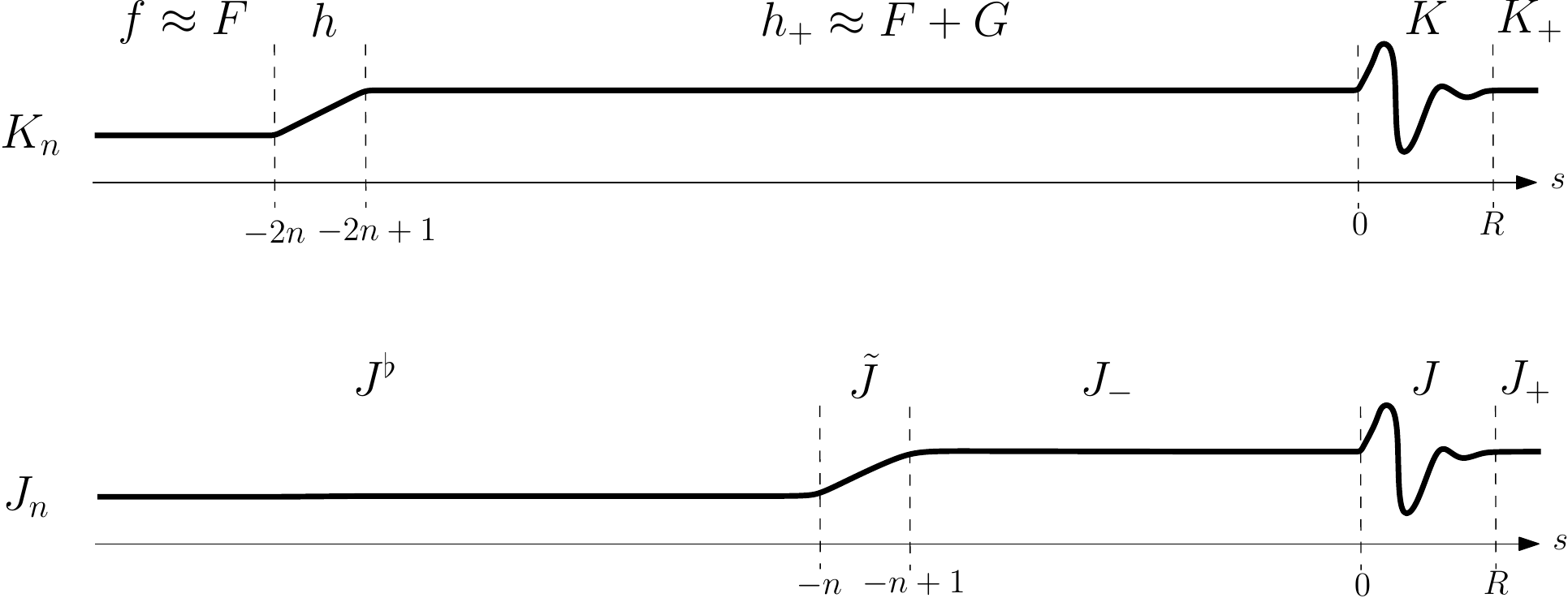}
		\caption{\small{An illustration of the pair of homotopies $(K_n,J_n)\in \cD(f)$ constructed out of a given pair $(K,J)\in \cD({h_+})$. 
			}}
			\label{fig:AHS_Kn_and_Jn}
	\end{figure}
	\begin{equation}
	K_n:=\begin{cases}
	K, & s\geq 0,\\
	{h_+}, & s\in [-2n+1, 0]\\
	\tau_{-2n} h, & s\in [-2n,-2n+1]\\
	f, & s\leq -2n.
	\end{cases} \quad\text{and}\quad	
	J_n:=\begin{cases}
	J, & s\geq 0,\\
	J_-, & s\in [-n+1,0]\\
	\tau_{-n} \tilde J, & s\in [-n,-n+1]\\
	J^\flat, & s\leq -n. 
	\end{cases}
	\end{equation}
	See Figure~\ref{fig:AHS_Kn_and_Jn} for an illustration.
	Noticing that $(K_n, J_n)\in \cD(f)$ for all $n$, we wish to show that there exists $n\in \N$ for which $A(K,J)\leq A(K_n, J_n)+\delta$.\\
	
	\noindent\underline{Step 3:} In this step we choose, for each $n$, a solution minimizing $a_f$ and extract a subsequence that {\it partially converges to a broken trajectory}. Namely, there exists a broken trajectory $\bar v= (v_1,\dots,v_N)$, whose pieces $v_i$ are solutions of (\ref{eq:FE}) with respect to the homotopies concatenated in $(K_n,J_n)$, and are obtained as limits of non-positive shifts of elements from $\{u_n\}$. In particular, for each $i<N$, the solution $v_i$ converges to periodic orbits at the ends, that match the limits of the adjacent pieces, i.e., $\lim_{s\rightarrow+\infty}v_i(s,t)=\lim_{s\rightarrow -\infty}v_{i+1}(s,t)$. Moreover, the left end of the first piece, $\lim_{s\rightarrow -\infty} v_1(s,t)$, coincides with the left end of each element from the subsequence. We stress that unlike the standard convergence to a broken trajectory, in our case, the right end of the last piece in $\bar v$ (as well as the right ends of the solutions $u_n$) does not necessarily converge. 
	The notion of partial convergence to a broken trajectory is defined formally in Proposition~\ref{pro:partial_broken_trajectory} below. 
	
	Let $u_n\in \cM_{(K_n,J_n)}$ be a minimizer of the functional $a_f$, namely 
	\begin{equation*}
	a_f(u_n) = \min_{u\in \cM_{(K_n,J_n)}} a_f(u) = A(K_n, J_n).
	\end{equation*}
	Since the supports of the homotopies $(H_n, J_n)$ are not uniformly bounded and the ends are not all non-degenerate, the sequence of solutions $\{u_n\}_n$ does not necessarily converge to a broken trajectory. However, noticing that for $s\leq 0$, $(H_n, J_n)$ are concatenations of homotopies with non-degenerate ends, one can prove a (weaker) convergence statement, as we do in Section~\ref{app:partial_broken_traj}. In this case,  Proposition~\ref{pro:partial_broken_trajectory} guarantees that there exists a subsequence of $\{u_n\}$, which  we still denote by $\{u_n\}$, partially converging to a broken trajectory 
	\begin{equation*}
	\overline{v} = \left(\{v^{(f,J^\flat),\ell}\}_{\ell=1}^{L_1},\ w^{(h,J^\flat)},\  \{v^{({h_+},J^\flat),\ell}\}_{\ell=1}^{L_2},\  w^{({h_+},\tilde J)},\  \{v^{({h_+},J_-),\ell}\}_{\ell=1}^{L_3},\  w^{(K,J)}\right),
	\end{equation*}
	where $v^{(\cdot,\cdot),\ell}\in \cM_{(\cdot, \cdot)}$ denote solutions of $s$-independent Floer equations, and $w^{(\cdot,\cdot)}\in \cM_{(\cdot, \cdot)}$ denote solutions of $s$-dependent Floer equations. Moreover, the subsequence $\{u_n\}$ is chosen such that for each $n$,   $\lim_{s\rightarrow-\infty} u_n(s,\cdot)=x^{1,0}(\cdot)$, where  $x^{1,0}:=\lim_{s\rightarrow-\infty} v^{(f,J^\flat),1}(s,\cdot) \in \cP(f)$. \\
	
	\noindent\underline{Step 4:} We now use the barricades in order to show that the first few pieces of the broken trajectory $\overline{v}$ are contained in $U_\circ$. It follows from the arguments made above that
	$$
	\cA_f(x^{1,0}) = a_f(u_n)=A(K_n,J_n)\leq c_{AHS}(f) <-2\delta,
	$$
	which implies, by our assumptions on $f$, that $x^{1,0}\subset U_\circ$. We claim that, since $(f,J^\flat)$ and $(h,J^\flat)$ have barricades in $U$ around $U_\circ$, the trajectories $\{v^{(f,J^\flat),\ell}\}_{\ell=1}^{L_1}$ and $w^{(h,J^\flat)}$ are contained in $U_\circ$.
	Indeed,  $\lim_{s\rightarrow-\infty} v^{(f,J^\flat),1}(s,\cdot) = x^{1,0}\subset U_\circ$ implies that $v^{(f,J^\flat),1}\subset U_\circ$ and, in particular, the image of $x^{1,1}(\cdot):=\lim_{s\rightarrow\infty} v^{(f,J^\flat),1}(s,\cdot)$ is contained in $U_0$. Since $x^{1,1}$ is the left end of $v^{(f,J^\flat),2}$, we can repeat this argument and conclude that $v^{(f,J^\flat),2}$ is contained in $U_\circ$. Continuing by induction, we find that	$\{v^{(f,J^\flat),\ell}\}_\ell$ are all contained in $U_\circ$ and, in particular, 
	$$
	x^{1,L_1}:= \lim_{s\rightarrow\infty} v^{(f,J^\flat),L_1}(s,\cdot) = \lim_{s\rightarrow- \infty} w^{(h,J^\flat)}(s,\cdot) \subset U_\circ.
	$$
	Now, since $(h,J^\flat)$ has a barricade in $U$ around $U_\circ$, we conclude that $ w^{(h,J^\flat)}\subset U_\circ$ as well. \\
	
	\noindent\underline{Step 5:} Let us now show that $a_{h_+}(w^{(K,J)})\leq \cA_f(x^{1,0})+\delta= a_f(u_n) +\delta$.
	For that end, we bound the action growth along the broken trajectory $\overline{v}$:
	\begin{enumerate}
		\item Along $v^{(\cdot,\cdot),\ell}$: these are solutions of the $s$-independent Floer equations and, by the energy identity (\ref{eq:energy_id_Hamiltonian}), the action is clearly non-increasing.
		\item Along $w^{(h,J^\flat)}$: this trajectory is contained in $U_\circ$,  where $h$ approximates a constant homotopy, as $F|_U=F+G|_U$. Taking $h$ to be close enough to $H$, we may assume that the derivative $\partial_s h|_{U_\circ}$ is bounded by $\delta$.  Denoting by $x^{1,L_1}\in \cP(f)$ and $x^{2,0}\in \cP({h_+})$ the orbits to which $w^{(h,J^\flat)}$ converges at the ends, it follows from the energy-identity (\ref{eq:energy_id_homotopies}) that
		\begin{eqnarray}\label{eq:action_growth_AHS}
		\nonumber 
		\cA_{h_+}(x^{2,0}) - \cA_f(x^{1,L_1}) &\leq& \Big|\int_{\R\times S^1}(\partial_s h)\circ w^{(h,J^\flat)} \ ds\ dt\Big|\\
		&\leq &\int_{[0,1]\times S^1} \max_{U_\circ} |\partial_s h|\ ds\ dt \leq\delta.
		\end{eqnarray}
		\item Along $w^{({h_+},\tilde J)}$: it follows from the energy-identity (\ref{eq:energy_id_homotopies}) that the action is non-increasing, since $\partial_s {h_+}=0$.
	\end{enumerate} 	
	Overall, we conclude that 
	\begin{eqnarray}
	\nonumber a_{h_+}(w^{(K,J)}) &=& \cA_{h_+}(x^{3,L_3}) \leq \cdots \leq \cA_{h_+}(x^{2,0})  \\
	\nonumber &\overset{(\ref{eq:action_growth_AHS})}{\leq}& \cA_f(x^{1,L_1}) +\delta \leq \cdots \leq \cA_f(x^{1,0}) +\delta = a_f(u_n) +\delta.
\end{eqnarray}
	Since $u_n$ were chosen to be minimizers, $a_f(u_n) = A(K_n,J_n)\leq c_{AHS}(f)$. On the other hand, the fact that $w^{(K,J)}\in \cM_{(K,J)}$ implies that $a_{h_+}(w^{(K,J)})\geq \min_{\cM{(K,J)}}(a_{h_+}) = A(K,J)$. We therefore have proved that for any $(K,J)\in\cD({h_+})$, $A(K,J) \leq c_{AHS}(f) +\delta$, which yields that $c_{AHS}({h_+}) \leq c_{AHS}(f) +\delta$, as required.  
\end{proof}

\section{The required transversality and compactness results.}\label{sec:transversality_compactness}
\subsection{Perturbing homotopies and Hamiltonians to achieve regularity.}\label{app:regular_pair}
Let $(M, \omega)$ be a closed symplectically aspherical manifold. Given a non-degenerate Hamiltonian $H$ and an almost complex structure $J$, we say that a pair $(H,J)$ is {\it Floer-regular} if for every pair of 1-periodic orbits $x_\pm$ of $H_\pm$ and for every $u\in \cM_{(H,J)}(x_-,x_+)$, the differential $(d\cF)_u:W^{1,p}(u^*TM)\rightarrow L^p(u^*TM)$ of the Floer map (see Notations~\ref{not:Transversality_stuff} below), is surjective. In this case,  the space of solutions $\cM_{(H,J)}(x_-,x_+)$ is a smooth manifold of dimension $\mu(x_-)-\mu(x_+)$.
It is well known that for any non-degenerate Hamiltonian $H$ and an almost complex structure  $J$, one can perturb $H$, without changing its periodic orbits, in order to make the pair $(H, J)$ Floer-regular. 
Let us cite a formal statement of this fact. 
\begin{claim}[{\cite[Theorem 5.1]{floer1995transversality}}] \label{clm:app_reg_Hamiltonians}
	Let $H$ be a non-degenerate Hamiltonian and let $J$ be an almost complex structure on $M$, and let $\cC_\varepsilon^\infty(H)$ be the space of perturbation which vanish on $\cP(H)$ up to second order\footnote{This space is endowed with Floer's $\varepsilon$-norm, which is defined below.}. Then, there exist a neighborhood of zero in $\cC_\varepsilon^\infty(H)$, and a  residual set $\hreg$ in this neighborhood, such that for every $h\in \hreg$, the pair $(H+h, J)$ is Floer-regular.
\end{claim}

When $H$ is a homotopy whose ends, $H_\pm$, are Floer-regular with respect to $J$, one can perturb $H$ on a compact set to guarantee that the pair $(H,J)$ is Floer-regular. For the purposes of this paper, we need to control the size of the support of the perturbation. In this section we prove that one can take the support of the perturbation to be any closed interval with non-empty interior. 
Before making a formal claim, let us fix some notations.
Throughout this section, we consider homotopies of Hamiltonians, $H:M\times S^1\times \R\rightarrow\R$, that are constant with respect to the $\R$-coordinate, $s$, outside of a compact set, namely $supp(\partial_s H)\subset M\times S^1\times [-R,R]$ for some $R>0$. We assume that the ends $H_\pm(\cdot, \cdot):=\lim_{s\rightarrow\pm\infty}H(\cdot, \cdot, s)$ are Floer-regular with respect to a fixed almost complex structure $J$.
For a closed finite interval $I\subset \R$ with non-empty interior, we consider the space $\cC_\varepsilon^\infty(I)$ of perturbations with support in $M\times S^1\times I$, whose definition is given in Section~\ref{subsec:C_varepsilon_space} below.
Our main goal for this section is to prove the following proposition.
\begin{prop}\label{pro:pert_hom_bdd_supp}
	Let $H$ be a homotopy such that $(H_\pm, J)$ are Floer-regular, where $J$ is an almost complex structure on $M$, and let $I\subset\R$ be a closed, finite interval with a non-empty interior. Then, there exists a residual subset $\hreg\subset\cC_\varepsilon^\infty(I)$, such that for every $h\in \hreg$, the pair $(H+h, J)$ is Floer-regular.
\end{prop}
The proof of this proposition is postponed to Section~\ref{subsec:pert_homo_proof}. We start by describing the space of perturbations and its relevant properties.

\subsubsection{The Banach space $\cC_\varepsilon^\infty(I)$.} \label{subsec:C_varepsilon_space}
In this section we define the perturbations space $\cC_\varepsilon^\infty(I)$ and prove useful properties. 

\begin{defin}\label{def:epsilon_norm}
	\begin{itemize}
		\item Let $\varepsilon = \{\varepsilon_n\}$ be a sequence of positive numbers. For $h\in \cC^\infty(M\times S^1\times \R)$, Floer's $\varepsilon$-norm is defined to be 
		$$
		\|h\|_\varepsilon:=\sum_{k\geq 0}\varepsilon_k \sup_{M\times S^1\times \R}|d^kh|,
		$$ 
		see \cite[p.230]{audin-damian} for details. For a proof that it is a norm, see \cite[Theorem B.2]{wendl1612lectures}
		
		\item For a closed and finite interval $I\subset\R$ with a non-empty interior, let $\cC_\varepsilon^\infty(I)$ be the space of functions  $h\in \cC^\infty(M\times S^1\times \R)$, supported in $M\times S^1\times I$, whose $\varepsilon$-norm is finite, namely $\|h\|_\varepsilon<\infty$. Then, $\cC_\varepsilon^\infty(I)$ is a Banach space. 
		In what follows we identify between the tangent space   $T_h\cC_\varepsilon^\infty(I)$ at a point $h$, and the space $\cC_\varepsilon^\infty(I)$ itself.
	\end{itemize}
\end{defin}
The following claims guarantee that the properties that are required of a space of perturbations hold for $\cC_\varepsilon^\infty(I)$.
\begin{claim}\label{clm:Cepsilon_dense_in_Cinfty}
	There exists a sequence $\varepsilon$ for which  $\cC_\varepsilon^\infty(I)$ is dense in $\cC^\infty(I)$.
\end{claim}
\begin{claim}\label{clm:C_epsilon_is_separable}
	The Banach space $\cC_\varepsilon^\infty(I)$ is separable. 
\end{claim}

In order to prove these claims we first state and prove two lemmas. We use notations and ideas from \cite[Section 8.3]{audin-damian} and \cite[Appendix B]{wendl1612lectures}.
\begin{lemma}\label{lem:C0_I_E_sec_count}
	Let $E$ be a finite dimensional vector bundle over $M\times S^1\times \R$, then, the space $\cC^0_I(E)$ of continuous sections of $E$ that are supported in $M\times S^1\times I$ is second countable with respect to the uniform norm.
\end{lemma}
\begin{proof}
	Embedding $M\times S^1\times I$ into $[-N,N]^m$ for some large $N, m$, the space $\cC^0_I(E)$ is isometrically embedded into $\cC^0([-N,N]^m;\R^k)$ for some $k\in\N$\footnote{This uses the fact that every vector bundle over a compact base is a sub-bundle of a trivial vector bundle, see \cite[Proposition 1.4]{hatcher2003vector}.}. By Weierstrass approximation theorem, the latter space is separable, and hence (being a normed space) is also second countable. We conclude that the same holds for the closed subspace $\cC^0_I(E)$.	
\end{proof}

Following \cite[Appendix B]{wendl1612lectures}, set $E^{(0)}:=E$ and $E^{(k+1)}:= \operatorname{Hom}\left(T(M\times S^1\times \R); E^{(k)}\right)$, then, fixing connections and bundle metrics on both $T(M\times S^1\times \R)$ and $E$, any section $\eta\in\Gamma(E^{(k)})$ has a covariant derivative $\nabla\eta\in \Gamma(E^{(k+1)})$. Set $F^{(k)}:=E^{(0)}\oplus\cdots\oplus E^{(k)}$ and consider the countable product $\prod_{k\in\N} \cC^0_I(F^{(k)})$, endowed with the product topology. By Lemma~\ref{lem:C0_I_E_sec_count}, each factor is second countable and therefore so is the product. 

\begin{lemma}\label{lem:C_infty_I_separable}
	The space $\cC^\infty(I)$ of smooth functions $M\times S^1\times \R\rightarrow\R$ supported on $M\times S^1\times I$, is separable with respect to the  $\cC^\infty$-topology.
\end{lemma}
\begin{proof}
	The space $\cC^\infty(I)$ can be embedded into the product $\prod_{k\in\N} \cC^0_I(F^{(k)})$, by 
	$$
	\eta\mapsto(\eta, (\eta,\nabla\eta), (\eta, \nabla\eta, \nabla^2\eta),\dots).
	$$
	As explained above, the product  $\prod_{k\in\N} \cC^0_I(F^{(k)})$ is second countable and hence so is any closed subspace of it. In particular, $\cC^\infty(I)$ is separable.
\end{proof}
We are now ready to prove Claim~\ref{clm:Cepsilon_dense_in_Cinfty}. The proof is exactly that of \cite[Proposition 8.3.1]{audin-damian}.
\begin{proof}[Proof of Claim~\ref{clm:Cepsilon_dense_in_Cinfty}]
	Let $f_n\in \cC^\infty(I)$ be a dense sequence, whose existence is guaranteed by Lemma~\ref{lem:C_infty_I_separable}. Let 
	\begin{equation*}
	\varepsilon_n :=1/\left(2^n \cdot \max_{k\leq n}\|f_k\|_{\cC^n(M\times S^1\times \R)}\right).
	\end{equation*}  
	For this choice of a sequence $\varepsilon$, it holds that $\|f_n\|_\varepsilon<\infty$ for all $n$, namely, $f_n\in \cC_\varepsilon^\infty(I)$.
\end{proof}
The proof of Claim~\ref{clm:C_epsilon_is_separable} is essentially that of Lemma~B.4 and Theorem~B.5 from \cite{wendl1612lectures}, we include it for the convenience of the reader.
\begin{proof}[Proof of Claim~\ref{clm:C_epsilon_is_separable}]
	Consider again the product $\prod_{k\in\N} \cC^0_I(F^{(k)})$  and let $X_\varepsilon$ be the space of sequences $\xi:=(\xi^0, \xi^1, \xi^2,\dots)\in \prod_{k\in\N} \cC^0_I(F^{(k)})$ such that 
	\begin{equation*}
	\|\xi\|_{X_\varepsilon} := \sum_{k=0}^\infty \varepsilon_k \cdot \|\xi^k\|_{\cC^0}<\infty.
	\end{equation*}
	We will first show that $X_\varepsilon$ is separable and then embed $\cC_\varepsilon^\infty(I)$ into $X_\varepsilon$ in order to prove the claim. Indeed, since $\cC^0_I(F^{(k)})$ is separable for each $k$ (by Lemma~\ref{lem:C0_I_E_sec_count}), we can fix a dense countable subset $P^k\subset \cC^0_I(F^{(k)})$. The set 
	\begin{equation*}
	P:= \left\{(\xi^0,\dots,\xi^N,0,0\dots)\in X_\varepsilon\ |\ N\geq 0 \text{ and } \forall\ 0\leq k\leq N,\ \xi^k\in P^k \right\}
	\end{equation*}
	is countable and dense in $X_\varepsilon$. Now consider the injective linear map 
	\begin{equation*}
	\cC_\varepsilon^\infty(I) \hookrightarrow X_\varepsilon \ :\ 
	\eta\mapsto (\eta, (\eta,\nabla\eta), (\eta, \nabla\eta, \nabla^2\eta),\dots).
	\end{equation*}
	It is an isometric embedding, and hence we may view $\cC_\varepsilon^\infty(I)$ as a closed subspace of the separable space $X_\varepsilon$. The latter is also second countable (being a normed space) and hence so is $\cC_\varepsilon^\infty(I)$. 
\end{proof}

\begin{rem}\label{rem:C_epsilon_separable} 
	The proof of Claim~\ref{clm:C_epsilon_is_separable} shows that spaces of perturbations with compact support are separable in general. This observation will be used in Section~\ref{app:Hamiltonian_reg_in_U}.
\end{rem}

\subsubsection{Proof of Proposition~\ref{pro:pert_hom_bdd_supp}.} \label{subsec:pert_homo_proof}
We follow the proofs from Chapters 8 and 11 of \cite{audin-damian} and make the necessary changes. 
Let us start by recalling the relevant notations.
\begin{notation}\label{not:Transversality_stuff}
	Let $H$ be a homotopy, let $J$ be an almost complex structure, and let $x_\pm$ be 1-periodic orbits of $H_\pm$ respectively.
	\begin{itemize}
		\item  We denote by $\cM_{(H,J)}(x_-, x_+)$ the set of solutions of the ($s$-dependent) Floer equation with respect to $H, J$ that converge to $x_\pm$ at the ends. We denote by $\cM_{(H,J)}$ the set of all finite energy solutions.
		
		\item (\cite[Def. 8.2.2]{audin-damian}) Denote by $\cP(x_-,x_+)$ the space of maps $\R\times S^1\rightarrow M$, of the form
		\begin{equation*}
		(s,t)\mapsto \exp_{w(s,t)}Y(s,t),
		\end{equation*}
		for $Y\in W^{1,p}(w^*TM)$ and $w\in \cC^\infty_{\searrow}(x_-,x_+)$. The latter is the space of smooth maps $\R\times S^1\rightarrow M$ converging to $x_\pm$ at the ends with exponentially decaying derivatives. We denote by $\cL^p(x_-,x_+)$ the fiber bundle over $\cP(x_-,x_+)$ whose fiber at $u$ is $L^p(u^*TM)$.
		
		\item The Floer map with respect to $H$ is
		\begin{eqnarray}
		\cF^H:\cP(x_-,x_+) &\rightarrow& \cL^p(x_-,x_+)\\
		\nonumber u &\mapsto& \frac{\partial u}{\partial s}+ J(\frac{\partial u}{\partial t} - X_{H}\circ u) = \frac{\partial u}{\partial s}+ J\frac{\partial u}{\partial t} +\grad_u H,
		\end{eqnarray}
		where $(\grad_u H)(s,t)$ is the gradient of $H(\cdot, t,s)$ with respect to $J$, restricted to $u$. In unitary (i.e., symplectic, orthonormal) coordinates, the differential of the Floer map, 
		$(d\cF^H)_u:W^{1,p}(\R\times S^1;u^*TM) \rightarrow L^p(\R\times S^1;u^*TM)$, can be written as 
		$(d\cF)_u (Y) = \overline{\partial}Y+SY$, where $S:\R\times S^1\rightarrow End(\R^{2n})$, see \cite[Section 8.4 and p.389]{audin-damian}.

		\item Set
		\begin{equation}
		\cZ(x_-,x_+):=\left\{(u,h)\in\cP(x_-,x_+)\times\cC_\varepsilon^\infty(I)\ |\ u\in\cM_{(H+h,J)}(x_-,x_+) \right\}
		\end{equation}
	\end{itemize}
\end{notation}

The main ingredients in the proof of Proposition~\ref{pro:pert_hom_bdd_supp} are the following two lemmas. 
\begin{lemma}\label{lem:Z_is_Banach}
	The set $\cZ(x_-,x_+)$ is a Banach manifold.
\end{lemma}
\begin{lemma}\label{lem:pi_is_Fredholm}
	The projection $\pi:\cZ(x_-,x_+)\rightarrow\cC_\varepsilon^\infty(I)$, $(u,h)\mapsto h$, is a Fredholm map.
\end{lemma}
The outline of the proof is as follows: We first prove that the set $\cZ(x_-,x_+)$ is a Banach manifold (Lemma \ref{lem:Z_is_Banach}), and then we show that the projection $\pi:\cZ(x_-, x_+)\rightarrow\cC_\varepsilon^\infty(I)$ is a Fredholm map (Lemma~\ref{lem:pi_is_Fredholm}). Taking $\hreg$ to be the set of regular values of $\pi$, the Sard-Smale theorem guarantees that it is a residual set. 
We will use the following claim from \cite{audin-damian}.
\begin{claim}[{\cite[Theorem 11.1.7]{audin-damian}}] \label{clm:dF_is_Fredholm}
	For every homotopy $H$ such that $(H_\pm, J)$ are Floer-regular and every $u\in \cM_{(H,J)}(x_-,x_+)$, the differential of the Floer map, $(d\cF^H)_u$, at $u$, is a Fredholm operator of index $\mu(x_-)-\mu(x_+)$.
\end{claim}

In order to prove Lemma~\ref{lem:Z_is_Banach}, we present $\cZ(x_-,x_+)$ as an intersection of a certain section with the zero section in a certain vector bundle. The following lemma will be used to guarantee that this intersection is transversal. Its proof, which is a combination of the proofs of \cite[Propositions 8.1.4, 11.1.8]{audin-damian}, contains the main difference between the proof of Proposition~\ref{pro:pert_hom_bdd_supp} and that of \cite[Theorem 11.1.6]{audin-damian}. 
\begin{lemma}\label{lem:Gamma_is_surj}
	For $(u,h)\in \cZ(x_-,x_+)$, the linear operator 
	\begin{eqnarray}
	\Gamma: W^{1,p}(\R\times S^1;\R^{2n})\times\cC_\varepsilon^\infty(I) &\rightarrow& L^p(\R\times S^1;\R^{2n})\\
	\nonumber (Y,\eta) &\mapsto& (d\cF^{H+h})_u(Y)+\grad_u \eta
	\end{eqnarray}
	is surjective and has a continuous right inverse.
\end{lemma}

\begin{proof}
	Assume for the sake of contradiction that $\Gamma$ is not surjective. By \cite[Lemma 8.5.1]{audin-damian}\footnote{This lemma is formulated for a slightly different space, but its proof applies to our case as it is.}, there exists a non-zero vector field $Z\in L^q(\R\times S^1;\R^{2n})$ (here $\frac{1}{p}+\frac{1}{q}=1$), of class $\cC^\infty$, such that for every $Y\in W^{1,p}(\R\times S^1;\R^{2n})$ and $\eta\in \cC_\varepsilon^\infty(I)$, 
	\begin{eqnarray}
	\left<Z,(d\cF^{H+h})_u (Y)\right> &=&0,\label{eq:ortho_to_dF}\\
	\left<Z,\grad_u\eta \right> &=&0,\label{eq:ortho_to_grad_eta}
	\end{eqnarray}
	where $\left<\cdot,\cdot\right>$ denotes the pairing of $L^q$ and $L^p$. 
	As mentioned above, the differential of the Floer map can be written in unitary coordinates as $\overline{\partial}+S(s,t)$. Since $Z$ is of class $\cC^\infty$, it follows from (\ref{eq:ortho_to_dF}) that $Z$ is a zero of the dual operator of $(d\cF)_u$, which is of a ``perturbed Cauchy-Riemann"-type. The continuation principle (\cite[Proposition 8.6.6]{audin-damian}) now implies that if $Z$ has an infinite-order zero, then it is identically zero, $Z\equiv 0$.
	
	Therefore, let us show that (\ref{eq:ortho_to_grad_eta}) guarantees that $Z$ vanishes on {$I\times S^1$}, and conclude that it vanishes identically, since we assumed that the interior of $I$ is not empty. The proof is roughly the same as that of \cite[Lemma 11.1.9]{audin-damian}, but we include it for the sake of completeness. An equivalent reformulation of (\ref{eq:ortho_to_grad_eta}) is:
	\begin{equation*}
	\int_{\R\times S^1}d\eta (Z)\ ds\ dt= 0\ \text{  for every } \eta\in\cC_\varepsilon^\infty(I).
	\end{equation*}
	Consider the map $\tilde u:\R\times S^1\rightarrow M\times \R\times S^1$ defined by $(s,t)\mapsto\left(u(s,t),s,t\right)$. It is easy to see that $\tilde u$ is an embedding. Viewing $Z$ as a vector field along $\tilde u$ on $M\times \R\times S^1$ that does not have components in the directions $\partial/\partial t\in TS^1$ and $\partial/\partial s\in T\R$, we see that it is not tangent to $\tilde u$ at the points where it is not zero. Assume for the sake of contradiction that there exists a point $(s_0,t_0)\in I\times S^1$ at which $Z$ does not vanish. Since $Z$ is continuous, there exists a small neighborhood $C_\delta$ of $(s_0, t_0)$, in which $Z(s,t)$ does not vanish and therefore is transversal to $\tilde u$ for all $(s,t)\in C_\delta$. Notice that if $(s_0,t_0)$ is not in the interior of $I\times S^1$, we may replace it with a point in $C_\delta\cap (int(I)\times S^1)$, and then replace $C_\delta$ by a smaller neighborhood that is contained in $int(I)\times S^1$. Therefore we assume, without loss of generality, that $C_\delta\subset int(I)\times S^1$.  Let $\beta:\R\times S^1\rightarrow\R$ be a smooth function supported in $C_\delta$, whose integral is not zero, $\int_{\R\times S^1}\beta(s,t)\ ds\ dt\neq 0$. Define $\eta:M\times S^1\times \R\rightarrow \R$ with support in a tubular neighborhood $B$ of $\tilde u (C_\delta)$ in such a way that if $\gamma_{(s,t)}(\sigma)$ is a parametrized integral curve of $Z$ passing through $\tilde u(s,t)$ at $\sigma=0$, then 
	$$
	\eta(\gamma_{(s,t)}(\sigma),t,s) := \beta(s,t)\cdot \sigma,\quad \text{ for } |\sigma|\leq\epsilon.
	$$ 
	The fact that $Z$ is transversal to $\tilde u(C_\delta)$ guarantees that  $\eta$ is well defined. We also assume that $B\cap\im (\tilde u)= \tilde u(C_\delta)$, which means that $supp(\eta)\cap\im(\tilde u)\subset \tilde u(C_\delta)$. Let us compute the integral of $d\eta(Z)$:
	\begin{eqnarray}
	\nonumber\int_{\R\times S^1} d\eta_{s,t}(Z(s,t))\ ds\ dt &=& \int_{C_\delta} d\eta_{s,t}(Z(s,t))\ ds\ dt \\
	\nonumber &=&\int_{C_\delta} d\eta_{s,t}\left(\frac{\partial \gamma_{s,t}(\sigma)}{\partial \sigma}\Big|_{\sigma=0}\right)\ ds\ dt\\
	\nonumber &=&\int_{C_\delta} \frac{\partial}{\partial \sigma}\left(\eta(\gamma_{s,t}(\sigma),t,s)\right)\Big|_{\sigma=0}\ ds\ dt\\	
	\nonumber &=&\int_{C_\delta} \frac{\partial}{\partial \sigma}\left(\beta(s,t)\cdot \sigma\right)\Big|_{\sigma=0}\ ds\ dt\\	
	\nonumber &=&\int_{C_\delta} \beta(s,t)\ ds\ dt.
	\end{eqnarray}
	As we chose $\beta$ to be a function with a non-vanishing integral, we find that (\ref{eq:ortho_to_grad_eta}) does not hold for the function $\eta$ constructed above. Note that $\eta$ is a smooth function, supported in $M\times S^1\times I$, but its $\varepsilon$-norm is not necessarily finite. Therefore, to arrive at a contradiction, it remains to approximate $\eta$ by $\eta'\in \cC_\varepsilon^\infty(I)$. This is possible due to Claim~\ref{clm:Cepsilon_dense_in_Cinfty}. When $\eta'$ is close to $\eta$, the integral of $d\eta'(Z)$ will be close to that of $d\eta(Z)$ (since their supports are contained in the compact set $M\times S^1\times I$), and hence equality (\ref{eq:ortho_to_grad_eta}) will not hold for $\eta'\in\cC_\varepsilon^\infty(I)$, in contradiction.
	
	This shows that $\Gamma$ is surjective. The fact that it has a continuous right inverse follows from \cite[Lemma~8.5.6]{audin-damian} and Claim~\ref{clm:dF_is_Fredholm}. 
\end{proof}

Having Lemma~\ref{lem:Gamma_is_surj}, the proof of Lemma~\ref{lem:Z_is_Banach}, which asserts that $\cZ(x_-,x_+)$ is a Banach manifold, is precisely that of \cite[Proposition 8.1.3]{audin-damian}:
\begin{proof}[Proof of Lemma~\ref{lem:Z_is_Banach}]
	Let $\cE:=\{(u,h,Y)\ |\ Y\in L^p(u^*TM)\}$ be a vector bundle over $\cP(x_-,x_+)\times \cC_\varepsilon^\infty(I)$, and consider the section induced by $\cF^{H+h}$:
	\begin{eqnarray}
	\nonumber \sigma:\cP(x_-,x_+)\times \cC_\varepsilon^\infty(I) &\rightarrow& \cE,\\
	\nonumber (u,h) &\mapsto& \left(u,h,\frac{\partial u}{\partial s}+ J\frac{\partial u}{\partial t} +\grad_u (H+h)\right).
	\end{eqnarray}
	Notice that the space $\cZ(x_-, x_+)$ is the intersection of $\sigma$ with the zero section in $\cE$. Therefore, in order to prove that $\cZ(x_-,x_+)$ is a Banach manifold, it is sufficient to show that $\sigma$ intersects the zero section transversally, or, equivalently, that $d\sigma$ composed with the projection onto the fiber is surjective and has a right inverse, at all points for which $\sigma(u,h)=0$. But, this composition is precisely the operator $\Gamma$ whose surjectivity and right-invertability are guaranteed by Lemma~\ref{lem:Gamma_is_surj}. 
\end{proof}

Our next goal is to show that $\pi$ is a Fredholm map, that is, to prove Lemma~\ref{lem:pi_is_Fredholm}.
\begin{proof}[Proof of Lemma~\ref{lem:pi_is_Fredholm}] 
	The projection $\pi:\cZ(x_-,x_+)\rightarrow \cC_\varepsilon^\infty(I)$, $\pi(u,h)=h$, is clearly smooth. Let us show that its differential, $d\pi$, has a finite dimensional  kernel and a closed image of finite co-dimension.
	\begin{itemize}
		\item $\ker (d\pi)_{(u,h)} = \left\{(Y,0)\in T_{(u,h)}\cZ(x_-,x_+)\right\}$. The tangent space of $\cZ(x_-,x_+)$ is
		\begin{equation*}
		T_{(u,h)}\cZ(x_-,x_+) = \left\{(Y,\eta)\ |\ (d\cF^{H+h})_u(Y)+\grad_u \eta =0	\right\},
		\end{equation*}
		and therefore, the kernel of $(d\pi)_{(u,h)}$ agrees with the kernel of $(d\cF^{H+h})_u$, which, is finite dimensional by Claim~\ref{clm:dF_is_Fredholm}.
		\item $\im (d\pi)_{(u,h)}=\left\{ \eta\ |\ \exists Y\in W^{1,p}(\R\times S^1;u^*TM),\ \grad_u\eta= -(d\cF^{H+h})_u(Y) \right\}$. Consider the linear map $G:\cC_\varepsilon^\infty(I)\rightarrow L^p(\R\times S^1;u^*TM)$, defined by $G(\eta)=\grad_u \eta$, then,
		\begin{equation}\label{eq:Im_dpi_vs_dF}
		\im (d\pi)_{(u,h)}=\left\{ \eta\ |\  \grad_u\eta\in \im (d\cF^{H+h})_u \right\} = G^{-1}\left(\im (d\cF^{H+h})_u\right).
		\end{equation}
		By Claim~\ref{clm:dF_is_Fredholm}, the image of $(d\cF^{H+h})_u$ is closed and of finite co-dimension. Let us show that the same hold for the image of $(d\pi)_{(u,h)}$. Consider the map induced by $G$ on the quotients,
		\begin{equation*}
		A:=\frac{\cC_\varepsilon^\infty(I)}{\im(d\pi)_{(u,h)}} \overset{G'}{\longrightarrow} B:=\frac{L^p(\R\times S^1;u^*TM)}{\im(d\cF^{H+h})_u},
		\end{equation*}
		which is well defined due to (\ref{eq:Im_dpi_vs_dF}). It is easy to see that $G'$ is injective and, together with the fact that $B$ is finite dimensional, this yields that $\operatorname{codim}(\im(d\pi)_{(u,h)}) = \dim(A)$ is finite. This now implies that the image of $(d\pi)_{(u,h)}$ is also closed and hence $(d\pi)_{(u,h)}$ is a Fredholm operator.
	\end{itemize}
\end{proof}
Having proved Lemmas \ref{lem:Z_is_Banach} and \ref{lem:pi_is_Fredholm}, we are ready to prove the main proposition. 
\begin{proof}[Proof of Proposition~\ref{pro:pert_hom_bdd_supp}]
	By Lemma~\ref{lem:pi_is_Fredholm}, the projection $\pi:\cZ(x_-,x_+)\rightarrow \cC_\varepsilon^\infty(I)$ is a (smooth) Fredholm map. By Claim~\ref{clm:C_epsilon_is_separable}, the space $\cC_\varepsilon^\infty(I)$ is separable. To see that $\cZ(x_-,x_+)$ is a separable Banach manifold, recall that it is modeled over a subspace of the Banach space $W^{1,p}(\R\times S^1;\R^{2n})\times \cC_\varepsilon^\infty(I)$. The latter is a separable metric space, and therefore second-countable. As any subspace of a second-countable space is also second-countable, and, in particular, separable, we conclude that $\cZ(x_-,x_+)$ is separable. It follows that we may apply Sard-Smale's theorem to $\pi$ and conclude that the set of regular values of $\pi$ is a countable intersection of open dense sets in $\cC_\varepsilon^\infty(I)$.
	The set $\hreg\subset \cC_\varepsilon^\infty(I)$ is defined to be the intersection of the regular values of the projections for all choices of 1-periodic orbits, $x_\pm$.
	
	Let us show that for each $h\in\hreg$, the pair $(H+h,J)$ is Floer-regular. 
	Fix 1-periodic orbits $x_\pm$, then  $h$ is a regular value of the projection $\pi:\cZ(x_-,x_+)\rightarrow\cC_\varepsilon^\infty(I)$. Let us show that for every $u\in\cM_{(H+h,J)}(x_-,x_+)$, the differential of the Floer map, $(d\cF^{H+h})_u$, is surjective. Indeed, otherwise, arguing as in the proof of Lemma~\ref{lem:Gamma_is_surj}, there exists $Z\in L^q(\R\times S^1;\R^{2n})$, where $\frac{1}{p}+\frac{1}{q}=1$, such that $\left<Z,(d\cF^{H+h})_u(Y)\right>=0$ for all $Y$. Since $(d\pi)_{(u,h)}$ is surjective, for every $\eta\in\cC_\varepsilon^\infty(I)$, there exists $Y$ such that $\grad_u \eta = -(d\cF^{H+h})_u(Y)$, and hence $\left<Z,\grad_u \eta\right>=0$ as well. We conclude that $Z$ satisfies both equations (\ref{eq:ortho_to_dF}) and (\ref{eq:ortho_to_grad_eta}), and, proceeding as in the proof of Lemma~\ref{lem:Gamma_is_surj}, we find $Z=0$. Thus $(d\cF^{H+h})_u$ is indeed surjective. 
	
	It remains to show that $\cM_{(H+h,J)}(x_-,x_+)$ is a smooth manifold of the correct dimension. The inverse image $\pi^{-1}(h)$ is the space of maps $u\in \cP(x_-,x_+)$, of class $W^{1,p}$, that are solutions of the Floer equation, $\cF^{H+h}(u)=0$. By elliptic regularity, these solutions are all smooth, and hence $\pi^{-1}(h) = \cM_{(H+h,J)}(x_-,x_+)$. Since $h$ is a regular value of $\pi$, we therefore conclude that $\cM_{(H+h,J)}(x_-,x_+)$ is indeed a smooth manifold. Its dimension is 
	\begin{equation*}
	\dim\ker (d\pi)_{(u,h)} = \dim \ker(d\cF^{H+h})_u = \ind (d\cF^{H+h})_u = \mu(x_-) -\mu(x_+),
	\end{equation*}
	where the last equality follows from Claim~\ref{clm:dF_is_Fredholm} above.
\end{proof}

\subsection{Convergence to broken trajectories.}\label{app:broken_traj}
A well known phenomenon in Floer theory on symplectically aspherical manifolds is the convergence of sequences of solutions to a {\it broken trajectory}. In this section we formulate and prove results of this sort for the settings that are considered throughout the paper. 

\subsubsection{Convergence for homotopies with non-degenerate ends.} \label{app:broken_traj_non_deg}
In what follows we consider homotopies with non-degenerate ends. We remark that  the same arguments apply for non-degenerate Hamiltonians, when one considers them as constant homotopies. 
Let $H$ be a homotopy that is constant outside of $M\times S^1\times [-R,R]$ for some fixed $R>0$, namely, $\partial_s H|_{|s|>R}=0$. Let $H_n$ be a sequence of homotopies converging $H$, such that for each $n$, 
\begin{equation}\label{eq:sequence_supp_and_cP}
supp(\partial_sH_n)\subset M\times S^1\times [-R,R]\ \text{ and }\ \cP(H_{n\pm}) = \cP(H_\pm).
\end{equation}
Recall that $\cM_{(H,J)}$ denotes the set of finite energy solutions of the Floer equation (\ref{eq:FE}) with respect to $H$ and $J$; for $x_\pm\in\cP(H_\pm)$, we denote by $\cM_{(H,J)}(x_-,x_+)\subset\cM_{(H,J)}$ the subset of solutions connecting $x_\pm$. 
Let 
$$
\cM(x_-,x_+):=\bigcup_n\cM_{(H_n,J)}(x_-,x_+)\ \cup\ \cM_{(H,J)}(x_-,x_+)
$$ 
be the space of all finite-energy solutions connecting $x_\pm$ with respect to $(H,J)$ and $(H_n, J)$ for all $n$, and set $\cM:=\cup_{x_\pm\in \cP(H_\pm)}\cM(x_-,x_+)$. 
The following proposition is an adjustment of \cite[Theorems 11.1.10, 11.3.10]{audin-damian} to our case.
\begin{prop}\label{pro:appendix_broken_traj}
	Let $H$ be a homotopy with non-degenerate ends, and let $H_n$ be a sequence converging to	 $H$ in $ \cC^\infty(M\times S^1\times \R)$ that satisfies (\ref{eq:sequence_supp_and_cP}) for each $n$. Given a sequence $u_n\in \cM_{(H_n,J)}(x_-,x_+)$ of solutions and a sequence of real numbers $\{\sigma_n\}$, there exist:
	\begin{itemize}
		\item Subsequences of $\{u_n\}$ and $\{\sigma_n\}$, which we still denote by $\{u_n\}$ and $\{\sigma_n\}$,
		\item periodic orbits $x_-=x_0,x_1,\dots,x_k\in \cP(H_-)$ and $y_0,y_1,\dots,y_\ell=x_+\in \cP(H_+)$,
		\item sequences of real numbers $\{s_n^i\}_{n}$ for $1\leq i\leq k$ and $\{s'^j_n\}_{n}$ for $1\leq j\leq \ell$,
		\item solutions $v_i\in \cM_{(H_-,J)}(x_{i-1},x_i)$, $1\leq i\leq k$ and $v'_j\in \cM_{(H_+,J)}(y_{j-1},y_j)$, $1\leq j\leq \ell$, 
		\item a solution $w\in\cM_{(H,J)}(x_k, y_0)$
	\end{itemize}
	such that, in $\cC^\infty_{\text{loc}}(\R\times S^1;M)$, for $1\leq i\leq k$ and $1\leq j\leq \ell$,
	\begin{equation*}
	\lim_{n\rightarrow\infty} u_n(\cdot +s_n^i,\cdot) =v_i,\quad
	\lim_{n\rightarrow\infty} u_n(\cdot +s'^j_n,\cdot) =v'_j,\quad \lim_{n\rightarrow \infty} u_n= w,
	\end{equation*}
	and the sequence $ u_n(\cdot +\sigma_n,\cdot)$ converges to one of $v_i, w, v_j'$, perhaps up to a shift in the $s$-coordinate.
\end{prop}
The finite sequence $(v_1,\dots,v_k,w,v'_1,\dots,v'_\ell)$ is called a broken trajectory of $(H,J)$. 
Before proving the above proposition, we state and prove two lemmas. The first is an analogous statement to \cite[Theorem 11.2.7]{audin-damian}, and gives a uniform bound for the $J$-gradient of a solution $u$ of the Floer equation with respect to $(H,J)$ or $(H_n,J)$. 
\begin{lemma}\label{lem:app_bdd_gradient}
	There exists a constant $A>0$ such that for every $u\in \cM$ and every $(s,t)\in \R\times S^1$, 
	$$
	\big\|\frac{\partial u}{\partial s}\big\|_J^2+\big\|\frac{\partial u}{\partial t}\big\|_J^2\leq A.
	$$
	
\end{lemma}
\begin{proof}
	For convenience we set $H_0:=H$. Let $x_\pm\in \cP(H_\pm)$ be periodic orbits such that $u\in \cM(x_-,x_+)$, then, by the energy identity (\ref{eq:energy_id_homotopies}),
	\begin{equation}\label{eq:unif_bnd_energy}
	E(u)\leq \cA_{H_-}(x_-)-\cA_{H_+}(x_+) +2R\cdot C',
	\end{equation}
	where
	\begin{equation*}
	C':= \sup\left\{\frac{\partial H_n}{\partial s}(x,t,s)\ |\ (x,t,s)\in M\times S^1\times \R,\  n\geq 0\right\},
	\end{equation*}
	and $R>0$ is the constant from (\ref{eq:sequence_supp_and_cP}). 
	The fact that $C'$ is finite follows from the uniform convergence (with derivatives) of $H_n$ to $H_0=H$. Setting 
	$$
	C:=\max_{x_\pm\in\cP(H_\pm)}\left(\cA_{H_-}(x_-)-\cA_{H_+}(x_+) \right) +2R\cdot C',
	$$
	we obtain a uniform bound for the energy, $E(u)\leq C$, for all $u\in \cM$. As in \cite[Propositions 6.6.2, 11.1.5]{audin-damian}, we conclude that there exists $A>0$ such that $\big\|\frac{\partial u}{\partial s}\big\|_J^2+\big\|\frac{\partial u}{\partial t}\big\|_J^2\leq A$. 
\end{proof}

The next lemma uses Arzel\'a-Ascoli theorem and elliptic regularity to show that every sequence of shifted solutions has a converging subsequence. It is an adjustment of Theorem 11.3.7 and Lemma 11.3.9 from \cite{audin-damian} to our setting.
\begin{lemma}\label{lem:app_shifts_have_conv_subseq}
	Let $u_n\in \cM_{(H_n,J)}(x_-, x_+)$ be a sequence of solutions and let $s_n\in\R$ be a sequence of numbers. Then, the sequence of shifted solutions $\tau_{s_n}u_n(\cdot, \cdot) = u_n(\cdot+s_n, \cdot)$ has a subsequence that converges in the $\cC^{\infty}_{{loc}}$ topology to a limit $v$. Moreover:
	\begin{enumerate}
		\item If $s_n\rightarrow\sigma \in \R$, then $v\in \cM_{(\tau_\sigma H,J)}$, where $\tau_\sigma H(x,t,s):= H(x,t,s+\sigma)$.
		\item If $s_n\rightarrow -\infty$, then  $v\in \cM_{(H_-,J)}$.
		\item If $s_n\rightarrow +\infty$, then  $v\in \cM_{(H_+,J)}$.
	\end{enumerate}
\end{lemma}
\begin{proof}
	Lemma~\ref{lem:app_bdd_gradient} implies that the sequence $v_n:=\tau_{s_n}u_n$ is equicontinuous. By Arzel\'a-Ascoli theorem and elliptic regularity (see \cite[Lemma 12.1.1]{audin-damian}), there exists a  subsequence, which we still denote by $\{v_n\}$, that converges to a limit $v$ in the $\cC^\infty_{loc}$ topology. The fact that the energy of $v$ is finite follows from the uniform bound (\ref{eq:unif_bnd_energy}) on the energies of $u_n$. It remains to show that the limit $v$ is a solution of the corresponding equation, for the above choices of shifts $s_n$. For each $n$,  $v_n$ is a solution of the equation
	\begin{eqnarray}
	\nonumber
	0&=&\frac{\partial v_n}{\partial s}+ J\frac{\partial v_n}{\partial t} + \grad_{v_n}(\tau_{s_n} H_n)\\
	\nonumber
	&=& \left(\frac{\partial v_n}{\partial s}+ J\frac{\partial v_n}{\partial t} + \grad_{v_n} (\tau_{s_n} H)\right)
	+ \grad_{v_n} (\tau_{s_n} (H_n-H)).
	\end{eqnarray}
	Since the sequence $H_n$ converges to $H$ uniformly with the derivatives, for every $\epsilon>0$ there exists $N$ such that for $n\geq N$,
	\begin{equation}\label{eq:app_almost_Floer_sol}
	\Big\|\frac{\partial v_n}{\partial s}+ J\frac{\partial v_n}{\partial t} + \grad_{v_n} (\tau_{s_n} H)\Big\| < \epsilon.
	\end{equation}	
	Let us split into cases:
	\begin{enumerate}
		\item Assume $s_n\rightarrow\sigma\in \R$. 
		Fix an arbitrarily large $r>|\sigma|$, then the derivatives of $H$ are uniformly continuous on the compact set $M\times S^1\times [-r,r]$.
		Using (\ref{eq:app_almost_Floer_sol}) together with our assumption that $s_n\rightarrow\sigma$, we have
		\begin{equation*}
		\max_{[-r,r]\times S^1}\Big|\frac{\partial v_n}{\partial s}+ J\frac{\partial v_n}{\partial t} + \grad_{v_n} (\tau_{\sigma} H)\Big|
		<\epsilon+ \max_{[-r,r]\times S^1}\Big| \grad_{v_n} (\tau_\sigma H - \tau_{s_n} H)\Big| <2\epsilon,
		\end{equation*}
		when $n$ is large enough. It follows that the limit $v$ of the sequence $v_n$ is a solution of the $s$-dependent Floer equation with respect to $\tau_\sigma H$ and $J$.
		\item Assume $s_n\rightarrow -\infty$.	
		Recalling that the homotopy $H$ is constant for $|s|\geq R$, we have $H(x,t,s) = H_-(x,t)$ whenever $s\leq-R$. Since $s_n\rightarrow-\infty$, for every $r>0$, there exists $N$ large enough, such that for $n\geq N$, $s_n<-R-r$. 
		For such $n$, the restriction of (\ref{eq:app_almost_Floer_sol}) to the compact subset $[-r,r]\times S^1$ is
		\begin{equation*}
		\max_{[-r,r]\times S^1}\Big|\frac{\partial v_n}{\partial s}+ J\frac{\partial v_n}{\partial t} + \grad_{v_n} H_-\Big| < \epsilon,
		\end{equation*}
		since $\tau_{s_n} H(x,t,s) = H(x,t,s+s_n) = H_-$, when $s\in [-r,r]$.
		Taking the limit when $n\rightarrow\infty$ (and $\epsilon\rightarrow0$), we conclude that $v$ is a solution of the Floer equation with respect to $(H_-,J)$. 
		\item When $s_n\rightarrow\infty$, the proof is as in the previous case.
	\end{enumerate} 
\end{proof}

Having Lemma~\ref{lem:app_shifts_have_conv_subseq},  the proof of Proposition~\ref{pro:appendix_broken_traj} (namely, the convergence to a broken trajectory) is similar to the that of \cite[Theorem 11.1.10]{audin-damian}. We follow it and make the necessary adjustments.

\begin{proof}[Proof of Proposition~\ref{pro:appendix_broken_traj}]
	Let us prove the claim for the case where $\sigma_n\rightarrow -\infty$, the other cases are analogous. We start by fixing $\epsilon>0$ small enough, such that the open balls 
	\begin{equation*}
	B(x,\epsilon):=\{\gamma\in \cL M\ |\ d_\infty(x,\gamma)<\epsilon\}
	\end{equation*}
	are disjoint for $x\in \cP(H_-)$. Here $\cL M$ is the space of contractible loops in $M$, endowed with the uniform metric $d_\infty$. By shrinking $\epsilon$ if necessary, we  assume that the balls $\{B(y,\epsilon)\}_{y\in \cP(H_+)}$ are also disjoint.
	Lemma~\ref{lem:app_shifts_have_conv_subseq} guarantees that after passing to a subsequence, the sequence $\tau_{\sigma_n}u_n$ converges in $\cC^\infty_{loc}$ to a finite energy solution $v\in\cM_{(H_-,J)}$. Since $H_-$ is non-degenerate, there exist periodic orbits $x_0, x_1\in\cP(H_-)$ such that $v\in\cM_{(H_-,J)}(x_0,x_1)$. Moreover, applying Lemma~\ref{lem:app_shifts_have_conv_subseq} to the sequence $u_n$ with zero shifts, we conclude that after extracting a subsequence, it converges to a finite energy solution $w\in \cM_{(H,J)}(x_k,y_0)$, for some $x_k\in\cP(H_-)$ and $y_0\in \cP(H_+)$. Let us find the solutions preceding to $v$, connecting $v$ to $w$ and following $w$ in the broken trajectory:
	\begin{itemize}
		\item\underline{Solutions preceding $v$:} There exists  $s_\star\leq0$ such that for any $s\leq s_\star$, $v(s,\cdot)\in B(x_0,\epsilon)$. Since $v=\lim \tau_{\sigma_n}u_n$, when $n$ is large enough, $u_n(s_\star+\sigma_n,\cdot)\in B(x_0,\epsilon)$ as well.
		If $x_0 = x_-$ there are no preceding solutions and we are done. Otherwise, $x_0\neq x_-$, and since $u_n$ converges to $x_-$ when $s\rightarrow-\infty$, it must exit the ball $B(x_0,\epsilon)$ for $s\leq s_\star$. Let us denote by $s_n$ the first exit point:
		\begin{equation*}
		s_n:=\inf\{s\leq s_\star\ |\ u_n(\sigma_n+s',\cdot)\in B(x_0,\epsilon)\ \text{ for } s'\in[s,s_\star]\}.
		\end{equation*}
		Let us now show that $s_n\rightarrow -\infty$. Indeed, if $\{s_n\}$ were bounded, it would have had a subsequence converging to some $s_\circ\in \R$. Since $\tau_{\sigma_n}u_n$ converges to $v$ in $\cC^\infty_{loc}$ and since $s_\circ\leq s_\star$, we would get 
		\begin{equation*}
		\lim_{n\rightarrow \infty} u_n(s_n+\sigma_n,\cdot) = v(s_\circ,\cdot)\in B(x_0,\epsilon),
		\end{equation*}
		in contradiction to our choice of $s_n$, namely, that $u_n(\sigma_n+s_n,\cdot)\in \partial B(x_0,\epsilon)$. Therefore, we conclude that  $s_n\rightarrow -\infty$ and, in particular, $s_n+\sigma_n\rightarrow-\infty$ as well. Using Lemma~\ref{lem:app_shifts_have_conv_subseq} for $\tau_{s_n+\sigma_n}u_n$, we conclude that, after passing to a subsequence, this shifted sequence converges to some $v_{-1}\in \cM_{(H_-,J)}$. We need to prove that $v_{-1}$ converges to $x_0$ when $s\rightarrow\infty$. Fix $s>0$, then for $n$ sufficiently large, $s_n< s+s_n<s_\star$ and  
		$$
		\tau_{s_n+\sigma_n }u_n(s,\cdot) \in B(x_0,\epsilon).
		$$ 
		This implies that $v_{-1}(s,\cdot)\in \overline{B(x_0,\epsilon)}$ for all $s>0$, and hence $v_{-1}\in \cM_{(H_-,J)}(x_{-1},x_0)$, for some $x_{-1}\in \cP(H_-)$. 
		
		Continuing in this way we find $v_{-2}$, $v_{-3}$ and so on, until $x_{-k'} = x_-$. This process is finite, since there are finitely many orbits in $\cP(H_-)$ and the action is strictly decreasing in each step, namely, $\cA_{H_-}(x_{-i})>\cA_{H_-}(x_{-i-1})$, for $0\leq i\leq k'$. 
		
		\item\underline{Solutions connecting $v$ to $w$:} Recall that $\tau_{\sigma_n}u_n$ converges to $v\in\cM_{(H_-,J)}(x_0,x_1)$ and that $u_n$ converges to $w\in \cM_{(H,J)}(x_k, y_0)$. Let us find the solutions that connect $v$ to $w$ (or prove that $x_1= x_k$).
		In analogy with the previous case, pick $s_\star\geq 0$ such that $v(s,\cdot)\in B(x_1,\epsilon)$ for all $s\geq s_\star$. Then, for $n$ large enough, $u_n(s_\star+\sigma_n,\cdot)\in B(x_1,\epsilon)$ as well. Arguing similarly for $w\in\cM_{(H,J)}(x_k,y_0)$, there exists $s_\dagger\leq 0$ such that $w(s,\cdot)\in B(x_k,\epsilon)$ for all $s\leq s_\dagger$ and, since $u_n$ converge to $w$, for $n$ large enough, $u_n(s_\dagger,\cdot )\in B(x_k,\epsilon)$ as well. As $\sigma_n\rightarrow-\infty$, we have $s_\star+\sigma_n<s_\dagger$ for large $n$.  Consider the first exit of $u_n$ from $B(x_1,\epsilon)$,
		\begin{equation*}
		s_n:=\sup\{s\geq s_\star\ |\  u_n(\sigma_n+s',\cdot)\in B(x_1,\epsilon) \ \text{ for } s'\in[s_\star,s] \},
		\end{equation*}
		then, repeating the arguments from the previous step, one sees that $s_n\rightarrow\infty$. 
		Moreover, it follows from the definitions of $s_n$ and $s_\dagger$ that $s_n+\sigma_n<s_\dagger$. Therefore, the sequence $\{\sigma_n+s_n\}$ is either bounded or tends to $-\infty$. In the first case, it converges, after passing to a subsequence, to some number $s_\circ\in\R$. Moreover, since $u_n$ converges to $w$ on compacts, we conclude that $\tau_{\sigma_n+s_n}u_n$ converges to $\tau_{s_\circ}w$. In particular, this implies that $x_1 = x_k$. Indeed, for every $s<s_\circ$ and $n$ sufficiently large, $s\in [\sigma_n+s_\star,\sigma_n+s_n]$, and thus $u_n(s,\cdot)\in B(x_1,\epsilon)$. As a consequence, $w(s,\cdot)\in B(x_1,\epsilon)$ for all $s<s_\circ$, which means that $x_k=x_1$ and we are done.
		Let us now deal with the case where $s_n+\sigma_n\rightarrow-\infty$. By Lemma~\ref{lem:app_shifts_have_conv_subseq}, there exists a subsequence of $\tau_{s_n+\sigma_n}u_n$ that converges to a finite energy solution $v_1\in\cM_{(H_-,J)}$. We need to show that the left end of $v_1$ converges to $x_1$, namely, that $v_1\in\cM_{(H_-,J)}(x_1, x_2)$ for some $x_2\in \cP(H_-)$. Fix $s<0$ and let us show that $v_1(s,\cdot)\in B(x_1,\epsilon)$. Since $s_n\rightarrow\infty$, when $n$ is large enough, $s+s_n\in [s_\star, s_n]$. As we saw above, this implies that $\tau_{\sigma_n+s_n}u_n(s,\cdot) = u_n(s+s_n+\sigma_n,\cdot)\in B(x_1,\epsilon)$, and thus $v_1(s,\cdot)\in \overline{B(x_1,\epsilon)}$ as required. Repeating this process, we find solutions $v_2,\cdots v_{k-1}$ such that $v_i\in \cM_{(H_-,J)}(x_i,x_{i+1})$, and therefore these connect $v$ to $w$. As in the previous case, this process is finite since every solution $v_i$ is action decreasing and $H_-$ has finitely many 1-periodic orbits. 
		
		\item\underline{Solutions following $w$:} The right end of $w$ converges to $y_0\in \cP(H_+)$, and hence there exists $s_\star\geq 0$ such that for every $s\geq s_\star$, $w(s,\cdot)\in B(y_0,\epsilon)$. As $u_n$ converge to $w$ in $\cC^\infty_{loc}$, for $n$ large enough, $u_n(s_\star,\cdot)\in B(y_0,\epsilon)$ as well. Assume that $y_0\neq x_+$, otherwise there is nothing to prove. Then, since $u_n$ converge to $x_+$ for each $n$, is must leave the ball $B(y_0,\epsilon)$ at some point. Consider the first exit,
		\begin{equation*}
		s_n:= \sup\{s\geq s_\star\ |\ u_n(s',\cdot)\in B(y_0,\epsilon)\ \text{ for } s'\in[s_\star,s] \},
		\end{equation*}
		then, arguing as above, $s_n\rightarrow\infty$. Applying Lemma~\ref{lem:app_shifts_have_conv_subseq} to the sequence $u_n$ shifted by $s_n$, it converges (up to a subsequence) to a finite energy solution $v_1'\in \cM_{(H_+,J)}$. We need to show that the left end of $v_1'$  converges to $y_0$, namely, that $v_1'\in \cM_{(H_+,J)}(y_0,y_1)$ for some $y_1\in\cP(H_+)$. As before, fix any $s<0$, then when $n$ is large enough, $s+s_n\in [s_\star, s_n]$ and therefore, $\tau_{s_n}u_n(s,\cdot) = u_n(s+s_n,\cdot)\in B(y_0,\epsilon)$. Again, we conclude that $v_1'(s,\cdot)\in \overline{B(y_0,\epsilon)}$, which guarantees that $v_1'$ converges to $y_0$. Continuing by induction and using the fact that each $v_j'$ reduces the action concludes the proof. 
	\end{itemize}
\end{proof}
\begin{figure}
	\centering
	\includegraphics[scale=0.75]{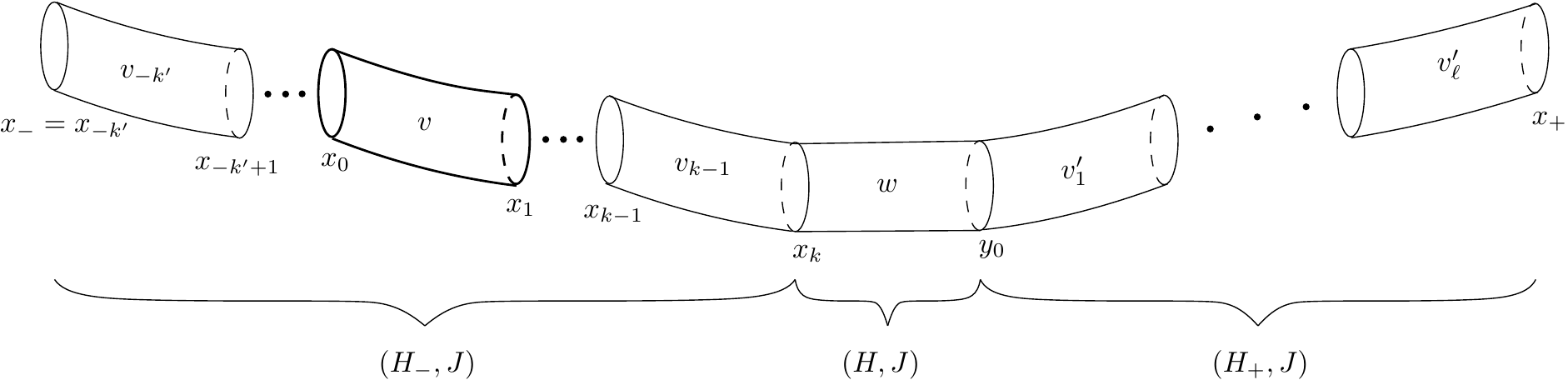}
	\caption{\small{An illustration of the broken trajectory as constructed in the Proof of Proposition~\ref{pro:appendix_broken_traj}.
	}}
	\label{fig:broken_traj1}
\end{figure}

\subsubsection{Concatenation of homotopies with possibly degenerate ends.} \label{app:partial_broken_traj}
In what follows, we study the breaking mechanism for solutions of (\ref{eq:FE}) with respect to homotopies of Hamiltonians, that are obtained as concatenations of finitely many given homotopies, with possibly degenerate ends.  In addition, we consider homotopies of almost complex structures, as opposed to the constant structures considered previously. When the ends of the first few concatenated homotopies are non-degenerate, we prove what we call a {\it partial convergence to a broken trajectory}.

Let $(H^1,J^1),\dots, (H^K, J^K)$  be pairs of homotopies of Hamiltonians and homotopies of almost complex structures, respectively, that are constant outside of $[0,1]$, namely 
\begin{equation*}
\partial_s H^k=0 \text{ and }\partial_s J^k=0,\ \ \text{ for }s\notin [0,1],\ k=1,\dots,K.
\end{equation*} 
Assume in addition that $H_+^k = H_-^{k+1}$ and $J_+^k = J_-^{k+1}$ for $k=1,\dots, K-1$.
Let  $\{\sigma_n^1\}_n,\dots, \{\sigma_n^K\}_n$ be monotone sequences of real numbers, such that for each $n$, $\sigma_n^1<\cdots<\sigma_n^K$ and for each $k\neq j$, the sequence of differences $\{\sigma_n^k-\sigma_n^j\}_n$ is unbounded. 
For the rest of this section, we consider the sequences $\{H_n\}$ and $\{J_n\}$ of homotopies of Hamiltonians and almost complex structures obtained by concatenating the shifts of $\{H^k\}$ and $\{J^k\}$ by the sequences $\{-\sigma_n^k\}$. More formally, $H_n$ and $J_n$ are the sequences satisfying 
\begin{equation*}	
H_n = \tau_{-\sigma_n^k} H^k \text{ on } M\times S^1\times [\sigma_n^k,\sigma_n^k+1],\quad \text{ and } \quad
J_n = \tau_{-\sigma_n^k} J^k \text{ on } S^1\times [\sigma_n^k,\sigma_n^k+1],
\end{equation*}
 for each $k=1, \dots, K$, and are locally constant elsewhere, see Figure~\ref{fig:App_concatenating_shifted_homotopies}.
\begin{figure}[h]
	\centering
	\includegraphics[scale=0.9]{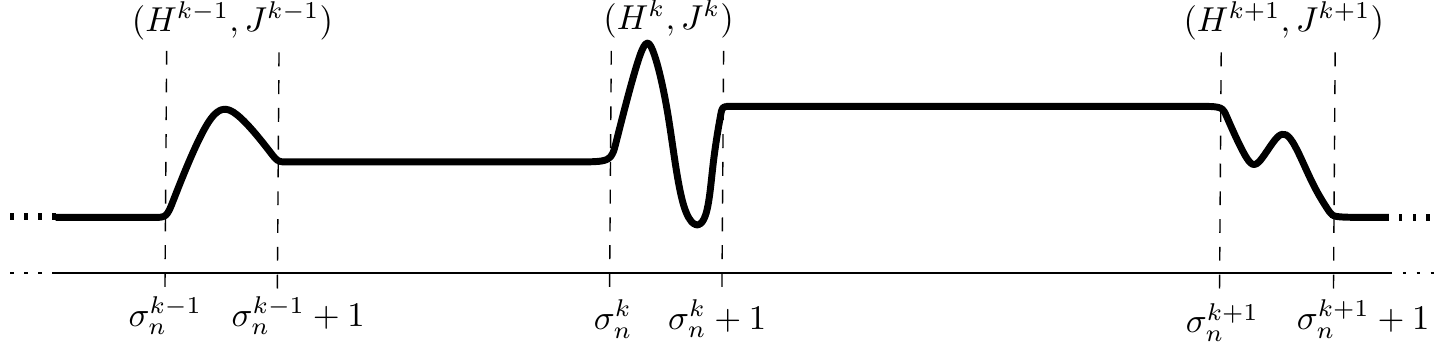}
	\caption{\small{An illustration of the pair $(H_n,J_n)$, which is a concatenation of the homotopies $(H^k, J^k)$ shifted by $\{\sigma_n^k\}$.
		}}
	\label{fig:App_concatenating_shifted_homotopies}
\end{figure}
Since the ends of the homotopies $H^k$ might be degenerate, a sequence of solutions $u_n\in \cM_{(H_n, J_n)}$ does not necessarily admit a subsequence converging to a broken trajectory. However, when some of the homotopies have non-degenerate ends, a slightly weaker statement holds:

\begin{prop}\label{pro:partial_broken_trajectory}
	Assume that there exists $1< K'\leq K$, such  that for every $k< K'$, the ends of the homotopy $H^k$ are non-degenerate. Then, for every sequence $u_n\in \cM_{(H_n, J_n)}$, there exist:
	\begin{itemize}
		\item A subsequence of $\{u_n\}$, which we still denote by $\{u_n\}$,
		\item periodic orbits $x^{k,\ell}\in \cP(H_-^{k+1})$  for  $\ell =0,\dots, L_k$ and $k=1,\dots, K'$, where $x^{1,0} = \lim_{s\rightarrow -\infty}u_n(s,\cdot)$ for all $n$,
		\item real numbers $s_n^{k,\ell}\in \R$  for $\ell =1,\dots, L_k$ and $k=1,\dots,K'$, such that $s_n^{k,\ell}<\sigma_n^{k+1}<s_n^{k+1,\ell'}$ for all $\ell =1,\dots, L_k$ and $\ell' =1,\dots, L_{k+1}$,
		\item solutions of $s$-independent Floer equations $v^{k,\ell}\in \cM_{(H_-^{k},J_-^{k})}(x^{k,\ell-1},x^{k,\ell})$ for $\ell=1,\dots, L_k$ and $k=1,\dots, K'$,
		\item solutions of $s$-dependent Floer equations $w^k\in\cM_{(H^k,J^k)}(x^{k,L_k}, x^{k+1,0})$, for $k=1,\dots, K'-1$, and $w^{K'}\in\cM_{(H^{K'}, J^{K'})}$ such that $\lim_{s\rightarrow-\infty}w^{K'}(s,\cdot)=x^{K', L_{K'}}$,
	\end{itemize}
	such that, in $\cC^\infty_{\text{loc}}(\R\times S^1;M)$, 
	\begin{equation*}
	\lim_{n\rightarrow \infty} u_n(\cdot +\sigma_n^k,\cdot)= w^k \quad \text{and} \quad 
	\lim_{n\rightarrow\infty} u_n(\cdot +s_n^{k,\ell},\cdot) =v^{k,\ell}.
	\end{equation*}
	for $1\leq \ell \leq L_k$ and $1\leq k\leq K'$.
\end{prop}
\noindent In this case, we say that $\{u_n\}$ partially converges to the {broken trajectory}
\begin{equation*}
\overline{v} = \left(\{v^{1,\ell}\}_{\ell=1}^{L_1},w^1,\{v^{2,\ell}\}_{\ell=1}^{L_2},w^2,\dots,w^{K'-1}, \{v^{K',\ell}\}_{\ell=1}^{L_{K'}},w^{K'}\right).
\end{equation*}

In order to prove Proposition~\ref{pro:partial_broken_trajectory}, we need  statements analogous to Lemma~\ref{lem:app_bdd_gradient} and Lemma~\ref{lem:app_shifts_have_conv_subseq} adapted for the current setting. Notice that due to our assumption, that $H^k$ has non-degenerate ends for $1\leq k<K'$, the left end of the homotopies $H_n$, which is equal to $H^1_-$, is non-degenerate. On the other hand, the right end, $H_{n+}=H^K_+$, might be degenerate. A solution $u$ of the Floer equation with respect to a homotopy with degenerate ends does not necessarily converge to periodic orbits at the ends. However, the following lemma asserts that the action of $u(s,\cdot)$ converges as $s\rightarrow\pm\infty$, to a limit that belongs to the action spectrum of the corresponding Hamiltonian. The following statement is proved in the proof of Proposition~2.1, (ii) from \cite{abbondandolo2019simple} for the left end of $u$, namely, $\lim_{s\rightarrow -\infty} \cA_{H_-}(u(s,\cdot)) \in \spec(H_-)$. The proof for the right end is completely analogous and we therefore omit it. 
\begin{lemma}[\cite{abbondandolo2019simple}]\label{lem:app_ends_action_in_spctrum}
	Let $(H, J)$ be a pair of homotopies of Hamiltonians and almost complex structures. Then, for every finite energy solution $u\in \cM_{(H,J)}$, 
	\begin{equation*}
	\lim_{s\rightarrow \pm\infty} \cA_{H_\pm}(u(s,\cdot)) \in \spec(H_\pm).
	\end{equation*}
\end{lemma}

Denoting by $\cM:=\cup_n \cM_{(H_n,J_n)}$ the set of finite energy solutions, the next lemma  provides a uniform bound for the energy of $u\in \cM$ and is an adjustment of Lemma~\ref{lem:app_bdd_gradient} to the current setting.
\begin{lemma}\label{lem:app_bdd_gradient_partial_bt}
	There exists a constant $A>0$ such that for every $u\in \cM$ and $(s,t)\in \R\times S^1$, one has $\|\grad_{(s,t)} u\|\leq A$.
\end{lemma}
\begin{proof}
	For a finite energy solution $u$ of a homotopy with possibly degenerate ends, the limits $\lim_{s\rightarrow \pm \infty} \cA_{H_\pm}(u(s,\cdot))$ exist and $u$ satisfies the following energy-identity: 
	\begin{equation*}
	E(u) = \lim_{s\rightarrow-\infty}\cA_{H_{n-}}(u(s,\cdot))- \lim_{s\rightarrow\infty}\cA_{H_{n+}}(u(s,\cdot)) +\int_{\R\times S^1} \partial_s H_n(u(s,t),t,s)\ ds\ dt,
	\end{equation*}
	see, for example, \cite[p.8]{abbondandolo2019simple}. When $u\in \cM_{(H_n, J_n)}$, it follows from Lemma~\ref{lem:app_ends_action_in_spctrum}, together with the fact that the action spectrum is a compact subset of $\R$, that 
	\begin{eqnarray}
	\nonumber
	E(u) &\leq& \max{\spec(H_{n-})}-\min{\spec(H_{n+})} +\int_{\R\times S^1} \partial_s H_n(u(s,t),t,s)\ ds\ dt \\
	\nonumber
	&=& \max{\spec(H_{-}^1)}-\min{\spec(H_{+}^K)} +\int_{\R\times S^1} \partial_s H_n(u(s,t),t,s)\ ds\ dt \\
	\nonumber
	&\leq& \max{\spec(H_{-}^1)}-\min{\spec(H_{+}^K)} +K\cdot C,
	\end{eqnarray}
	where, by our construction, $K$ bounds the area of the support of $\max_{x\in M}\partial_s H_n(x,t,s)$ in $S^1\times \R$, and $C$ is defined by 
	\begin{eqnarray}
	\nonumber
	C&:=&\sup\left\{\frac{\partial H_n}{\partial s}(x,t,s)\ |\ (x,t,s)\in M\times S^1\times \R,\  n\geq 0\right\}\\
	&=& \max_{k\leq K}\  \sup\left\{\frac{\partial H^k}{\partial s}(x,t,s)\ |\ (x,t,s)\in M\times S^1\times [0,1]\right\}.
	\end{eqnarray}
	We therefore have obtained a uniform bound on the energies of solutions in $\cM$. Arguing as in \cite[Propositions 6.6.2, 11.1.5]{audin-damian}, we conclude that there exists $A>0$ such that $\|\grad_{(s,t)} u\|\leq A$. 
\end{proof}
The last lemma for this section is analogous to Lemma~\ref{lem:app_shifts_have_conv_subseq}. It can be viewed as a special case of Proposition~2.1 from \cite{abbondandolo2019simple}, but we include the proof for the sake of completeness.
\begin{lemma}\label{lem:app_partial_shifts_have_conv_subseq}
	Let $u_n\in \cM_{(H_n,J_n)}$ be a sequence of solutions and let $s_n\in\R$ be a sequence of numbers such that, for some $0\leq k \leq K$ and for every $n$, $\sigma_n^k\leq s_n\leq\sigma_n^{k+1}$, where we set $\sigma_n^0=-\infty$ and $\sigma_n^{K+1}=+\infty$ to simplify the notations. Then, the sequence of shifted solutions $\tau_{s_n}u_n(\cdot, \cdot) = u_n(\cdot+s_n, \cdot)$ has a subsequence that converges in the $\cC^{\infty}_{{loc}}$ topology to a limit $v$. Moreover:
	\begin{enumerate}
		\item If $s_n-\sigma_n^k\rightarrow\sigma\in \R$, then $v\in \cM_{(\tau_\sigma H^k,\tau_\sigma J^k)}$.
		\item If $s_n-\sigma_n^{k+1}\rightarrow\sigma\in \R$, then $v\in \cM_{(\tau_\sigma H^{k+1},\tau_\sigma J^{k+1})}$.
		\item If $s_n-\sigma_n^{k+1}\rightarrow -\infty$ and $s_n-\sigma_n^{k}\rightarrow \infty$, then  $v\in \cM_{(H_+^k,J_+^k)}=\cM_{(H_-^{k+1},J_-^{k+1})}$.
	\end{enumerate}
\end{lemma}
\begin{proof}
	The proof is very similar to that of Lemma~\ref{lem:app_shifts_have_conv_subseq} and therefore we only sketch the changes. As before, Lemma~\ref{lem:app_bdd_gradient_partial_bt} implies that the sequence $v_n:=\tau_{s_n}u_n$ is equicontinuous, and by the Arzel\`a-Ascoli theorem and elliptic regularity there exists a subsequence converging to $v$. The maps $v_n$ solve the Floer equation with respect to the translated pair $(\tau_{s_n} H_n,\tau_{s_n} J_n)$:
	\begin{equation*}
	0=\frac{\partial v_n}{\partial s}+ (\tau_{s_n}J_n)\frac{\partial v_n}{\partial t} + \grad (\tau_{s_n} H_n)\circ v_n.
	\end{equation*}
	In each case, in order to prove that $v$ is a solution of the corresponding equation, one shows that the translated homotopies $(\tau_{s_n} H_n, \tau_{s_n} J_n)$ converge uniformly on compacts to the required pair. For example, in the first case, where $s_n-\sigma_n^k\rightarrow\sigma\in \R$, it follows from the definition of $(H_n, J_n)$ that, given $r>0$, the sequence $(\tau_{\sigma_n^k} H_n, \tau_{\sigma_n^k} J_n)$ eventually stabilizes to $(H^k,J^k)$ on $\{|s|\leq r\}$. As a consequence, $(\tau_{s_n} H_n, \tau_{s_n} J_n) \xrightarrow{\cC^\infty_{loc}} (\tau_\sigma H^k,\tau_\sigma J^k)$.
\end{proof} 

We are now ready to sketch the proof of Proposition~\ref{pro:partial_broken_trajectory}. Note that we will skip some of the details appearing in the proof of Proposition~\ref{pro:appendix_broken_traj}.
\begin{proof}[Proof of Proposition~\ref{pro:partial_broken_trajectory}]
	As mentioned above, for each $n$, the left end of $H_{n}$ is a non-degenerate Hamiltonian. As a consequence, the left end of $u_n$ converges to a periodic orbit, namely, there exist $x_{n-}\in \cP(H_{n-})$, such that $\lim_{s\rightarrow-\infty}u_n(s,\cdot) = x_{n-}(\cdot)$ (see, for example, the proof of Theorem 6.5.6 from \cite{audin-damian}). Since $ \cP(H_{n-})=\cP(H_-^1)$ is a finite set, we may assume, by passing to a subsequence, that $x^{1,0}:=x_{n-}$ is independent of $n$. 
	
	Next, let us apply Lemma~\ref{lem:app_partial_shifts_have_conv_subseq} to the sequence $\{u_n\}$ with the shifts $\sigma_n^k$. Then, after passing to a subsequence, for each $k=1,\dots, K'$ we obtain $w^k\in\cM_{(H^k, J^k)}$ such that $\tau_{\sigma_n^k}u_n\xrightarrow{\cC^\infty_{loc}}w^k$. Fixing $1\leq k\leq K'$, we need to find solutions $\{v^{k,\ell}\}_{\ell=1}^{L_k}$ connecting $w^{k-1}$ to $w^{k}$ (and $x^{1,0}$ to $w^1$).  It follows from the non-degeneracy of $H^k_-$ that $\cP(H_-^k)= \cP(H_+^{k-1})$ is a  finite set (notice that the left end of $H^{K'}$ is non-degenerate, as it coincides with the right end of $H^{K'-1}$). Therefore, we can repeat the arguments from the proof of Proposition~\ref{pro:appendix_broken_traj}. For $\epsilon>0$ small enough, the balls $\{B(x_-,\epsilon)\}_{x_-\in \cP(H^k_-)}$ are disjoint, and denoting $y^{k}:=\lim_{s\rightarrow-\infty} w^k(s,\cdot)$ and  $x^{k,0}=\lim_{s\rightarrow\infty} w^{k-1}(s,\cdot)$, there exists $s_\star \in \R$ such that $w^{k-1}(s,\cdot)\in B(x^{k,0},\epsilon)$ for $s\geq s_\star$. It follows from the convergence of  $\tau_{\sigma_n^{k-1}}u_n$ to $w^{k-1}$  that $u_n(s_\star+\sigma_n^{k-1},\cdot)\in  B(x^{k,0},\epsilon)$ when $n$ is large. Denoting by 
	\begin{equation*}
	s_n^{k,1}:=\sup\{s\geq s_\star+\sigma_n^{k-1}\ |\  u_n(s',\cdot)\in B(y^k,\epsilon) \ \text{ for } s'\in[s_\star+\sigma_n^{k-1},s] \}
	\end{equation*}
	the first exit, one can argue as in the proof of Proposition~\ref{pro:appendix_broken_traj} to show that $s_n^{k,1}-\sigma_n^{k-1}\xrightarrow[{n\rightarrow\infty}]{}\infty$. Applying Lemma~\ref{lem:app_partial_shifts_have_conv_subseq} to $\{u_n\}$ shifted by $s_n^{k,1}$, we conclude that $\tau_{s_n^{k,1}}u_n$ either converges to $\tau_\sigma w^{k}$, for some $\sigma\in \R$, or to $v^{k,1}\in\cM_{(H^k_-, J^k_-)}=\cM_{(H^{k-1}_+, J^{k-1}_+)}$. In the first case, the right end of $w^{k-1}$ and the left end of $w^{k}$ coincide, namely $x^{k,0}=y^k$, and we are done. Otherwise, we continue by induction and find $v^{k,\ell}\in\cM_{(H^k_-, J^k_-)}$ connecting $w^{k-1}$ to $w^{k}$. As argued previously, this process is finite since each $v^{k,\ell}$ decreases the action, and $\spec(H_-^k)$ is a finite set. 
\end{proof}

\subsection{Barricades and perturbations.}\label{app:barricades_and_perturbations}
Throughout this section, we fix an almost complex structure $J$ on $M$, a CIB domain $U$, and $U_\circ\Subset U$. We will consider non-degenerate Hamiltonians, or homotopies with non-degenerate ends, that have a barricade in $U$ around $U_\circ$, when paired with $J$.

\subsubsection{Barricades survive under small enough perturbations.}\label{app:barricade_survives_perturbation}
In this section we  show that barricades survive under perturbations of $H$. Here $H$ denotes a homotopy with non-degenerate ends and we consider Hamiltonians as a special case, by identifying them with constant homotopies.

\begin{prop}\label{pro:appendix}
	Let $H$ be a homotopy with non-degenerate ends, such that $\partial_s H|_{|s|>R}=0$ for some $R>0$ (in particular, $H$ can be a non-degenerate Hamiltonian), and such that the pairs $(H,J)$, $(H_\pm,J)$ have a barricade in $U$ around $U_\circ$. Then, for every $\cC^\infty$-small enough perturbation $H'$ of $H$ that satisfies $\cP(H_\pm)=\cP(H_\pm')$ and $\partial_s H'|_{|s|>R} =0$, the pair $(H',J)$ has a barricade in $U$ around $U_\circ$. 
\end{prop}

In order to prove this proposition, we will use the convergence to broken trajectories, which was established in Section~\ref{app:broken_traj_non_deg}. Therefore, we start by showing that barricades also restrict broken trajectories.
\begin{lemma}\label{lem:barricade_ban_broken_traj}
	Let $H$ be a homotopy with non-degenerate ends (or, in particular, a non-degenerate Hamiltonian) such that the pairs $(H,J)$ and $(H_\pm, J)$  have a barricade in $U$ around $U_\circ$. Then, for a broken trajectory $\overline{v} = (v_1,\dots,v_k,w,v_1',\dots,v_\ell')$ connecting $x_\pm\in\cP(H_\pm)$, we have:
	\begin{itemize}
		\item If $x_-\subset U_\circ$, then $\overline{v}\subset U_\circ$.
		
		\item If $x_+\subset U$, then $\overline{v}\subset U$.
	\end{itemize}
\end{lemma}
\begin{proof}
	We prove the first statement, the second statement is completely analogous.
	Let 
	$$
	\overline{v}:=(v_1,\dots,v_k,w,v'_1,\dots,v'_\ell)
	$$ 
	be a broken trajectory of $(H,J)$ such that the periodic orbit $x_0:= \lim_{s\rightarrow-\infty} v_1(s,\cdot)$ is contained in $U_\circ$. Then, since $(H_-, J)$ has a barricade in $U$ around $U_\circ$, $v_1\subset U_\circ$ and in particular, the periodic orbit $x_1:=\lim_{s\rightarrow+\infty}v_1(s,\cdot)$ is contained in $U_\circ$. By definition (see Proposition~\ref{pro:appendix_broken_traj}), $x_1$ is the negative end of $v_2$, namely $x_1=\lim_{s\rightarrow-\infty}v_2(s,\cdot)$. Applying the same argument again and again, we conclude that $v_2,\dots,v_k \subset U_\circ$.  Now, $x_k:=\lim_{s\rightarrow+\infty}v_k(s,\cdot) =\lim_{s\rightarrow-\infty}w(s,\cdot)$ is also contained in $U_\circ$ and since $(H,J)$ has a barricade in $U$ around $U_\circ$, this means that $w\subset U_\circ$. Arguing the same way and using the fact that $(H_+,J)$ has a barricade in $U$ around $U_\circ$ we conclude that $v'_j\subset U_\circ$ for all $1\leq j\leq \ell$, and so the broken trajectory is completely contained in $U_\circ$. 
\end{proof}

Given the above lemma, the proof of Proposition~\ref{pro:appendix} is a simple application of Proposition~\ref{pro:appendix_broken_traj}.
\begin{proof}[Proof of Proposition~\ref{pro:appendix}]	
	Let $\{H_n\}$ be a sequence of regular homotopies converging to $H$, such that for each $n$, $\cP(H_{n\pm}) = \cP(H_\pm)$ and $\partial_s H_n|_{|s|>R} =0$. Assume for the sake of contradiction that, for each $n$, there exists a solution $u_n\in \cM_{(H_n,J)}$ such that $x_-^n:=\lim_{s\rightarrow-\infty}u_n(s,\cdot)$ is contained in $U_\circ$ but $u_n$ is not. For each $n$, let $\sigma_n\in \R$ be such that $u_n(\sigma_n,\cdot)$ is not contained in $U_\circ$. Since $x_\pm^n\in\cP(H_{n\pm})=\cP(H_\pm)$  are elements of a finite sets, by passing to a subsequence, we may assume that $x_\pm^n=x_\pm$ are independent of $n$, which means that $u_n\in \cM(x_-,x_+)$ for all $n$. Applying Proposition~\ref{pro:appendix_broken_traj} to the sequence of solutions $\{u_n\}$ and the sequence of shifts $\{\sigma_n\}$, after passing again to a subsequence, $\{u_n\}$ converges to a broken trajectory $\overline{v}$ of $(H,J)$, and the sequence $u_n(\cdot+\sigma_n,\cdot)$ converges to one of the solutions in $\overline{v}$ (perhaps up to a shift). Lemma~\ref{lem:barricade_ban_broken_traj}, together with our assumption that $x_-=x_0\subset U_\circ$, guarantee that the entire broken trajectory $\overline{v}$ is contained in $U_\circ$, and  in particular $\lim_{n\rightarrow \infty}u_n(\cdot+\sigma_n,\cdot)\subset U_\circ$.
	Since the latter limit is uniform on compacts, it follows that 
	$$
	\lim_{n\rightarrow\infty} u_n(\sigma_n,\cdot) = \lim_{n\rightarrow\infty}u_n(0+\sigma_n,\cdot)
	$$
	is also contained in $U_\circ$. Recalling that we chose $\sigma_n$ such that, for each $n$, the loop $u_n(\sigma_n,\cdot)$ is not contained in the open set $U_\circ$, we arrive at a contradiction. 
	
	Similarly, one can prove that when $n$ is large enough, every solution $u_n$ of the Floer equation with respect to $(H_n,J)$ ending in $U$ is contained in $U$.
\end{proof}

\subsubsection{Perturbing Hamiltonians that are regular on a subset.} \label{app:Hamiltonian_reg_in_U}
In this section, we define the notion of regularity on a subset, $U\subset M$,  for a pair $(H,J)$ of a Hamiltonian and an almost complex structure that has a barricade in $U$ around some $U_\circ\Subset U$. We prove that for such a pair, the restriction of the Floer differential to the set is well defined, and is stable under (regular) perturbations. 
Since Floer-regularity concerns the differential of the Floer map we start with a reminder. Given a Hamiltonian $H$ and an almost complex structure $J$, the Floer map associated to the pair $(H,J)$ is 
\begin{eqnarray}
\nonumber
\cF=\cF^H:\cC^\infty(\R\times S^1;M) &\longrightarrow& \cC^\infty(\R\times S^1;TM),\\
\nonumber
u &\mapsto& \frac{\partial u}{\partial s}+J\frac{\partial u}{\partial t} +\grad_u(H_t),
\end{eqnarray} 
where $\grad_u H:=\nabla_J H\circ u$ is the gradient of $H$ with respect to $J$, composed on $u$.
\begin{defin}\label{def:Ham_reg_on_U}
	Let $H$ be a non-degenerate Hamiltonian such that the pair $(H,J)$ has a barricade in $U$ around $U_\circ$.
	\begin{enumerate}
		\item We say that the pair $(H,J)$ is {\it regular on $U$} if for 
		every solution $u$ of the Floer equation that is contained in $U$, the linearization $(d\cF)_u$ of the Floer map $\cF$ at $u$ is surjective. 
		
		In particular, by \cite[Theorem 8.1.2]{audin-damian}, for every $x_\pm \in \cP(H)$ such that $x_+\subset U$, the space of solutions $\cM_{(H,J)}(x_-, x_+)$ is a smooth manifold of dimension $\mu(x_-)-\mu(x_+)$. 
		\item We say that the pair $(H,J)$ is {\it semi-regular on $U$} if for every $x_\pm \in \cP(H)$, with $\mu(x_-)\leq\mu(x_+)$ and such that $x_+\subset U$, we have:
		\begin{enumerate}
			\item If $x_-\neq x_+$, then $\cM_{(H,J)}(x_-, x_+)=\emptyset$.
			\item If $x_- = x_+$, then $\cM_{(H,J)}(x_-, x_+)$ contains only the constant solution $u(s,t)=x_-(t)$.
		\end{enumerate}
	\end{enumerate}		
\end{defin}

\begin{rem}\label{rem:reg_on_U}
	\begin{itemize}
		\item If $(H,J)$ is regular on $U$, then it is also semi-regular on $U$.
		
		\item If $(H,J)$  has a barricade in $U$ around $U_\circ$ and agrees, on $U$, with a Floer-regular pair, then it is regular on $U$.
		
		\item For a pair $(H,J)$ that is regular on $U$, the differential of the Floer complex might not be defined everywhere. However, using Proposition~\ref{pro:appendix_broken_traj} (see also the proof of Lemma~\ref{lem:semi_regular_is_open} below), one can show that when $\mu(x_-)-\mu(x_+)=1$ and $x_+\subset U$, the quotient manifold $\cM_{(H,J)}(x_-, x_+)/\R$ is compact and of dimension 0, and hence finite.
		Therefore, the composition $\pi_U\circ \partial_{(H,J)}$ can be defined by counting the elements of the latter quotients. This is a slight abuse of notations, as the map $\partial_{(H,J)}$ is not defined on its own. Similarly, one can define the composition $\partial_{(H,J)}\circ\pi_{U_\circ}$ using the fact that $x_-\subset U_\circ$ implies that $x_+\subset U_\circ \subset U$, due to the barricade.	
	\end{itemize}
\end{rem}

Our main goal for this section is to prove the following statement.
\begin{prop}\label{pro:app_pert_Ham_regular_on_U}
	Suppose that $H$ is a non-degenerate Hamiltonian such that $(H,J)$ is regular on $U$. Let $H'$ be a small perturbation of $H$ such that the pair $(H',J)$ is Floer-regular and $H'$ agrees with $H$ on $\cP(H)$ up to second order. Then, the compositions of the differentials and projections agree:
	\begin{equation}\label{eq:comp_of_diff_agree}
	\pi_U\circ \partial_{(H,J)}=\pi_U\circ \partial_{(H',J)},
	\qquad\partial_{(H,J)}\circ\pi_{U_\circ} = \partial_{(H',J)}\circ\pi_{U_\circ}. 
	\end{equation}  
\end{prop}
We remark that the second equation in (\ref{eq:comp_of_diff_agree}) follows immediately from the first. Indeed, due to Proposition~\ref{pro:appendix}, $(H',J)$ also has a barricade, and $\partial\circ\pi_{U_\circ}=\pi_U\circ\partial\circ\pi_{U_\circ}$ for both $(H,J)$ and $(H',J)$.  
In order to prove Proposition~\ref{pro:app_pert_Ham_regular_on_U}, we connect $H$ and $H'$ by a path of Hamiltonians $\{H_\lambda\}_{\lambda\in[0,1]}$, such that, for each $\lambda\in [0,1]$, $H_\lambda$ agrees with $H$ on the 1-periodic orbits up to second order, and the pair $(H_\lambda, J)$ is semi-regular on $U$. Note that the first condition implies that, for each $\lambda$, $\cP(H_\lambda)=\cP(H)$. Given $x_\pm\in \cP(H)$, such that $x_+\subset U$, the space 
\begin{equation}\label{eq:cobordism_space}
\cM_\Lambda(x_-, x_+):=\left\{(\lambda, u): u \in \cM_{(H_\lambda,J)}(x_-, x_+)\right\}
\end{equation}
is invariant under the $\R$ action $u(\cdot,\cdot)\mapsto u(\sigma+\cdot,\cdot)$. We show that when $\mu(x_-)-\mu(x_+)=1$,  the quotient $\overline{\cM}_\Lambda(x_-, x_+)=\cM_\Lambda(x_-, x_+)/\R$ is a smooth, compact 1-dimensional manifold with boundary, that realizes  a cobordism between $\cM_{(H,J)}(x_-,x_+)/\R$ and $\cM_{(H',J)}(x_-,x_+)/\R$. We will then conclude that the number of elements in the quotients  $\cM_{(H,J)}(x_-,x_+)/\R$ and  $\cM_{(H',J)}(x_-,x_+)/\R$ coincides modulo 2. 

The existence of a semi-regular path between $H$ and $H'$ follows from the fact that semi-regularity is an open condition.
\begin{lemma}\label{lem:semi_regular_is_open}
	Suppose that $(H,J)$ is semi-regular on $U$, then, for every Hamiltonian $H'$ that is close enough to $H$ and agrees with $H$ on $\cP(H)$ up to second order, the pair $(H',J)$ is also semi-regular on $U$.
\end{lemma} 
\begin{proof}
	Consider a sequence $H_n$, converging to $H$, such that, for each $n$, $H_n$ agrees with $H$ on $\cP(H)$. Then, in particular, $\cP(H_n)=\cP(H)$. Suppose that for each $n$, there exist a solution $u_n \in\cM_{(H_n,J)}(x_-^n, x_+^n)$, for some $x_\pm^n\in\cP(H_n)$, such that $\mu(x_-^n)\leq\mu(x_+^n)$ and $x_+^n\subset U$. Moreover, we assume that if $x_-^n=x_+^n$, then $u_n$ is non-constant. Since $x_\pm^n\in  \cP(H)$ are elements of a finite set, we may assume, by passing to a subsequence, that $x_\pm^n = x_\pm$ are independent of $n$. By Proposition~\ref{pro:appendix_broken_traj}, there exists a subsequence of the solutions $u_n$ that converges to a broken trajectory $\overline{v}$ of $(H,J)$. Moreover, the ends of the broken trajectory are $x_\pm$. Since $x_+$ is contained in $U$ and $(H,J)$ has a barricade in $U$ around $U_\circ$, it follows from Lemma~\ref{lem:barricade_ban_broken_traj} that the broken trajectory $\overline{v}$ is contained in $U$. As the pair $(H,J)$ is semi-regular on $U$, for every non-constant solution in the broken trajectory, the index difference between the left end and the right end is positive. Therefore, in the notations of Proposition~\ref{pro:appendix_broken_traj}, we have
	$$
	\mu(x_-) = \mu(x_0)>\mu(x_1)> \cdots > \mu(x_+).
	$$
	Together with our assumption that  $\mu(x_-)=\mu(x_-^n)\leq\mu(x_+^n)=\mu(x_+)$, this implies that the broken trajectory $\overline{v}$ contains only one solution: $v_1(s,t) =x_-(t)=x_+(t)$. In particular, we conclude that $u_n \in\cM_{(H_n,J)}(x_-, x_+)$ are Floer-solutions with equal ends. By the energy identity (\ref{eq:energy_id_Hamiltonian}), the energy of $u_n$ vanishes, 
	$$
	E(u_n) = \cA_{H_n}(x_-)-\cA_{H_n}(x_-)=\cA_{H}(x_-)-\cA_{H}(x_-) =0,
	$$ 
	which guarantees that $u_n$ is a constant solution, $u_n(s,t)=x_-(t)$ for all $n$, in contradiction. 
\end{proof}

Our next aim is to show that for a suitable choice of a path of Hamiltonians $\{H_\lambda\}$, the set (\ref{eq:cobordism_space}) is a smooth manifold. Let us start with preliminary definitions. Let $\{H_\lambda\}_{\lambda\in [0,1]}$ be a path of Hamiltonians that is stationary for $\lambda \notin[\delta,1-\delta]$ for some fixed $\delta>0$, and such that $H_\lambda$ agrees with $H_0$ on $\cP(H_0)$ up to second order, for all $\lambda\in [0,1]$.
We will consider the space $\cC_\varepsilon^\infty(\{H_\lambda\}_\lambda)$ (of perturbations) consisting of maps:
\begin{equation*}
h:M\times S^1\times [0,1]\rightarrow\R, 
\end{equation*}
with compact support in $M\times S^1\times [\delta,1-\delta]$, that vanish up to second order on $\cP(H_0)\times [0,1]$, and such that $\|h\|_\varepsilon<\infty$. Here $\|\cdot\|_\varepsilon$ is Floer's $\varepsilon$-norm, see Definition~\ref{def:epsilon_norm} and \cite[p.230]{audin-damian}. We identify the map $h$ with the path of time-dependent Hamiltonians $\{h_\lambda(\cdot, \cdot):=h(\cdot, \cdot,\lambda)\}_\lambda$.

The next claim is an adjustment of \cite[Theorem 11.3.2]{audin-damian} to our setting and is proved similarly. For the sake of completeness we include the proof, but we postpone it until the end of this section.
\begin{claim}\label{clm:reg_cobordism}
	Let $\{H_\lambda\}_{\lambda\in [0,1]}$ be a path of Hamiltonians as above, and assume that $(H_0, J)$ and $(H_1,J)$ are regular on $U$. Then, there exist a neighborhood of 0 in $\cC_\varepsilon^\infty (\{H_\lambda\}_\lambda)$ and a residual set $\hreg$ in this neighborhood, such that if $h\in \hreg$, then for $\Lambda = (\{H_\lambda +h_\lambda\}_\lambda, J)$ and every $x_\pm\in\cP(H_0)$ with $x_+\subset U$, the space $\cM_\Lambda(x_-,x_+)$ is a manifold with boundary, of dimension $\mu(x_-)-\mu(x_+)+1$, and its boundary is  
	\begin{equation}\label{eq:boundary_of_cob}
	\partial \cM_\Lambda(x_-, x_+) = \{0\}\hspace{-0.08cm}\times \hspace{-0.08cm} \cM_{(H,J)}(x_-,x_+)\ \cup \ \{1\}\hspace{-0.08cm}\times\hspace{-0.08cm}\cM_{(H',J)}(x_-,x_+).
	\end{equation}
\end{claim}

\begin{proof}[Proof of Proposition~\ref{pro:app_pert_Ham_regular_on_U}]
	Recall that $H$ is a non-degenerate Hamiltonian such that $(H,J)$ is regular on $U$. Let $H'$ be a small perturbation of $H$ that agrees with $H$ on $\cP(H)$ up to second order, and such that the pair $(H',J)$ is Floer-regular. We wish to show that the compositions of the differentials with respect to $(H,J)$ and $(H',J)$ with the projections onto $C_U$ and $C_{U_\circ}$ agree.
	Let $H_\lambda$ be a linear path (or, a linear homotopy) between $H$ and $H'$ that is stationary near $\lambda =0,1$, and such that for each $\lambda$, $H_\lambda$ agrees with $H$ on $\cP(H)$ up to second order (in particular, $\cP(H_\lambda) = \cP(H)$). Taking $H'$ to be close enough to $H$, and using Lemma~\ref{lem:semi_regular_is_open}, one can guarantee that all of the pairs $(H_\lambda,J)$ are semi-regular on $U$.
	
	By Claim~\ref{clm:reg_cobordism}, there exists a small perturbation of the path $\{H_\lambda\}$, such that for $\Lambda = (\{H_\lambda +h_\lambda\}_\lambda, J)$ and for every $x_\pm\in\cP(H_0)$ with $x_+\subset U$, the space $\cM_\Lambda(x_-,x_+)$ is a manifold with boundary, of dimension $\mu(x_-)-\mu(x_+)+1$. Let us show that when $\mu(x_-)-\mu(x_+)=1$, the quotient of this manifold by the $\R$ action,  $\overline{\cM}_\Lambda(x_-, x_+)=\cM_\Lambda(x_-, x_+)/\R$, is compact. Let $(\lambda_n, u_n)\in \cM_\Lambda(x_-, x_+)$ be any sequence. Since $\lambda_n\in[0,1]$, we may assume, by passing to a subsequence, that the sequence $\lambda_n$ converges to a number $\lambda_\star\in [0,1]$. By the definition of the space $\cM_\Lambda(x_-, x_+)$, $u_n\in\cM_{(H_{\lambda_n}, J)}(x_-, x_+)$ are solutions to the Floer equation with respect to Hamiltonians converging to $H_{\lambda_\star}$. By Proposition~\ref{pro:appendix_broken_traj}, there exists a subsequence of $u_n$ converging to a broken Floer trajectory $\overline{v}=\{v_1,\cdots,v_k\}$ of $(H_{\lambda_\star},J)$. Since the pair $(H_{\lambda_\star}, J)$ is semi-regular on $U$ and $x_+\subset U$, every solution in $\overline{v}$ that is non-constant (in the $s$-coordinate) decreases the index:
	\begin{equation*}
	\mu(x_-) = \mu(x_0)>\mu(x_1)>\cdots>\mu(x_k)=\mu(x_+).
	\end{equation*} 
	Recalling that $\mu(x_-)-\mu(x_+)=1$, we conclude that $\overline{v}$ contains exactly one non constant solution, $\overline{v}=v_1\in \cM_{(H_{\lambda_\star}, J)}(x_-,x_+)$. In other words, given the sequence $(\lambda_n, u_n)\in \cM_\Lambda(x_-, x_+)$, there exists a sequence of shifts $s_n\in \R$ such that, after passing to a subsequence, 
	\begin{equation*}
	(\lambda_n, \tau_{s_n} u_n)\xrightarrow[n\rightarrow\infty]{\cC^\infty_{loc}} (\lambda_\star, v_1).
	\end{equation*}
	In particular, after dividing by the  (free, proper and smooth) $\R$-action, the subsequence $(\lambda_n,[u_n])\in \overline{\cM}_\Lambda(x_-, x_+)$ converges to an element of the same space,
	\begin{equation*}
	(\lambda_n,[u_n])\xrightarrow[n\rightarrow\infty]{}(\lambda_\star, [v_1])\in \overline{\cM}_\Lambda(x_-, x_+),
	\end{equation*}
	and therefore $\overline{\cM}_\Lambda(x_-, x_+)$ is compact.
	Overall, $\overline{\cM}_\Lambda(x_-, x_+)$ is a smooth, compact manifold of dimension $\mu(x_-) - \mu(x_+)+1-1=1$, and its boundary is the zero-dimensional compact manifold $\partial \overline{\cM}_\Lambda(x_-, x_+) = \{0\}\times\overline{\cM}_{(H,J)}(x_-,x_+)\ \cup \ \{1\}\times\overline{\cM}_{(H',J)}(x_-,x_+)$. 
	Hence, the latter are finite sets with equal number of elements mod 2:
	\begin{equation*}
	\#_2\overline{\cM}_{(H,J)}(x_-,x_+) = \#_2\overline{\cM}_{(H',J)}(x_-,x_+).
	\end{equation*}  
	The equalities (\ref{eq:comp_of_diff_agree}) now follow immediately form the definition of the differential map.
\end{proof}

Let us sketch the proof of Claim~\ref{clm:reg_cobordism}, which follows the arguments made in \cite[Chapter 11.3.b]{audin-damian}.
\begin{proof}[Proof of Claim~\ref{clm:reg_cobordism}]
	Fix $x_\pm\in\cP(H_0)$ such that $x_+\subset U$. We first show that the space $\cM_\Lambda(x_-,x_+)$ has a structure of a manifold with boundary near the boundary (\ref{eq:boundary_of_cob}), and afterwards we prove that for perturbed paths the interior is a smooth manifold. 
	
	Let $\delta>0$ such that the path $\{H_\lambda\}$ is stationary for $\lambda\notin[\delta, 1-\delta]$ and every $h\in \cC_\varepsilon^\infty(\{H_\lambda\}_\lambda)$ satisfies 
	 $supp(h)\subset M\times S^1\times [\delta,1-\delta]$. In this case, 
	 \begin{equation*}
	 H_\lambda + h_\lambda =\begin{cases}
	 H_0, & \lambda\leq \delta,\\
	 H_1, & \lambda \geq 1-\delta,
	 \end{cases}
	 \end{equation*}
	for all $h \in \cC_\varepsilon^\infty(\{H_\lambda\}_\lambda)$. Fixing such $h$ and setting $\Lambda = (\{H_\lambda+h_\lambda\}_\lambda, J)$, 
	\begin{eqnarray}
	\cM_\Lambda (x_-,x_+)\cap\{\lambda<\delta\} &=& [0,\delta)\times \cM_{(H_0,J)}(x_-,x_+),\nonumber\\
	\cM_\Lambda (x_-,x_+)\cap\{\lambda>1-\delta\} &=& (1-\delta,1]\times \cM_{(H_1,J)}(x_-,x_+),\nonumber
	\end{eqnarray}
	are smooth manifolds with boundary, since the pairs $(H_0,J)$, $(H_1,J)$ are regular on $U$ and $x_+\subset U$.
	We conclude that near $\{0\}\times\cM_{(H,J)}(x_-,x_+)\cup \{1\}\times\cM_{(H',J)}(x_-,x_+)$ the space $\cM_\Lambda(x_-,x_+)$ has a structure of a manifold with boundary. 
	
	Let us now show that the interior of $\cM_\Lambda(x_-,x_+)$ is a smooth manifold. {Since the spaces $\cM_{(H,J)}(x_-,x_+)$ and $\cM_{(H',J)}(x_-,x_+)$ composing the boundary are one dimensional, it will follow that $\dim \cM_\Lambda(x_-,x_+)=2.$}
	The following statement is taken from \cite{audin-damian}, and states that the linearization $(d\cF)_u$ of the Floer map $\cF$ 
	is a Fredholm operator. 
	\begin{lemma}[{\cite[Theorem 8.1.5]{audin-damian}}] \label{lem:dF_is_Fredholm}
		For every non-degenerate Hamiltonian $H$, every almost complex structure $J$, compatible with $\omega$, and every $u\in \cM_{(H,J)}(x_-,x_+)$, $(d\cF)_u$ is a Fredholm operator of index $\mu(x_-)-\mu(x_+)$. 
	\end{lemma}
	As in Notations~\ref{not:Transversality_stuff}, we denote by $\cP(x_-,x_+)$ the space of maps $(s,t)\mapsto \exp_{w(s,t)}Y(s,t)$, where $Y\in W^{1,p}(w^*TM)$, for $p>2$, and $w\in \cC^\infty(\R\times S^1;M)$ converges to $x_\pm$ with exponential decay. Consider the vector bundle $\cE\rightarrow \cP(x_-,x_+)\times \cC_\varepsilon^\infty(\{H_\lambda\}_\lambda)$, given by
	\begin{equation*}
	\cE = \left\{(u,h,Y)|(u,h)\in\cP(x_-,x_+)\times \cC_\varepsilon^\infty(\{H_\lambda\}_\lambda),\ Y\in L^p(u^*TM)\right\}.
	\end{equation*}
	We define a family of sections $\{\sigma_\lambda\}_{\lambda\in (0,1)}$ by 
	\begin{equation*}
	\sigma_\lambda(u,h) = \left(u,h,\frac{\partial u}{\partial s} +J(u)\frac{\partial u}{\partial t} +\grad_u(H_\lambda+h_\lambda)\right).
	\end{equation*}
	For fixed $\lambda_0\in (0,1)$, the map $\sigma_{\lambda_0}$ is transversal to the zero section of the vector bundle $\cE$ if and only if, when $\sigma_{\lambda_0} (u,h)=(u,h,0)$, the linearized map $(d\sigma_{\lambda_0})_{(u,h)}$ composed with the projection onto the fiber, namely
	\begin{eqnarray}
	\nonumber \tilde\Gamma_{\lambda_0} : W^{1,p}(\R\times S^1;\R^{2n})\times \cC_\varepsilon^\infty(\{H_\lambda\}_\lambda) &\rightarrow& L^p(\R\times S^1;\R^{2n}),\\
	\nonumber
	(Y,\eta) &\mapsto& (d\cF^{H_{\lambda_0}+h_{\lambda_0}})_u(Y)+\grad_u \eta_{\lambda_0}, 
	\end{eqnarray} 
	is surjective. Here, and in what follows, we identify the linear space $\cC_\varepsilon^\infty( \{H_\lambda\}_\lambda)$ with its tangent space. If $\lambda_0 \notin (\delta,1-\delta)$, then $h_{\lambda_0}=0$ and $H_{\lambda_0}$ is equal to either $H_0$ or $H_1$, which are both regular on $U$, when paired with $J$. In this case, the surjectivity of $\tilde\Gamma_{\lambda_0}$ follows from that of $d\cF^{H_{\lambda_0}+h_{\lambda_0}}$ (which is guaranteed due to the regularity of $H_0$, $H_1$). 
	Let us prove the surjectivity of $\tilde\Gamma_{\lambda_0}$ for $\lambda_0\in(\delta,1-\delta)$. To do this, we embed $\cC_\varepsilon^\infty(H_0)= \cC_\varepsilon^\infty(H_{\lambda_0})$ into $\cC_\varepsilon^\infty(\{H_\lambda\}_\lambda)$ by mapping  $h_{\lambda_0}\in \cC_\varepsilon^\infty(H_{\lambda_0})$ to a locally constant path, $h(\cdot, \cdot, \lambda) =h_{\lambda_0} (\cdot, \cdot)$ near $\lambda = \lambda_0$. Here we have used our assumption that  $\{H_\lambda\}_\lambda$ all have the same periodic orbits as $H_0$. 
	It is now clear that the surjectivity of the restricted map, 
	\begin{eqnarray}
	\nonumber 
	\Gamma : W^{1,p}(\R\times S^1;\R^{2n})\times \cC_\varepsilon^\infty(H_{\lambda_0}) &\rightarrow& L^p(\R\times S^1;\R^{2n}),\\
	\nonumber
	(Y,\zeta) &\mapsto& (d\cF^{H_{\lambda_0}+h_{\lambda_0}})_u(Y)+\grad_u \zeta_{\lambda_0}, 
	\end{eqnarray} 
	which is guaranteed by \cite[Proposition 8.1.4]{audin-damian}, implies the surjectivity of $\tilde\Gamma_{\lambda_0} $. We conclude that for every $\lambda_0\in(0,1)$, the section $\sigma_{\lambda_0}$ intersects the zero section transversely.
	As a consequence, the section
	\begin{eqnarray}
	\nonumber \sigma:(\delta,1-\delta)\times \cP(x_-,x_+)\times \cC_\varepsilon^\infty(\{H_\lambda\}_\lambda)&\rightarrow&(\delta,1-\delta)\times\cE\\
	\nonumber
	(\lambda, u, h) &\mapsto& \left(\lambda,\sigma_\lambda(u,h)\right)
	\end{eqnarray}
	also intersects the zero section transversely and we conclude that the intersection 
	\begin{equation*}
	\cZ(x_-,x_+) = \left\{(\lambda,u,\{H_\lambda+h_\lambda\}_\lambda)\ \big|\ \lambda\in (\delta,1-\delta), u\in \cM_{(H_\lambda+ h_\lambda, J)}(x_-,x_+)
	\right\}
	\end{equation*} 
	is a Banach manifold (see \cite[Propositions 8.1.3, 11.3.4]{audin-damian} for the analogous statements).
	The tangent space of $\cZ(x_-,x_+)$ at a point $(\lambda,u,\{H_\lambda+h_\lambda\})$ consists of all $(a, Y, \eta)\in \R\times W^{1,p}(\R\times S^1;\R^{2n})\times \cC_\varepsilon^\infty(\{H_\lambda\})$ that satisfy the equation 
	\begin{equation}\label{eq:tangent_to_Z}
	a\cdot\grad_u \frac{\partial (H_\lambda+h_\lambda)}{\partial \lambda} + (d\cF^{H_\lambda+h_\lambda})_u(Y) +\grad_u(\eta_\lambda)=0.
	\end{equation}
	Let $\pi:\cZ(x_-,x_+)\rightarrow \cC_\varepsilon^\infty(\{H_\lambda\}_\lambda)$ be the projection. In order to conclude the proof of the claim, it is sufficient to show that the set of regular values of $\pi$ is a residual subset of $\cC_\varepsilon^\infty(\{H_\lambda\}_\lambda)$. This will follow from the Sard-Smale theorem (see \cite[Theorem 8.5.7]{audin-damian}), as soon as we show that $\pi$ is a Fredholm map (the separability of the spaces follows from Claim~\ref{clm:C_epsilon_is_separable} and Remark~\ref{rem:C_epsilon_separable} above). Let us therefore show that for every $(\lambda, u, \{H_\lambda+h_\lambda\}_\lambda)\in \cZ(x_-,x_+)$, the operator
	\begin{eqnarray}
	\nonumber 
	(d\pi)_{(\lambda, u, \{H_\lambda+h_\lambda\}_\lambda)}:T_{(\lambda, u, \{H_\lambda+h_\lambda\}_\lambda)}\cZ(x_-,x_+) &\rightarrow&  \cC_\varepsilon^\infty (\{H_\lambda\}_\lambda),\\
	\nonumber
	(a, Y,\eta) &\mapsto& \eta,
	\end{eqnarray}   
	is a Fredholm operator. In analogy with the proof of \cite[Proposition 11.3.5]{audin-damian}, denote by  
	$
	V := \grad_u \frac{\partial (H_\lambda+h_\lambda)}{\partial \lambda}\in L^p(u^*TM)
	$ the vector field multiplying $a$ in (\ref{eq:tangent_to_Z}). 
	Then, the kernel of $(d\pi)_{(\lambda, u, \{H_\lambda+h_\lambda\}_\lambda)}$ is the space
	\begin{equation*}
	\left\{(a,Y,0)\ |\ (a,Y)\in \R\times W^{1,p}(u^*TM)\text{ and } aV+(d\cF^{H_\lambda+h_\lambda})_u(Y) =0 \right\}.
	\end{equation*}
	Let us show that this space is finite dimensional by splitting into two cases:
	\begin{enumerate}
		\item $V\notin \im((d\cF^{H_\lambda+h_\lambda})_u)$. In this case we find \begin{equation*}
		\ker (d\pi)_{(\lambda, u, \{H_\lambda+h_\lambda\}_\lambda)} = \left\{ (0,Y,0)\ |\ Y\in \ker((d\cF^{H_\lambda+h_\lambda})_u)\right\},
		\end{equation*}
		which is finite dimensional by Lemma~\ref{lem:dF_is_Fredholm}.
		\item $V\in \im((d\cF^{H_\lambda+h_\lambda})_u)$. Choose $Y_0\in W^{1,p}(u^*TM)$ such that $(d\cF^{H_\lambda+h_\lambda})_u(Y_0)=V$.  It follows that 
		\begin{equation*}
		\ker (d\pi)_{(\lambda, u, \{H_\lambda+h_\lambda\}_\lambda)} = \left\{ (a,Y,0)\ |\ a(d\cF^{H_\lambda+h_\lambda})_u(Y_0)+(d\cF^{H_\lambda+h_\lambda})_u(Y)=0\right\}.
		\end{equation*}
		This space is isomorphic to $\R Y_0+\ker ((d\cF^{H_\lambda+h_\lambda})_u)$, which is also finite-dimensional.
	\end{enumerate}
	Next, let us show that the image of $ (d\pi)_{(\lambda, u, \{H_\lambda+h_\lambda\}_\lambda)} $ is closed and has finite codimension. 
	Indeed, it is the inverse image of the subspace 
	\begin{equation}\label{eq:image_dpi_vs_dF}
	\R V+ \im ((d\cF^{H_\lambda+h_\lambda})_u)\subset L^p(u^* TM)
	\end{equation}
	under the linear map $\eta\mapsto \grad_u\eta$, viewed as a map $\cC_\varepsilon^\infty( \{H_\lambda\}_\lambda) \rightarrow L^p(u^*TM)$. By Lemma~\ref{lem:dF_is_Fredholm}, the subspace (\ref{eq:image_dpi_vs_dF}) is closed and of finite-codimension, and hence we conclude the same for  the image of $ (d\pi)_{(\lambda, u, \{H_\lambda+h_\lambda\}_\lambda)} $.
	Consequently, $\pi$ is indeed a Fredholm map, and by the Sard-Smale theorem, the set of its regular values is a residual subset  $\cC_\varepsilon^\infty( \{H_\lambda\}_\lambda)$.
	
	Denote by $\hreg\subset \cC_\varepsilon^\infty( \{H_\lambda\}_\lambda)$ the set of regular values of $\pi$, then for any $h\in\hreg$, setting $\Lambda=(\{h_\lambda+h_\lambda\}_\lambda, J)$, the set 
	$$
	\pi^{-1}(h) = \cM_\Lambda(x_-,x_+)\cap\{\lambda\in(0,1)\}
	$$ 
	is a smooth manifold (with respect to the $\cC^\infty_{loc}$ topology). Together with the discussion from the beginning of the proof, this implies that $\cM_\Lambda(x_-,x_+)$ is a manifold with boundary.
\end{proof}

\subsubsection{Perturbing homotopies that are regular on a subset.}
In this section we state and prove results, analogous to the ones from  Section~\ref{app:Hamiltonian_reg_in_U}, but for homotopies instead of Hamiltonians. 
Fix an almost complex structure $J$ on $M$, a CIB domain $U$, and $U_\circ\Subset U$. 

\begin{defin}\label{def:hom_reg_on_U}
	Let $H$ be a homotopy of Hamiltonians such that the pair $(H,J)$ has a barricade in $U$ around $U_\circ$.
	\begin{enumerate}
		\item We say that the pair $(H,J)$ is {\it regular on $U$} if $(H_\pm,J)$ are regular on $U$ (see Definition~\ref{def:Ham_reg_on_U}) and for every solution $u$ of the $s$-dependent Floer equation with respect to $(H,J)$, the linearization $(d\cF)_u$ of the Floer map $\cF$ is surjective. 
		
		In particular, by \cite[Theorem~8.1.2]{audin-damian}, for every $x_\pm \in \cP(H_\pm)$ such that $x_+\subset U$, the space of solutions $\cM_{(H,J)}(x_-, x_+)$ is a smooth manifold of dimension $\mu(x_-) - \mu(x_+)$. 
		\item We say that the pair $(H,J)$ is {\it semi-regular on $U$} if $(H_\pm,J)$ are semi-regular on $U$ (as in Definition~\ref{def:Ham_reg_on_U}) and for every $x_\pm \in \cP(H_\pm)$, with $\mu(x_-)<\mu(x_+)$ and such that $x_+\subset U$, we have  $\cM_{(H,J)}(x_-, x_+)=\emptyset$.
	\end{enumerate}		
\end{defin} 
As in Section~\ref{app:Hamiltonian_reg_in_U}, if a pair is regular on $U$, then it is also semi-regular on $U$, and every Floer-regular pair with a barricade is regular on $U$. 
\begin{rem}
	For a pair $(H,J)$ that is regular on $U$, the continuation map might not be defined everywhere. However, using Proposition~\ref{pro:appendix_broken_traj}, one can see that when $\mu(x_-)=\mu(x_+)$ and $x_+\subset U$, the zero dimensional manifold $\cM_{(H,J)}(x_-, x_+)$ is compact and hence finite. The composition $\pi_U\circ \Phi_{(H,J)}$ can be defined by counting the elements of such manifolds. We remark that this is a slight abuse of notations, as the continuation map $\Phi_{(H,J)}$ is not necessarily defined on its own. Due to the barricade, if $x_-\subset U_\circ$ then $x_+\subset U_\circ\subset U$. It follows that the composition $\Phi_{(H,J)}\circ\pi_{U_\circ}$ is well defined as well. 
\end{rem}
Our main goal for this section is to prove the following statement.
\begin{prop}\label{pro:app_pert_hom_regular_on_U}
	Suppose that $H$ is a homotopy such that $(H,J)$ is regular on $U$, and let $R>0$ such that $\partial_sH|_{|s|>R} =0$. Let $H'$ be a homotopy such that
	\begin{enumerate}
		\item $\partial_s H'|_{|s|>R}=0$, \label{itm:hom_reg_U_bdd_supp}
		\item $H'$ is $\cC^\infty$-close to $H$, and $H_\pm'$ agree with $H_\pm$ on their 1-periodic orbits up to second order,\label{itm:hom_reg_U_pert}
		\item $(H',J)$ is regular on $U$. \label{itm:hom_reg_U_reg_U}
	\end{enumerate}
	Then, the compositions of the continuation maps and projections agree:
	\begin{equation}\label{eq:comp_of_cont_agree}
	\pi_U\circ \Phi_{(H,J)}=\pi_U\circ \Phi_{(H',J)},
	\qquad\Phi_{(H,J)}\circ\pi_{U_\circ} = \Phi_{(H',J)}\circ\pi_{U_\circ}.
	\end{equation}  
\end{prop}
As before, the second equation in (\ref{eq:comp_of_cont_agree}) follows from the first, since both $(H,J)$ and $(H',J)$ have a barricade in $U$ around $U_\circ$ and thus $\Phi\circ\pi_{U_\circ} = \pi_U\circ\Phi\circ\pi_{U_\circ}$.
In analogy with the previous section, in order to prove Proposition~\ref{pro:app_pert_hom_regular_on_U}, we connect $H$ and $H'$ by a linear path (or, linear homotopy) of homotopies $\{H_\lambda\}_{\lambda\in[0,1]}$, such that the pairs $(H_\lambda, J)$ are all semi-regular on $U$. Then, given $x_\pm\in \cP(H_\pm)$, with $\mu(x_-)=\mu(x_+)$ and $x_+\subset U$, we show that the space 
\begin{equation}\label{eq:hom_cobordism_space}
\cM_\Lambda(x_-, x_+):=\left\{(\lambda, u): u \in \cM_{(H_\lambda,J)}(x_-, x_+)\right\}
\end{equation}
is a smooth, compact, 1-dimensional manifold with boundary, that realizes a cobordism between $\cM_{(H,J)}(x_-,x_+)$ and $\cM_{(H',J)}(x_-,x_+)$. We will then conclude that the number of elements in $\cM_{(H,J)}(x_-,x_+)$ and  $\cM_{(H',J)}(x_-,x_+)$ coincides modulo 2. 

As for the case of Hamiltonians, semi-regularity of homotopies is also an open condition, as the following lemma guarantees.
\begin{lemma}\label{lem:hom_semi_regular_is_open}
	Suppose that $(H,J)$ is semi-regular on $U$, and fix $R>0$. Then, for every homotopy $H'$ that is close enough to $H$, such that $\partial_s H'|_{|s|>R}=0$ and $H_\pm'$ agree with $H_\pm$ on their 1-periodic orbits up to second order, the pair $(H',J)$ is also semi-regular on $U$.
\end{lemma} 
\begin{proof}
	First, notice that by Proposition~\ref{pro:appendix}, for every homotopy $H'$ that satisfies the conditions of the lemma, the pair $(H',J)$ has a barricade in $U$ around $U_\circ$.
	Assume for the sake of contradiction that there exist a sequence of homotopies $H_n$, converging to $H$, such that for each $n$, $H_n$ satisfies the conditions of the lemma, and $(H_n,J)$ is not semi-regular on $U$. Then, for each $n$, there exist $x_\pm^n$ satisfying $\mu(x_-^n)<\mu(x_+^n)$ and $x_+^n\subset U$, and a solution $u_n \in\cM_{(H_n,J)}(x_-^n, x_+^n)$. Since $x_\pm^n\in \cP(H_{n\pm}) = \cP(H_\pm)$ are elements of finite sets, we may assume, by passing to a subsequence, that $x_\pm^n = x_\pm$ are independent of $n$. By Proposition~\ref{pro:appendix_broken_traj}, there exists a subsequence of the solutions $u_n$ that converges to a broken trajectory
	$$
	\overline{v} =(v_1,\dots,v_k,w,v'_1,\dots,v'_\ell)
	$$ 
	of $(H,J)$.  Here $v_i$ and $v_j'$ are Floer solutions with respect to the Hamiltonians $H_-$ and $H_+$ respectively, and $w\in\cM_{(H,J)}$ is a solution with respect to the homotopy $H$.
	Moreover, the ends of the broken trajectory are $x_\pm$. Since $x_+$ is contained in $U$ and the pairs $(H,J)$, $(H_\pm,J)$ all have barricades in $U$ around $U_\circ$, it follows from Lemma~\ref{lem:barricade_ban_broken_traj} that the broken trajectory $\overline{v}$	is contained in $U$. As the pair $(H,J)$ is semi-regular on $U$, for every non-constant $v_i$ or $v_j'$, the index difference between the left end and the right end is positive. Moreover, the index difference between the ends of $w$ is non-negative. Therefore, under the notations of Proposition~\ref{pro:appendix_broken_traj}, we have
	$$
	\mu(x_-) = \mu(x_0)> \cdots>\mu(x_k)\geq \mu(y_0) > \cdots>\mu(y_\ell)=\mu(x_+),
	$$
	which contradicts our assumption that  $\mu(x_-^n)<\mu(x_+^n)$.
\end{proof}

As in the previous section, we show that for a suitable choice of a path of homotopies, $\{H_\lambda\}$, the set (\ref{eq:hom_cobordism_space}) is a smooth manifold. Our starting point is a path $\{H_\lambda\}_{\lambda\in [0,1]}$ that is stationary for $\lambda \notin[\delta,1-\delta]$, such that for all $\lambda\in [0,1]$, $H_\lambda$ satisfies properties \ref{itm:hom_reg_U_bdd_supp}-\ref{itm:hom_reg_U_pert} from Proposition~\ref{pro:app_pert_hom_regular_on_U}.
This time the space of perturbations, $\cC_\varepsilon^\infty(\{H_\lambda\}_\lambda)$, will consist of maps
\begin{equation*}
h:M\times S^1\times \R \times [0,1]\rightarrow\R, 
\end{equation*}
supported in $M\times S^1\times[-R,R]\times[\delta,1-\delta]$, such that
$\|h\|_\varepsilon<\infty$, where again $\|\cdot\|_\varepsilon$ Floer's norm from Definition~\ref{def:epsilon_norm}. We identify the map $h$ with the path of homotopies $\{h_\lambda(\cdot, \cdot):=h(\cdot, \cdot,\lambda)\}_\lambda$.

The following claim is an adjustment of \cite[Theorem 11.3.2]{audin-damian} to  the case where the ends of the path, $(H_0,J)$ and $(H_1,J)$, are not necessarily Floer-regular, but are regular on $U$, and the support of the perturbations is uniformly bounded. The proof is completely analogous to that of Claim~\ref{clm:reg_cobordism} above, with the single difference, that the surjectivity of the operator $\Gamma$ for homotopies is guaranteed by Lemma~\ref{lem:Gamma_is_surj}, instead of \cite[Proposition 8.1.4]{audin-damian}. We therefore omit the proof.
\begin{claim}\label{clm:reg_cobordism_hom}
	Let $\{H_\lambda\}_{\lambda\in [0,1]}$ be a path of homotopies as above, and assume that $(H_0, J)$ and $(H_1,J)$ are regular on $U$. Then, there exists a residual subset $\hreg\subset \cC_\varepsilon^\infty (\{H_\lambda\}_\lambda)$, such that if $h\in \hreg$, then for $\Lambda = (\{H_\lambda +h_\lambda\}_\lambda, J)$ and for every $x_\pm\in\cP(H_{0\pm})$ with $x_+\subset U$, the space $\cM_\Lambda(x_-,x_+)$ is a manifold with boundary, of dimension $\mu(x_-)-\mu(x_+)+1$, and its boundary is  
	\begin{equation}\label{eq:boundary_of_cob_hom}
	\partial \cM_\Lambda(x_-, x_+) = \{0\}\hspace{-0.08cm}\times\hspace{-0.08cm}\cM_{(H,J)}(x_-,x_+)\ \cup \ \{1\}\hspace{-0.08cm}\times\hspace{-0.08cm}\cM_{(H',J)}(x_-,x_+).
	\end{equation}
\end{claim}

\begin{proof}[Proof of Proposition~\ref{pro:app_pert_hom_regular_on_U}]
	Recall that $H$ is a homotopy such that $(H,J)$ is regular on $U$, and $H'$ is a homotopy satisfying properties \ref{itm:hom_reg_U_bdd_supp}-\ref{itm:hom_reg_U_reg_U} above.
	Let $H_\lambda$ be a linear path (or, linear homotopy) between the homotopies $H$ and $H'$, that is stationary for $\lambda \notin[\delta, 1-\delta]$. Then, for each $\lambda$, the homotopy $H_\lambda$ is close to $H$ and its ends, $H_{\lambda\pm}$, agree with the ends of $H$ on $\cP(H_\pm)$. In particular, $\cP(H_{\lambda\pm}) = \cP(H_\pm)$ for all $\lambda\in[0,1]$. Taking $H'$ to be close enough to $H$, Lemma~\ref{lem:hom_semi_regular_is_open} guarantees that all of the homotopies $H_\lambda$ are semi-regular on $U$, when paired with $J$. In particular, for each $\lambda$, the pairs $(H_\lambda, J)$ and $(H_{\lambda\pm}, J)$ have a barricade in $U$ around $U_\circ$.
	
	By Claim~\ref{clm:reg_cobordism_hom}, there exists a small perturbation  $h\in \cC_\varepsilon^\infty(\{H_\lambda\}_\lambda)$, such that for $\Lambda = (\{H_\lambda +h_\lambda\}_\lambda, J)$ and for every $x_\pm\in\cP(H_{0\pm})$ with $x_+\subset U$, the space $\cM_\Lambda(x_-,x_+)$ is a manifold with boundary, of dimension $\mu(x_-)-\mu(x_+)+1$. Let us show that when $\mu(x_-)-\mu(x_+)=0$, the manifold $\cM_\Lambda(x_-, x_+)$ is compact. Let $(\lambda_n, u_n)\in \cM_\Lambda(x_-, x_+)$ be any sequence. After passing to a subsequence, we have $\lambda_n\rightarrow\lambda_\star\in [0,1]$, and hence $u_n\in\cM_{(H_{\lambda_n}, J)}(x_-, x_+)$ are solutions with respect to homotopies that converge to $H_{\lambda_\star}$. By Proposition~\ref{pro:appendix_broken_traj}, there exists a subsequence of $u_n$ converging to a broken trajectory $\overline{v}=\{v_1,\cdots,v_k,w,v_1',\dots,v_\ell'\}$ of $(H_{\lambda_\star},J)$. Since the pairs $(H_{\lambda_\star}, J)$ and $(H_{\lambda_\star\pm}, J)$ have a barricade in $U$ around $U_\circ$ and $x_+\subset U$, Lemma~\ref{lem:barricade_ban_broken_traj} guarantees that the broken trajectory is completely contained in $U$. The fact that $(H_{\lambda_\star}, J)$ is semi-regular on $U$ now implies that $v_i$ and $v_j'$ are index-decreasing, and $w$ is index non-increasing:
	\begin{equation*}
	\mu(x_-) = \mu(x_0)> \cdots>\mu(x_k)\geq \mu(y_0) > \cdots>\mu(y_\ell)=\mu(x_+).
	\end{equation*} 
	Recalling that $\mu(x_-)-\mu(x_+)=0$, we conclude that $\overline{v}$ does not contains non-constant solutions of the $s$-independent Floer equations, and hence $\overline{v}=w\in \cM_{(H_{\lambda_\star}, J)}(x_-,x_+)$. This implies that the above subsequence converges to an element of the space,
	\begin{equation*}
	(\lambda_n,u_n)\xrightarrow[n\rightarrow\infty]{}(\lambda_\star, w)\in \cM_\Lambda(x_-, x_+),
	\end{equation*}
	and therefore $\cM_\Lambda(x_-, x_+)$ is compact.
	
	Overall, $\cM_\Lambda(x_-, x_+)$ is a smooth, compact manifold of dimension $\mu(x_-) - \mu(x_+)+1=1$, and its boundary is $\partial \cM_\Lambda(x_-, x_+) = \{0\}\times\cM_{(H,J)}(x_-,x_+)\ \cup \ \{1\}\times\cM_{(H',J)}(x_-,x_+)$. Consequently, the latter are finite sets with equal number of elements mod 2:
	\begin{equation*}
	\#_2\cM_{(H,J)}(x_-,x_+) = \#_2\cM_{(H',J)}(x_-,x_+).
	\end{equation*}  
	The equalities (\ref{eq:comp_of_cont_agree}) follow immediately form the definition of the continuation maps.
\end{proof}

\subsubsection*{Perturbing homotopies that are constant on a subset.}
A particular application of Proposition~\ref{pro:app_pert_hom_regular_on_U} that will be useful, is when $H$ is a homotopy that is constant on the set $U$, and whose ends, $H_\pm$, are regular on $U$ when paired with $J$. 
In this case, it follows from Definition~\ref{def:hom_reg_on_U} that the pair $(H,J)$ is also regular on $U$. Moreover, for periodic orbits $x_\pm\in \cP(H_\pm)$ such that $x_+\subset U$, the space $\cM_{(H,J)}(x_-,x_+)$ coincides with $\cM_{(H_-,J)}(x_-,x_+)$. As a consequence, when $\mu(x_-)=\mu(x_+)$, the space  $\cM_{(H,J)}(x_-,x_+)$ is empty if $x_-\neq x_+$ and contains only constant solutions otherwise. We conclude that the continuation map with respect to $(H,J)$ agrees with the identity map after composing with the projections:
\begin{equation*}
 \pi_{U}\circ\Phi_{(H,J)} = \pi_{U}\circ\id,\quad
 \Phi_{(H,J)}\circ \pi_{U_\circ} = \id \circ  \pi_{U_\circ}.
\end{equation*}
Applying Proposition~\ref{pro:app_pert_hom_regular_on_U} we conclude that the same holds for perturbations of $H$.
\begin{cor}\label{cor:almost_constant_homotopy}
	Suppose that $H$ is a homotopy between two non-degenerate Hamiltonians $H_\pm$, such that $(H,J)$ is constant on $U$, namely $\partial_sH|_U=0$, and $(H_\pm,J)$ are regular on $U$. Fix $R>0$ and let $H'$ be a $\cC^\infty$-small perturbation of $H$, satisfying
	\begin{enumerate}
		\item $\partial_s H'|_{|s|>R}=0$,
		\item $H_\pm'$ agree with $H_\pm$ on their 1-periodic orbits up to second order,
		\item $(H',J)$ is regular on $U$.
	\end{enumerate}
	Then,
	\begin{equation}
	\Phi_{(H',J)}\circ\pi_{U_\circ}=\id\circ\pi_{U_\circ} , \qquad\pi_U\circ \Phi_{(H',J)}=\pi_U\circ \id.
	\end{equation}  
\end{cor}

\appendix
\section{Incompressibility of domains with incompressible boundaries.}\label{app:incompressible}
Let $M^n$ be a smooth $n$-dimensional orientable manifold, and let $N^n$ be a smooth $n$\nobreakdash-dimensional orientable manifold with boundary such that there exists an embedding $\iota\colon N \to M$.
Denote by $U:=\operatorname{Im}\left(\iota\left(N\setminus \partial N\right)\right)$, and note that $\partial U = \operatorname{Im}\left(\iota\left(\partial N\right)\right)$.
\begin{prop}\label{pro:incompressible}
	If $\partial U$ is incompressible in $M$, then $U$ is incompressible in $M$.
\end{prop}
\begin{proof}
	In order to show that $\iota_*\colon \pi_1(U) \to \pi_1(M)$ is injective it is sufficient to prove that if a loop $\gamma$ in $U$ is contractible in $M$ then it is contractible in $U$.
	Let $\gamma\colon S^1 \to U$ be a loop that is contractible in $M$. Then, there exists a map $u\colon D \to M$ such that $u\vert^{}_{\partial D}\equiv \gamma$, where $D\subset \mathbb{R}^2$ denotes the unit disk.
	
	Without loss of generality we may assume that $\gamma$ and $u$ are smooth, and that $u\pitchfork\partial U$.
	Indeed, by Whitney's smooth approximation theorem, $u$ is homotopic to a smooth map, $\tilde{u}$. Since $\operatorname{Im}\gamma$ is compact and $U$ is open,  we can choose the smooth approximation such that $\tilde{\gamma}:=\tilde{u}\vert^{}_{\partial D}$ is homotopic to $\gamma$ in $U$. Applying Thom's transversality theorem, we may assume that $\tilde{u} \pitchfork \partial U$. We replace the maps $\gamma$ and $u$ by $\tilde{\gamma}$ and $\tilde{u}$, in order to keep the notations.
	
	Under the assumptions above, the preimage $C=u^{-1}\left(\partial U\right)$ is a compact one dimensional submanifold of $D$, hence a disjoint union of embedded closed curves, $C=\bigsqcup_j C_j$.
	Some of the curves $C_j$ may encompass others. We call a curve $C_j$ \emph{a maximal curve} if it is not encompassed by any other component of $C$. More formally, for each component $C_j$, denote by $D_j\subset D$ the embedded topological disk such that $\partial D_j = C_j$. The curve $C_j$ is maximal if $C_j \nsubseteq D_k$ for all $k\neq j$.
	We denote the set of maximal curves by $\mathcal{C}_{max} := \left\lbrace C_{j_1}, \ldots, C_{j_\ell}\right\rbrace$, and by $\mathcal{D}_{max} := \left\lbrace D_{j_1}, \ldots, D_{j_\ell}\right\rbrace$ the set of the corresponding topological disks.
	
	For every $1 \le i \le \ell$, the restriction $u\vert^{}_{C_{j_i}}$ is a loop in $\partial U$ which is contractible in $M$ by $u\vert^{}_{D_{j_i}}$. By the incompressibility of $\partial U$, the loop $u\vert^{}_{C_{j_i}}$ is contractible in $\partial U$, namely there exists a map $v_i \colon D_{j_i} \to \partial U$ such that $v_i\vert^{}_{C_{j_i}} \equiv u\vert^{}_{C_{j_i}}$. Using the maps $u$ and $v_i$ we can define a map that contracts $\gamma$ inside $\bar {U}$:
	\[
	\hat{u}= 
	\begin{cases} 
	v_i(x) & x\in D_{j_i} \\
	u(x) & \text{otherwise.}
	\end{cases}
	\]
	
	Let us check that $\hat{u}$ is a contraction of $\gamma$ in $\bar{U}$. Indeed, recalling that $u(\partial D)=\gamma \subset U$ and that $C:=u^{-1}\left(\partial U\right)$, it follows from the maximality of the curves in $\mathcal{C}_{max}$  that for all $x\in D\setminus \bigsqcup_{\mathcal{D}_{max}}{D_{j_i}}$, one has $u(x)\in U$, and therefore $\hat{u}(x) \in U$.
	Moreover, for every $x\in D_{j_i}$, $\hat{u}(x) \in \partial {U}$, and we conclude that $\operatorname{Im}\left(\hat u\right) \subseteq U\cup \partial U$.
	
	Using the fact that $\partial U$ has a collar neighborhood in $\bar{U}$, one can construct a continuous map $w\colon \bar{U} \to U$. The composition $w\circ \hat{u}$ is the desired contraction of $\gamma$ in $U$.
	
\end{proof}
\bibliographystyle{plain}
\bibliography{refs}

\begin{thebibliography}{10}

\bibitem{abbondandolo2019simple}
Alberto Abbondandolo, Carsten Haug, and Felix Schlenk.
\newblock A simple construction of an action selector on aspherical symplectic
  manifolds.
\newblock {\em L’Enseignement Math{\'e}matique}, 65(1):221--267, 2020.

\bibitem{abouzaid2010open}
Mohammed Abouzaid and Paul Seidel.
\newblock An open string analogue of {V}iterbo functoriality.
\newblock {\em Geometry \& Topology}, 14(2):627--718, 2010.

\bibitem{audin-damian}
Mich{\`e}le Audin and Mihai Damian.
\newblock {\em Morse theory and {F}loer homology}.
\newblock Springer, 2014.

\bibitem{cieliebak2005quantitative}
K.~Cieliebak, Helmut Hofer, J.~Latschev, and F.~Schlenk.
\newblock {\em Quantitative symplectic geometry}, page 1–44.
\newblock Mathematical Sciences Research Institute Publications. Cambridge
  University Press, 2007.

\bibitem{cieliebak2018symplectic}
Kai Cieliebak and Alexandru Oancea.
\newblock Symplectic homology and the {E}ilenberg--{S}teenrod axioms.
\newblock {\em Algebraic \& Geometric Topology}, 18(4):1953--2130, 2018.

\bibitem{entov2009rigid}
Michael Entov and Leonid Polterovich.
\newblock Rigid subsets of symplectic manifolds.
\newblock {\em Compositio Mathematica}, 145(3):773--826, 2009.

\bibitem{floer1995transversality}
Andreas Floer, Helmut Hofer, Dietmar Salamon, et~al.
\newblock Transversality in elliptic {M}orse theory for the symplectic action.
\newblock {\em Duke Mathematical Journal}, 80(1):251--292, 1995.

\bibitem{frauenfelder2007hamiltonian}
Urs Frauenfelder and Felix Schlenk.
\newblock Hamiltonian dynamics on convex symplectic manifolds.
\newblock {\em Israel Journal of Mathematics}, 159(1):1--56, 2007.

\bibitem{hatcher2003vector}
Allen Hatcher.
\newblock {\em Vector bundles and {K}-theory}.
\newblock 2003.

\bibitem{hofer2012symplectic}
Helmut Hofer and Eduard Zehnder.
\newblock {\em Symplectic invariants and Hamiltonian dynamics}.
\newblock Birkh{\"a}user, 2012.

\bibitem{humiliere2016towards}
Vincent Humili{\`e}re, Fr{\'e}d{\'e}ric Le~Roux, and Sobhan Seyfaddini.
\newblock Towards a dynamical interpretation of {H}amiltonian spectral
  invariants on surfaces.
\newblock {\em Geometry \& Topology}, 20(4):2253--2334, 2016.

\bibitem{ishikawa2015spectral}
Suguru Ishikawa.
\newblock Spectral invariants of distance functions.
\newblock {\em Journal of Topology and Analysis}, page 1650025, 2015.

\bibitem{mcduff2012j}
Dusa McDuff and Dietmar Salamon.
\newblock {\em J-holomorphic curves and symplectic topology}, volume~52.
\newblock American Mathematical Soc., 2012.

\bibitem{polterovich2014symplectic}
Leonid Polterovich.
\newblock Symplectic geometry of quantum noise.
\newblock {\em Communications in Mathematical Physics}, 327(2):481--519, 2014.

\bibitem{polterovich2014function}
Leonid Polterovich and Daniel Rosen.
\newblock {\em Function theory on symplectic manifolds}.
\newblock American Mathematical Society, 2014.

\bibitem{ritter2013topological}
Alexander~F Ritter.
\newblock Topological quantum field theory structure on symplectic cohomology.
\newblock {\em Journal of Topology}, 6(2):391--489, 2013.

\bibitem{schwarz2000action}
Matthias Schwarz.
\newblock On the action spectrum for closed symplectically aspherical
  manifolds.
\newblock {\em Pacific Journal of Mathematics}, 193(2):419--461, 2000.

\bibitem{seyfaddini2014spectral}
Sobhan Seyfaddini.
\newblock Spectral killers and {P}oisson bracket invariants.
\newblock {\em J. Mod. Dyn.}, 9:51--66, 2015.

\bibitem{usher2011boundary}
Michael Usher.
\newblock Boundary depth in {F}loer theory and its applications to
  {H}amiltonian dynamics and coisotropic submanifolds.
\newblock {\em Israel Journal of Mathematics}, 184(1):1, 2011.

\bibitem{varolgunes2018mayer}
Umut Varolgunes.
\newblock Mayer-vietoris property for relative symplectic cohomology.
\newblock {\em arXiv preprint arXiv:1806.00684}, 2018.

\bibitem{viterbo1999functors}
Claude Viterbo.
\newblock Functors and computations in floer homology with applications, i.
\newblock {\em Geometric \& Functional Analysis GAFA}, 9(5):985--1033, 1999.

\bibitem{wendl1612lectures}
Chris Wendl.
\newblock Lectures on {S}ymplectic {F}ield {T}heory. {EMS} {S}eries of
  {L}ectures in {M}athematics (to appear).
\newblock {\em Preprint arXiv}, 1612.

\end{thebibliography}

\paragraph{Yaniv Ganor,}$ $\\
Department of Mathematics\\ 
Technion - Israel Institute of Technology \\
Haifa, 3200003\\
Israel\\
E-mail: ganory@gmail.com

\paragraph{Shira Tanny,}$ $\\
School of Mathematical Sciences\\ 
Tel Aviv University \\
Ramat Aviv, Tel Aviv 69978\\
Israel\\
E-mail: tanny.shira@gmail.com

\end{document}